\documentclass[11pt,reqno]{amsart}
\usepackage[letterpaper, margin=1in]{geometry}
\usepackage{amsmath,amscd,amssymb}
\usepackage{url}
\usepackage{graphicx}

\usepackage{color}
\usepackage{tikz-cd}
\usepackage{enumerate}
\usepackage[pagebackref,colorlinks,citecolor=blue,linkcolor=magenta]{hyperref}
\usepackage[linesnumbered,ruled]{algorithm2e}
\RequirePackage{amsthm,amsmath,amsfonts,amssymb}
\usepackage[utf8]{inputenc}
\usepackage{chngcntr}
\usepackage{caption}
\usepackage{subcaption}
\usepackage{placeins}
\usepackage{xcolor}
\usepackage{tikz}
\usetikzlibrary{decorations.pathreplacing,calligraphy}
\usetikzlibrary{arrows}
\usepackage{nicematrix}
\usepackage{bbm}
\usepackage[normalem]{ulem}
\usepackage{centernot}

\usepackage{algpseudocode}

\newtheorem{theorem}{Theorem}
\newtheorem*{theorem*}{Theorem}
\numberwithin{theorem}{section}
\newtheorem{proposition}[theorem]{Proposition}

\newtheorem{corollary}[theorem]{Corollary}
\newtheorem{lemma}[theorem]{Lemma}

\theoremstyle{definition}
\newtheorem{definition}[theorem]{Definition}

\newtheorem*{definition*}{Definition}
\newtheorem{remark}[theorem]{Remark}
\newtheorem{notation}[theorem]{Notation}
\newtheorem{observation}[theorem]{Observation}
\newtheorem{example}[theorem]{Example}
 \newtheorem{problem}[theorem]{Problem}

\newcommand{\rr}{\mathbb{R}}

\newcommand{\dd}{\mathbb{D}}

\newcommand{\ind}{\mathbbm{1}}

\newcommand{\ci}{\mathcal{I}}
\newcommand{\cj}{\mathcal{J}}
\newcommand{\cm}{\mathcal{M}}

\newcommand{\ct}{\mathcal{T}}
\newcommand{\cg}{\mathcal{G}}

\newcommand{\ch}{\mathcal{H}}
\newcommand{\cx}{\mathcal{X}}

\renewcommand{\vec}[1]{\boldsymbol{#1}}
\newcommand{\vX}{\vec{X}}
\newcommand{\vx}{\vec{x}}
\newcommand{\vY}{\vec{Y}}
\newcommand{\vy}{\vec{y}}

\newcommand{\from}{\leftarrow}

\newcommand{\eye}{\mathrm{I}}

\newcommand\independent{\protect\mathpalette{\protect\independenT}{\perp}}
\def\independenT#1#2{\mathrel{\rlap{$#1#2$}\mkern2mu{#1#2}}}

\tikzstyle{wB}=[circle, draw, fill=black, inner sep=0pt, minimum width=4.5pt]
\tikzstyle{wR}=[circle, draw, fill=red, inner sep=0pt, minimum width=4.5pt]
\tikzstyle{wW}=[circle, draw, fill=white, inner sep=0pt, minimum width=4.5pt]

\DeclareMathOperator{\CIM}{CIM}
\DeclareMathOperator{\cim}{CIM}
\DeclareMathOperator{\SIM}{SIM}

\DeclareMathOperator{\pa}{pa}
\DeclareMathOperator{\fa}{fam}

\DeclareMathOperator{\conv}{conv}
\DeclareMathOperator{\skel}{skel}

\DeclareMathOperator{\diag}{diag}

\DeclareMathOperator{\Phil}{Fill}
\DeclareMathOperator{\BIC}{BIC}
\DeclareMathOperator{\pool}{pool}

\DeclareMathOperator{\interior}{int}

\definecolor{benpurple}{RGB}{180, 0, 240}
\definecolor{joegreen}{RGB}{18, 135, 8}

\begin{document}

\title[Interventional Characteristic Imset Polytopes]{Hyperplane Representations of Interventional Characteristic Imset Polytopes}

\date{\today}


\author{Benjamin Hollering}
\address{Technische Universit\"at M\"unchen, 85748 Garching b. München, Boltzmannstr. 3.,  Germany}
\email{benjamin.hollering@tum.de}

\author{Joseph Johnson}
\address{Institutionen f\"or Matematik, KTH, SE-100 44 Stockholm, Sweden}
\email{josjohn@kth.se}

\author{Liam Solus}
\address{Institutionen f\"or Matematik, KTH, SE-100 44 Stockholm, Sweden}
\email{solus@kth.se}

\begin{abstract}
Characteristic imsets are $0/1$-vectors representing directed acyclic graphs whose edges represent direct cause-effect relations between jointly distributed random variables. 
A characteristic imset (CIM) polytope is the convex hull of a collection of characteristic imsets. 
CIM polytopes arise as feasible regions of a linear programming approach to the problem of causal disovery, which aims to infer a cause-effect structure from data. 
Linear optimization methods typically require a hyperplane representation of the feasible region, which has proven difficult to compute for CIM polytopes despite continued efforts. 
We solve this problem for CIM polytopes that are the convex hull of imsets associated to DAGs whose underlying graph of adjacencies is a tree. 
Our methods use the theory of toric fiber products as well as the novel notion of interventional CIM polytopes. 
Our solution is obtained as a corollary of a more general result for interventional CIM polytopes. 
The identified hyperplanes are applied to yield a linear optimization-based causal discovery algorithm for learning polytree causal networks from a combination of observational and interventional data.
\end{abstract}


\maketitle
\thispagestyle{empty}


\section{Introduction}

This paper studies H-representations of V-polytopes, known as characteristic imset (CIM) polytopes, relevant to the problem of causal discovery in statistics and machine learning.
A directed acyclic graphical (DAG) model is a collection of probability distributions $\cm(\cg)$ associated to a DAG $\cg = ([p], E)$ where $[p]:=\{1,\ldots,p\}$ is the node set of $\cg$ and $E$ is the set of edges. 
DAG models are popular tools for both statistical \cite{koller2009probabilistic} and causal inference \cite{pearl2009causality}, with applications arising in biological sciences, medicine, sociology, economics and other industries \cite{beinlich1989alarm, friedman2000using, robins2000marginal, sachs2005causal}. 
The problem of \emph{causal discovery} is to infer the DAG $\cg$ using data drawn from a distribution $\vX = [X_1,\ldots, X_p]^T$ contained in the model $\cm(\cg)$. 

It is, however, possible that $\cm (\cg)=\cm (\ch)$ for distinct DAGs $\cg$ and $\ch$. 
In this case, $\cg$ and $\ch$ are called \emph{Markov equivalent}, and the collection of DAGs Markov equivalent to $\cg$ are called its \emph{Markov equivalence class} (MEC). 
From observational data alone, the DAG $\cg$ is identifiable from a random sample drawn from $\vX\in\cm(\cg)$ only up to Markov equivalence. 
However, with the help of additional data obtained through experiments that perturb targeted sets of variables, we may refine the MEC of $\cg$, thereby improving structural identifiability. 
Each such \emph{intervention target} is a set $I_k \subseteq [p]$, which we collect into a sequence $\ci = (I_0, I_1,\ldots, I_K)$. The resulting equivalence classes of DAGs are called \emph{$\ci$-Markov equivalence classes} ($\ci$-MECs) \cite{hauser2012characterization, wang2017permutation, yang2018characterizing}. The problem of causal discovery is then well-defined as:

\begin{problem}[Causal Discovery]
Given a random sample from a distribution $\vX = [X_1,\ldots, X_p]^T\in\cm(\cg)$, infer the MEC of $\cg$. 
Given experimental data for the intervention targets $\ci = (I_0, I_1,\ldots, I_K)$, identify the $\ci$-MEC of $\cg$. 
\end{problem}

Algorithms for estimating a MEC (or $\ci$-MEC), are called \emph{causal discovery algorithms}, and they are typically broken down by their underlying methodology and assumptions on the data. 
Common methodologies include \emph{constraint-based methods}, such as the classic \emph{PC algorithm} \cite{spirtes1991algorithm}, which treat causal discovery as a constraint satisfaction problem implemented via statistical hypothesis tests. 
Alternatively, \emph{greedy optimization methods}, such as the \emph{Greedy Equivalence Search (GES)} \cite{chickering2002optimal}, assign a score to each MEC and then use graph transformations to search over the space of all MECs until a score-optimal MEC is found. 
Finally, there are the \emph{hybrid methods}, which use a mixture of constraint-based hypothesis testing and greedy optimization. 
Examples include the \emph{Max-Min Hill Climbing (MMHC)} algorithm \cite{tsamardinos2006max}, which uses constraint-based tests to estimate a set of possible edges that could appear in the DAGs in the true MEC and then applies greedy-hill climbing methods to search for an optimal MEC given these constraints. 
More recent examples include \emph{GreedySP} \cite{solus2017consistency}, \emph{GrASP} \cite{lam2022greedy} and \emph{BOSS} \cite{andrews2023fast}. 
These latter algorithms rely on intuition from convex geometry, as they search over candidate MECs associated to the vertices of a convex polytope by greedily optimizing the score via a walk along the edge graph of the polytope. 

Convex geometry also appears in causal discovery via the linear programming approach proposed by Studen\'y. 
In \cite{hemmecke2012characteristic}, Hemmecke, Studen\'y and Lindner defined a unique $0/1$-vector for each MEC of DAGs called its \emph{characteristic imset vector}.
We let $\CIM_\mathcal{S}$ denote the convex hull of a subset $\mathcal{S}$ of characteristic imsets and $\CIM_p$ denote the convex hull of all imsets for DAGs on nodes $[p]$. 
The \emph{characteristic imset (CIM) polytope} $\CIM_\mathcal{S}$ is then the feasible region of a linear program.
Studen\'y showed that classic score functions are expressible as linear functions $\alpha\cdot x$ over these polytopes, where $\alpha$ is a vector depending only on the data and pre-specified modeling assumptions \cite{studeny2006probabilistic}. 

Employing linear optimization techniques typically requires a representation of $\CIM_\mathcal{S}$ as an intersection of half-spaces; e.g., an \emph{H-representation} of $\CIM_{\mathcal{S}}$. 
As is well-known in convex geometry and combinatorial optimization, it is often difficult to recover the H-representation of a polytope from its \emph{V-representation}; i.e., its representation as a convex hull of finitely many points \cite{ziegler2012lectures}. 
This difficulty is observed for CIM polytopes, where only partial results and conjectures regarding H-representations exist for interesting choices of $\mathcal{S}$ \cite{lindner2012discrete, studeny2017towards}.
In the special case that one fixes a variable ordering and a directed bipartite graph structure, Xi and Yoshida showed that the corresponding CIM polytope is a product of simplices \cite{xi2015characteristic}. 


To obtain more general results, we draw intuition from the historical development of causal discovery, which began in the case where the underlying causal DAG has adjacencies specified by a tree $G= ([p],E)$. 
One of our main results is an H-representation of $\CIM_G$, which denotes the convex hull of all characteristic imsets of DAGs whose \emph{skeleton}, e.g, its undirected graph of adjacencies, is equal to $G$. 
Hence, our results offer a novel two-phased optimization approach to causal discovery with a first phase that estimates a tree skeleton $\hat G$ for the data-generating DAG and a second phase which solves the linear program of Studen\'y over $\CIM_{\hat G}$. 
We note that the assumption of a tree skeleton is not new, having been used in \cite{jakobsen2022structure,LRS2022b,rebane2013recovery}.

As a second contribution, we introduce the \emph{interventional CIM polytopes} for learning $\ci$-MECs in Section~\ref{sec:interventionalCIMpolytopes}.
The interventional CIM polytope $\CIM_\mathcal{S}^\ci$ is the convex hull of characteristic imsets of $\ci$-DAGs with imsets in $\mathcal{S}$ for experiments with interventions targeting $\ci = (I_0, I_1,\ldots, I_K)$. 
Our results on $\CIM_G$ are a corollary to our results on $\CIM_G^\ci$, which are summarized as follows.

\begin{theorem}
\label{thm:mainthm}
Let $G$ be an undirected tree and $\mathcal{I} = (\emptyset, \{i_1\},\ldots, \{i_K\})$ a sequence of intervention targets where $i_1,\ldots, i_K$ are leaves of $G$. 
The CIM polytope $\CIM_G^\mathcal{I}$ has H-representation given by
\begin{itemize}
	\item the star inequalities (Definition~\ref{def:starinequality}),
	\item the bidirected-edge inequalities (Definition~\ref{definition:bidirectedEdgeInequality}), and
	\item the forked-tree inequalities (Definition~\ref{def:forkedinequality}).
\end{itemize}
\end{theorem}

The H-representation for $\CIM_G$ is obtained from Theorem~\ref{thm:mainthm} by taking $\ci = (\emptyset)$. 
The proof of Theorem~\ref{thm:mainthm} follows from Theorem~\ref{theorem:facetsOfCIMTIJ}, which captures a slightly more general result. 
These results are obtained by realizing the relevant interventional CIM polytopes as \emph{toric fiber products} \cite{sullivant2007toric}. 
This technique for identifying H-representations of CIM polytopes is novel, and may admit broader applications. 
Finally, in Section~\ref{sec:applications}, we derive the data vector $\alpha$ for Gaussian $\ci$-DAG models and apply the causal discovery algorithm resulting from Theorem~\ref{thm:mainthm} to learn an $\ci$-MEC on a real data example. 
Since the CIM-polytopes used in our method are shown to be \emph{quasi-independence gluings} (QIGs) in \cite{hollering2022toric}, our causal discovery algorithm is called QIGTreeLearn.
We end with a brief discussion of future directions motivated by this work.

\section{Preliminaries}

\subsection{DAG models}\label{subsec:dagmodels}

Let $\vX = [X_1,\ldots, X_p]^T$ be a vector of jointly distributed random variables with joint probability density (or mass) function $f_{\vX}(\vx)$, where $\vx = [x_1,\ldots, x_p]^T$ denotes an outcome (or realization) of the random vector $\vX$.
Given a subset $S\subseteq[p]$ we let $\vX_S = [X_i : i\in S]^T$ denote the subvector of $\vX$ consisting of those variables with indices in $S$.  
Given disjoint subsets $A,B\subseteq [p]$, we let $f_{\vX_A}(\vx_A)$ denote the probability density (or mass) function of $\vX_A$; e.g., the density function for the marginal distribution $\vX_A$:
We further let $f_{\vX_A | \vX_B}(\vx_A | \vx_B)$ denote the density (mass) function of the conditional distribution of $\vX_A$ given $\vX_B = \vx_B$. 


Let $\cg = ([p], E)$ be a directed acyclic graph (DAG) with vertex set $[p] := \{1,\ldots, p\}$ and set of edges $E$. 
For $i\in[p]$ we denote the set of \emph{parents} of $i$ in $\cg$ by
\[
\pa_\cg(i) := \{j\in [p] : j\rightarrow i \in E\}. 
\]
A distribution $\vX = [X_1,\ldots, X_p]^T$ is \emph{Markov} to the DAG $\cg$ if 
\[
f_{\vX}(\vx) = \prod_{i=1}^p f_{X_i | \vX_{\pa_\cg(i)}}(x_i | \vx_{\pa_\cg(i)}).
\]
The set $\cm(\cg)$ of all distributions Markov to the DAG $\cg$ is called the \emph{DAG model} of $\cg$. 

\begin{example}\label{ex:gaussians}
It is typical to consider only those distributions drawn from a specific, parametrized family when modeling a data-generating distribution. 
For instance, we may be interested in modeling with a multivariate normal (Gaussian) distribution. 
In this case, we restrict our attention to the \emph{Gaussian DAG model}, also denoted $\cm(\cg)$ for convenience, consisting of all multivariate normal distributions that are Markov to the DAG $\cg$. 
Distributions in the Gaussian DAG model $\cm(\cg)$ may be parameterized by $(\vec{\lambda}, \vec{\omega})\in\rr^{E}\times \rr^{p}_{>0}$ and correspond to all random vectors  $\vX$ given by the system of equations
\[
X_i = \sum_{j\in\pa_\cg(i)}\lambda_{ij}X_j + \varepsilon_i, \qquad \mbox{for all $i\in[p]$,}
\]
where $\varepsilon_1,\ldots, \varepsilon_p$ are independent univariate normal random variables for which $\varepsilon_i$ has variance $\omega_i$. 
Following standard practices, we may also assume that $\varepsilon_i$ has mean $0$ for all $i\in[p]$. 

Encoding the parameters $(\vec{\lambda}, \vec{\omega})$ as a pair of matrices $\Lambda = [\lambda_{ij}]$ where $\lambda_{ij} = 0$ whenever $j\rightarrow i\notin E$ and $\Omega = \diag(\omega_1,\ldots, \omega_p)$, we obtain that $\vX$ is a multivariate normal distribution with mean $\vec{0}$ and covariance matrix $\Sigma = \phi_\cg(\Lambda, \Omega)$ where $\phi_\cg$ is the map
\[
\phi_\cg: (\Lambda,\Omega) \longmapsto (\eye_p - \Lambda)^{-T}\Omega(\eye_p - \Lambda)^{-1}. 
\]
The covariance matrix $\Sigma\in \rr^{p\times p}$ is a positive definite matrix that uniquely identifies the (mean $\vec{0}$) multivariate normal distribution $\vX$.
Moreover, the parameters $(\Lambda, \Omega)$ can be uniquely identified from the matrix $\Sigma$ for a given DAG $\cg$. 
Hence, the Gaussian DAG model for $\cg$ may be viewed as a subset of the $p$-dimensional positive definite cone $\textrm{PD}^p$
\[
\cm(\cg) = \{\phi_\cg(\Lambda, \Omega) \in\textrm{PD}^p: (\vec{\lambda},\vec{\omega})\in \rr^{E}\times \rr^p_{>0}\}.
\]
Gaussian DAG models are extremely well-studied in the fields of graphical models \cite{maathuis2018handbook,koller2009probabilistic}, causal inference \cite{peters2017elements} and algebraic statistics \cite{drton2018algebraic, sullivant2018algebraic}.
\end{example}

A fundamental problem is to identify an optimal DAG representation $\cg$ of a data-generating distribution $\vX = [X_1,\ldots, X_p]^T$ based on the information contained in a random sample $\dd = [x_{ij}]\in\rr^{p\times n}$. 
Here, each column $\vx_{:, j}$ of the data matrix $\dd$ is a single outcome drawn from $\vX$, and the columns are assumed to be mutually independent. 
Estimating the desired DAG $\cg$ from $\dd$ is the process known as \emph{causal discovery}. 
To make the problem of causal discovery well-defined we have to account for the possibility that two DAGs may yield the same DAG model. 
We say that two DAG $\cg$ and $\ch$ are \emph{Markov equivalent} if $\cm(\cg) = \cm(\ch)$, and we call the set of all DAGs Markov equivalent to $\cg$ its \emph{Markov equivalence class} (MEC). 
A result of Verma and Pearl characterizes Markov equivalent DAGs as follows:

\begin{theorem}\cite{verma1990equivalence}
\label{thm:verma}
    Two DAGs $\cg$ and $\ch$ are Markov equivalent if and only if they have the same skeleton and v-structures. 
\end{theorem}

In Theorem~\ref{thm:verma}, the \emph{skeleton} of $\cg$ refers to the undirected graph obtained from $\cg$ by replacing each edge in $\cg$ with an undirected edge. 
A \emph{v-structure} in a DAG is an induced path of length two of the form $i\rightarrow j \leftarrow k$.
We sometimes say that $j$ is the center of the v-structure, or that $i$ and $k$ form a v-structure at $j$.

\subsection{Interventional DAG models}\label{subsec:interventionaldagmodels}

In causal inference one would like to interpret an edge $i\rightarrow j$ in a DAG $\cg$ as representing the belief that $X_i$ is a direct cause of $X_j$. 
Markov equivalence may make this interpretation impossible for some edges in the graph $\cg$, since by Theorem~\ref{thm:verma}, we can equally represent the same distribution with another DAG $\ch$ in the MEC of $\cg$. 
Hence, we would like to refine MECs of DAGs, ideally down to singletons so that each graph is identifiable. 

The gold standard approach to refining a Markov equivalence class is to use a mixture of \emph{observational} and \emph{interventional} (sometimes called \emph{experimental}) data. 
In this set-up, we assume there is a data-generating distribution $\vX$ from which we can draw samples. 
This is the \emph{observational distribution}. 
An \emph{interventional distribution} is produced by perturbing the observational distribution through some experiment, and the samples drawn from this distribution are called \emph{interventional data}. 
To formalize this, consider a sequence of subsets $\ci = (I_0,I_1,\ldots, I_K)$, $I_k\subseteq[p]$ where we let $I_0 = \emptyset$.
A set $I_k\in\ci$ is called an \emph{intervention target}. 

Suppose now that $\vX\in \cm(\cg)$ is our observational distribution with density $f^{0}(\vx)$.
Since $\vX\in \cm(\cg)$, we have that
\[
f^{0}(\vx) = \prod_{i=1}^pf^{0}_{X_i | \vX_{\pa_\cg(i)}}(x_i | \vx_{\pa_\cg(i)}). 
\]
For $I_k\in\ci$ with $k\neq 0$, we define the \emph{interventional density} (with respect to $\cg$)
\[
f^{k}(\vx) = \prod_{i\in I_k} f^{k}_{X_i | \vX_{\pa_\cg(i)}}(x_i | \vx_{\pa_\cg(i)})\prod_{i\notin I_k} f^{0}_{X_i | \vX_{\pa_\cg(i)}}(x_i | \vx_{\pa_\cg(i)}). 
\]
An \emph{interventional setting} is a sequence $(f^{0},\ldots, f^{K})$ of interventional densities with respect to $\cg$ and $\ci$. 
The \emph{interventional DAG model} for $\cg$ and $\ci$ is the collection of interventional settings
\begin{equation*}
    \begin{split}
        \cm(\cg,\ci) := \{(f^{0},\ldots&, f^{K}) : f^{k}\in\cm(\cg) \mbox{ for all $k\in[K]$, }\\
        &f^{k}_{X_i | \vX_{\pa_\cg(i)}}(x_i | \vx_{\pa_\cg(i)}) = f^{0}_{X_i | \vX_{\pa_\cg(i)}}(x_i | \vx_{\pa_\cg(i)}) \mbox{ for all $i\notin I_k$ for all $k\in [K]$}\}.
    \end{split}
\end{equation*}
One should interpret the above as stating that the distribution $f^k$ is produced from the observational distribution by performing some experiment that changes the behavior of the variable $X_k$ in the system. 
When building an interventional model for observed experimental data where one targets the node $j$, we can capture the observed consequences of the experiment by introducing additional parameters when modeling the conditional factor $f^{k}_{X_i | \vX_{\pa_\cg(i)}}(x_i | \vx_{\pa_\cg(i)})$.  

\begin{example}
    \label{ex:interventionalgaussian}
    Let $\vX\in\cm(\cg)$ be a multivariate normal observational distribution specified by the parameters $(\vec{\lambda},\vec{\omega})$ as in Example~\ref{ex:gaussians}. 
    It follows that the observational distribution is given by 
    \[
    X_i = \sum_{j\in\pa_\cg(i)}\lambda_{ij}X_j + \varepsilon_i, \qquad \mbox{for all $i\in[p]$}.
    \]
    Suppose now that we consider $\ci = (I_0 = \emptyset, I_1,\ldots, I_K)$ for some $I_k \subseteq[p]$ and $k\in[K]$. 
    An interventional distribution $\vX^{(k)} = \left[X_1^{(k)},\ldots, X_p^{(k)}\right]^T$ for the target set $I_k$ is 
    \begin{equation*}
    \begin{split}
    X_i^{(k)} &= \sum_{j\in\pa_\cg(i)}\lambda_{ij}^{(k)}X_j^{(k)} + \varepsilon_i^{(k)}, \qquad \mbox{for all $i\in [p]$,}
    \end{split}
    \end{equation*}
    where $\lambda_{ij}^{(k)} = \lambda_{ij}$ and $\omega_i^{(k)} = \omega_i$ whenever $i\notin I_k$.
    Letting $\Sigma^{(0)} = \phi_\cg(\vec{\lambda}, \vec{\omega})$ and $\Sigma^{(k)} = \phi_\cg(\vec{\lambda}^{(k)},\vec{\omega}^{(k)})$ for $k\in[K]$, we obtain the interventional setting $\left(\Sigma^{(0)},\ldots, \Sigma^{(K)}\right)$ in the Gaussian interventional model $\cm(\cg,\ci)$. 
\end{example}

We say that two DAGs $\cg$ and $\ch$ are \emph{$\ci$-Markov equivalent} whenever $\cm(\cg,\ci) = \cm(\ch,\ci)$, and that they belong to the same \emph{$\ci$-Markov equivalence class} ($\ci$-MEC). 
To graphically represent the interventions, we introduce nodes $z_k$ for $k\in[K]$ and define the $\ci$-DAG associated to $\cg$ and $\ci$. \begin{definition}
    \label{def:Idag}
    Let $\cg = ([p], E)$ be a DAG and $\ci = (I_0 = \emptyset,I_1,\ldots, I_K)$ a sequence of intervention targets with $I_k\subseteq[p]$. 
    The $\ci$-DAG associated to $\cg$ and $\ci$ is $\cg^\ci = ([p]\cup \mathcal{Z}_\ci,E^\ci)$ where 
    \[
    \mathcal{Z}_\ci = \{z_k : k\in[K]\} \qquad \mbox{and} \qquad E^\ci := E \cup \bigcup_{k=1}^K\{z_k \rightarrow i : i\in I_k\}.
    \]
\end{definition}
The following theorem generalizes Theorem~\ref{thm:verma} and shows how $\ci$-MECs refine MECs of DAGs:

\begin{theorem}\cite{yang2018characterizing}
    \label{thm:Iverma}
    Two DAGs $\cg$ and $\ch$ are $\ci$-Markov equivalent if and only if $\cg^\ci$ and $\ch^\ci$ have the same skeleton and v-structures.
\end{theorem}

\subsection{Imsets}\label{subsec:imsets}
Imsets are integer-valued functions introduced by Studen\'y \cite{studeny2006probabilistic} to offer discrete geometric and integer linear programming perspectives on DAG models.
Let $2^{[p]}$ denote the power set of $[p]$.
For a subset $A\subseteq[p]$, we let $\delta_A: 2^{[p]} \longrightarrow \{0,1\}$ denote the \emph{identifier} function of $A$; e.g., 
\[
\delta_{A}(B) = 
\begin{cases}
    1   &   \mbox{ if $B = A$},\\
    0   &   \mbox{otherwise}.
\end{cases}
\]

Given a DAG $\cg = ([p], E)$ we define the \emph{standard imset} of $\cg$ to be
\[
u_\cg := \delta_{[p]} - \delta_\emptyset + \sum_{i = 1}^p\left(\delta_{\pa_\cg(i)} - \delta_{\fa_\cg(i)}\right), 
\]
where $\fa_\cg(i) = \{i\}\cup\pa_\cg(i)$ is called the \emph{family} of $i$. 


Since $u_\cg$ is an integer-valued function we can treat it as an integer-valued vector $u_\cg\in\rr^{2^{[p]}}$, where we let $u_\cg(S)$ denote the coordinate of $u_\cg$ indexed by $S\subseteq[p]$.

\begin{definition}\cite{hemmecke2012characteristic}
    \label{def:characteristicimset}
    The \emph{characteristic imset} of a DAG $\cg$ is the vector $c_\cg\in\rr^{2^{[p]}}$ where
    \[
    c_\cg(S) = 1 - \sum_{S \subseteq T}u_\cg(T), \qquad \mbox{for all $S\subseteq[p]$.}
    \]
    We let $\textrm{IM}(p) = \{c_\cg : \cg \mbox{ is a DAG on node set $[p]$} \}$.
\end{definition}

We note that $c_\cg(S) = 1$ for all $S\subseteq[p]$ with $|S| \leq 1$. 
Hence, we may realize $c_\cg$ as a vector in $\rr^{2^{p} - p - 1}$. 
We note also that the transformation from standard imsets $u_\cg$ to characteristic imsets $c_\cg$ is an isomorphism with inverse given by M\"obius inversion. 
Characteristic imsets are $0/1$-vectors with the following combinatorial interpretation.

\begin{theorem}\cite[Theorem~1]{hemmecke2012characteristic}
    \label{thm:imsetinterp}
    Let $\cg = ([p], E)$ be a DAG. The characteristic imset $c_\cg$ satisfies
    \[
    c_\cg(S) = 
    \begin{cases}
        1 & \mbox{if there exists $i\in S$ such that $j\in\pa_\cg(i)$ for all $j\in S\setminus i$,}\\
        0 & \mbox{otherwise.}
    \end{cases}
    \]
\end{theorem}
In \cite{hemmecke2012characteristic} it is shown that $c_\cg = c_\ch$ if and only if $\cg$ and $\ch$ are Markov equivalent.

\section{Interventional CIM polytopes}\label{sec:interventionalCIMpolytopes}

In this section we introduce interventional characteristic imset polytopes, which serve as feasible regions for integer linear programs whose solutions correspond to optimal $\ci$-DAGs. 
In subsection~\ref{subsec:convexpolytopes} we first recall the necessary preliminaries for convex polytopes. 
Then in subsection~\ref{subsec:imsetpolytopes} we recall the characteristic imset polytopes introduced in \cite{hemmecke2012characteristic}.
The interventional generalization is then presented in subsection~\ref{subsec:interventionalCIMpolytopes}.

\subsection{Convex polytopes.}\label{subsec:convexpolytopes}
A \emph{V-polytope} $P$ is the convex hull of a finite set of points $\{v_1,\ldots, v_m\}\in \rr^n$:
\[
P = \conv(v_1,\ldots, v_m) = \left\{ x\in \rr^n : x = \sum_{i = 1}^m\lambda_i v_i, \, \sum_{i = 1}^m \lambda_i = 1, \, \lambda_i \geq 0 \mbox{ for all } i\in[m]\right\}.
\]
The set of points $\{v_1,\ldots, v_m\}$ is called a \emph{V-representation} of $P$. 
An \emph{H-polyhedron} $P$ is the solution set to a system of linear inequalities $\langle a_1,x\rangle  \leq b_1,\ldots, \langle a_t,x\rangle\leq b_t$ where $a_1,\ldots, a_t\in\rr^n$ and $b_1,\ldots, b_t\in \rr$.
An H-polyhedron that does not contain any ray (a set of the form $\{x + \lambda y : \lambda \geq 0\}$ for some $x, y\in \rr^n$ with $y\neq 0$) is called an \emph{H-polytope}.
\begin{theorem}\cite[Theorem~1.1]{ziegler2012lectures}
    \label{thm:VH}
    A subset $P\subseteq\rr^n$ is a V-polytope if and only if it is an H-polytope.
\end{theorem}
Hence we refer to both V-polytopes and H-polytopes as polytopes. In combinatorial optimization problems, it is typical to define a convex polytope $P$ as the convex hull of a finite set of points that correspond to combinatorial objects, such as a collection of DAGs or MECs. Theorem~\ref{thm:VH} says that such a polytope $P$ admits an H-representation, whose inequalities are typically used to solve the linear program with feasible region $P$. Recovering an H-representation of a V-polytope is, in general, a nontrivial task. 

An inequality $a\cdot x \leq b$ is \emph{valid} for a polytope $P$ of dimension $d$ if it is satisfied for all $x\in P$. 
Each valid inequality $a\cdot x \leq b$ on $P$ defines a \emph{face} of $P$, which is a set $F = \{x\in \rr^n : a\cdot x = b\}\cap P$. The face of $P$ defined by $a\cdot x \leq b$ is itself a polytope with H-representation given by the inequalities defining $P$ together with $a\cdot x \leq b$ and $-a\cdot x \leq -b$.

The \emph{dimension} of a polytope $P$ is the dimension of its \emph{affine span}
\[
\left\{x\in \rr^n : x = \sum_{i = 1}^m\lambda_i p_i \mbox{ for some $p_1,\ldots, p_m\in P$, and $\lambda_1,\ldots, \lambda_m\in\rr$ with $\sum_{i=1}^m \lambda_i = 1$}\right\}.
\]
Similarly the dimension of a face $F$ is the dimension of its affine hull.
The $0$-dimensional faces of a d-dimensional polytope $P$ are the \emph{vertices} of $P$, its $1$-dimensional faces are called its \emph{edges} and its $(d-1)$-dimensional faces are its \emph{facets}. 
The minimal set of inequalities needed to constitute an H-representation of a polytope $P$ are the equations defining the affine span of $P$ together with the facet-defining inequalities of $P$. 
This representation of $P$ is called its \emph{minimal H-representation}.
For more details on the basics of convex polytopes, we refer the reader to \cite{ziegler2012lectures}. 


\subsection{Characteristic imset polytopes}\label{subsec:imsetpolytopes}
Imsets allow us to phrase the problem of causal discovery introduced in Subsection~\ref{subsec:dagmodels} as a linear program with feasible region a convex polytope in which each vertex corresponds to a unique MEC of DAGs. 
\begin{definition}
The \emph{characteristic imset polytope} (or \emph{CIM polytope}) of a collection of characteristic imsets $\mathcal{S}$ is
\[
\CIM_\mathcal{S} := \conv(c_\cg\in \mathcal{S})\subset \rr^{2^{[p]}}.
\]
\end{definition}
When $\mathcal{S}$ is all characteristic imsets for DAGs on vertex set $[p]$, we let $\CIM_p := \CIM_\mathcal{S}$, and when $\mathcal{S} = \{c_\cg : \cg \mbox{ a DAG with skeleton $G$}\}$ we let $\CIM_G := \CIM_\mathcal{S}$. 
The polytope $\CIM_p$ provides a general discrete geometry model for the problem of causal discovery.
The subpolytopes $\CIM_\mathcal{S}$, such as $\CIM_G$, provide a discrete geometry model for hybrid algorithms that restrict their optimization phase to searching over the MECs in $\mathcal{S}$, where $\mathcal{S}$ is obtained via the constraint-based phase of the algorithm. 
The polytopes $\CIM_G$ were observed to have the following geometric relationship with $\CIM_p$.

\begin{proposition}\cite[Proposition 2.4]{LRS2022a}
    \label{prop:skeletalfaces}
    Let $G = ([p],E)$ and $H = ([p], E^\prime)$ be undirected graphs satisfying $E\subseteq E^\prime$, and let 
    \[
    \mathcal{S} = \{c_\cg : \cg \mbox{ has skeleton $S = ([p], D)$ satisfying $E\subseteq D\subseteq E^{\prime}$}\}.
    \]
    Then $\CIM_\mathcal{S}$ is a face of $\CIM_p$. 
\end{proposition}
More recently, it has also be observed that the collection of imsets whose associated DAG models entail the CI relation $\vX_A \independent \vX_B | \vX_C$ form a face of $\CIM_p$ \cite{hu2022towards}.

Other well-studied examples of CIM polytopes include $\CIM_\mathcal{S}$ where $\mathcal{S}$ is the set of all perfect DAGs on nodes $[p]$ \cite{lindner2012discrete, studeny2017towards, studeny2021dual}. 
Since perfect DAGs represent the same models as chordal graphs, this CIM polytope is commonly referred to as the \emph{chordal graph polytope}.

CIM polytopes are defined as V-polytopes, but to use them for linear optimization purposes we require a description of its higher dimensional faces; ideally, an H-representation. 
Aside from one example associated to medical diagnostic models where the associated CIM polytope is a product of simplices \cite{xi2015characteristic}, there are no families of CIM polytopes with known H-representations, despite consistent efforts \cite{lindner2012discrete, studeny2017towards, studeny2021dual}.  
A main contribution of this paper is an H-representation for $\CIM_T$ when $T$ is a tree. 

\subsection{Interventional CIM polytopes}\label{subsec:interventionalCIMpolytopes}
To derive an H-representation for $\CIM_T$, we compute the H-representations for a more general family of polytopes associated to causal discovery using a combination of observational and interventional data. 

\begin{definition}
Let $p$ and $K$ be positive integers. The \emph{interventional standard imset polytope} $\SIM_{p,K}$ is defined by
\[\SIM_{p,K} := \conv\left\{u_{\cg^\ci}: \cg^\ci~\mbox{is an}~\ci\mbox{-DAG on}~[p]~\mbox{with intervention targets}~\ci = (\emptyset, I_1,\dots,I_K)\right\}.\]
The \emph{interventional characteristic imset polytope} $\cim_{p,K}$ is defined by
\[\cim_{p,K} := \conv\left\{c_{\cg^\ci}: \cg^\ci~\mbox{is an}~\ci\mbox{-DAG on}~[p]~\mbox{with intervention targets}~\ci = (\emptyset, I_1,\dots,I_K)\right\}.\]
\end{definition}

Up to relabeling of nodes, an $\ci$-DAG on $[p]$ with $K$ intervention targets is a DAG on $[p+K]$. We have the following observation.

\begin{observation}
\label{obs:cimpKasSubpolytope}
$\SIM_{p,K}$ is a subpolytope of $\SIM_{p+K}$ and $\cim_{p,K}$ is a subpolytope of $\cim_{p+K}$.
\end{observation}

From Observation~\ref{obs:cimpKasSubpolytope} and Definition~\ref{def:characteristicimset} the following proposition is immediate.

\begin{proposition}
There is an affine isomorphism $\SIM_{p,K} \to \cim_{p,K}$.
\end{proposition}

\begin{proof}
By Definition~\ref{def:characteristicimset} there is an affine isomorphism $\SIM_{p+K} \to \cim_{p+K}$ which maps $u_{\cg^\ci} \mapsto c_{\cg^\ci}$. This isomorphism restricts to an isomorphism $\SIM_{p,K} \to \cim_{p,K}$.    
\end{proof}

We later express the Bayesian information criterion as a linear functional over $\SIM_{p,K}$ in the Gaussian setting (Theorem~\ref{thm:IBIC}). More generally, any regular decomposable score function will factor as a linear functional over $\SIM_{p,K}$, see \cite{studeny2006probabilistic}. 

As in the setting without interventions, there are many subpolytopes of $\cim_{p,K}$ that are of interest in causal discovery.

\begin{definition}
\label{def:fixedInterventionalCIM}
Let $p$ be a positive integer and fix a sequence of intervention targets $\ci = (\emptyset,I_1,\dots,I_K)$. Then
\[\cim_p^\ci := \conv\left\{ c_{\cg^\ci}: \cg~\mbox{is a DAG on}~[p]\right\}.\]
\end{definition}

Given a sequence of $K$ experiments, the targeted vertices may be unknown if the specific way in which the observational distribution was affected by each experiment is unknown. 
In this case one may learn the intervention targets and the DAG simultaneously via a linear program over $\cim_{p,K}$. However, in some cases, such as for gene knockout experiments with CRISPR-Cas9 \cite{dixit2016perturb,sachs2005causal}, the intervention targets may be specified. 
In such cases, one may reduce the dimension of the polytope and solve a linear program in a significantly smaller space. This is illustrated in the following proposition.

\begin{proposition}
\label{prop:interventionTargetFace}
Let $p$ be a positive integer and let $\ci = (\emptyset, I_1,\dots,I_K)$ be a fixed sequence of intervention targets. Then $\cim_p^\ci$ is a face of $\cim_{p,K}$.
\end{proposition}

\begin{proof}
Define a functional $a \in \left(\mathbb{R}^{2^{p + |I|} - p - |I| - 1}\right)^*$ to have $S$ coordinate
\[a_S = \begin{cases}
    1 &\mbox{if}~S=\{i,z_j\}~\mbox{where}~i \in I_j \\
    -1 &\mbox{if}~S=\{i,z_j\}~\mbox{where}~i \not\in I_j \\
    0 &\mbox{otherwise.}
\end{cases}\]
For a DAG $\cg$ with intervention targets $\cj$, clearly $a \cdot c_{\cg^\cj}$ is maximized if and only if $\cj = \ci$.
\end{proof}

If one also knows the skeleton of the DAG to be learned, then the search space can be further narrowed.

\begin{definition}
\label{def:cimGI}
Let $G$ be an undirected graph on $[p]$ and let $\ci = (\emptyset, I_1,\dots,I_K)$ be a sequence of intervention targets. Then
\[\cim_G^\ci = \conv\left\{ c_{\cg^\ci}: \cg~\mbox{is a DAG on}~[p]~\mbox{with}~\skel(\cg) = G\right\}.\]
\end{definition}

\begin{proposition}
\label{prop:cimGI}
Let $G$ be a undirected graph on $[p]$ and let $\ci = (\emptyset,I_1,\dots,I_K)$ be a sequence of intervention targets. Then $\cim_G^\ci$ is a face of both $\cim_p^\ci$ and $\cim_{p,K}$.
\end{proposition}

\begin{proof}
It suffices to show that $\cim_G^\ci$ is a face of $\cim_p^\ci$ since by Proposition~\ref{prop:interventionTargetFace} we know $\cim_p^\ci$ is a face of $\cim_{p,K}$. Analogous to the proof of Proposition~\ref{prop:interventionTargetFace}, $\cim_p^\ci$ is exposed by a functional which penalizes nonedges of $G$ and promotes edges of $G$.
\end{proof}

More generally, given a subset $\mathcal{S}$ of characteristic imsets $c_{\cg^\ci}$ for a sequence of intervention targets $\ci$, we denote their convex hull with $\CIM_\mathcal{S}^\ci$.

\section{Notation}\label{sec:notation}


This section provides some notation that will be used in the derivation of the proof of Theorem~\ref{thm:mainthm}. 
This notation will be used throughout Sections~\ref{section:starGraphs},~\ref{section:interventionsontrees} and~\ref{section:eliminatingInterventions}. Throughout, given a graph $G$ we will denote its vertex set by $V(G)$.
In these sections, we will focus on models for $\ci$-DAGs $\cg^\ci$ where $\cg = ([p], E)$ is a DAG on node set $[p]$ and $\ci$ is a sequence of intervention targets $\ci = (I_0 = \emptyset, I_1, \ldots, I_K)$ where $|I_k| = 1$ for all $k\in [K]$.
For each target $I_k = \{i\}$, we may also represent the node $z_k$ in the $\ci$-DAG $\cg^\ci$ simply by $i^\prime$. 
Since each nonempty intervention target $I_k$ is a singleton, we may represent the sequence of intervention targets with the set $I = \bigcup_{i\in[K]}I_k$. 
Analogously, we denote the corresponding the interventional CIM polytope $\CIM_\mathcal{S}^\ci$ simply by $\CIM_\mathcal{S}^I$. 

Given an undirected graph $G = ([p], E)$ and a sequence of intervention targets $\ci$ with all nonempty targets being singletons, we denote by $G^I$ the partially directed graph on node set $[p]\cup \{i^\prime : i\in I\}$ with edge set
\[
E \cup \{i^\prime \rightarrow i : i\in I\}. 
\]

We will also make use of certain undirected graphs, commonly called star graphs in the combinatorics literature. 
A \emph{star graph with center node $c$} is an undirected tree $S = ([p], E)$ with $p \geq 3$ where each node in $[p]\setminus c$ is a \emph{leaf node} (a degree 1 node) in $S$. 
The unique edge incident to a leaf node is called a \emph{leaf}.
We emphasize that, for the purposes of this article, star graphs have at least 2 leaf nodes.

Additionally, we will say a directed edge in a DAG $\cg$ is \emph{essential} if the edge points the same direction in all elements of the MEC (or $\ci$-MEC) of $\cg$. 

\begin{example}
    \label{ex:notation}
    Consider the star graph $S$ on node set $\{a,b,c,d\}$ with center node $c$ and the sequence of intervention targets $\ci = (I_0 = \emptyset, I_1 = \{a\}, I_2 = \{d\})$ depicted in Figure~\ref{fig:starInterventions}. 
    Then $I = \{a,d\}$. 
    The DAG $\cg$, on the left in Figure~\ref{fig:starInterventions} has skeleton $S$, depicted to its right. 
    To the right of $S$ is an $\ci$-DAG $\cg^\ci$ with skeleton $S$ and intervention targets $\ci$. 
    To further indicate that the nodes $a$ and $d$ are targeted by interventions, we color them white. 
    Note that the edges $a\rightarrow c$ and $c\rightarrow b$ are essential due to the presence of the edge $a'\rightarrow a$ as reversing either (or both) of $a\rightarrow c$ or $c\rightarrow b$ results in a new v-structure and hence a DAG not in the same $\ci$-MEC of $\cg$ via Theorem~\ref{thm:Iverma}. 
    Similarly, the edge $c\rightarrow d$ is essential due to the presence of the edge $d'\rightarrow d$. 
    
    The partially directed graph $S^I$ is the right-most graph in Figure~\ref{fig:starInterventions}.
    The interventional CIM polytope $\CIM_S^I$ is the convex hull of the characteristic imset vectors of all $\ci$-DAGs $\cg^\ci$ where $\cg$ has skeleton $S$. 
\end{example}


\section{Leaf Interventions on Star Graphs}
\label{section:starGraphs}
Throughout this section we let $S$ be a star graph with center node $c$ and consider interventions on a subset $I$ of the leaves of $S$. Recall that in Section~\ref{sec:notation} we defined star graphs to have at least 3 nodes, and so at least 2 leaves. We compute the facet-defining inequalities of $\cim_S^I$, which have three distinct types, and give combinatorial interpretations of the inequalities. 

\subsection{Star Inequalities}
\label{subsection:starInequalities} 
We first consider the following inequality.

\begin{definition}\label{def:starinequality}
Let $S'$ be a star subgraph of $S$. The \emph{star inequality} associated to $S'$ is
\begin{equation}
\label{ieq:starInequality}
\sum\limits_{V(S') \subseteq A \subseteq V(S)} (-1)^{\#A \setminus V(S')} x_A \geq 0.
\end{equation}
\end{definition}

\begin{example}
Consider the star with leaf interventions in Figure~\ref{fig:starInterventions}. The inequalities defining $\cim_S^I$ are
\begin{align*}
x_{abc} - x_{abcd} \geq 0 \quad\quad&\left(x_{aa' c} \right) + \left( x_{abc} + x_{acd} - x_{abcd} \right) \leq 1 \\
x_{bcd} - x_{abcd} \geq 0 \quad\quad&\left(x_{cdd'} \right) + \left( x_{acd} + x_{bcd} - x_{abcd} \right) \leq 1 \\
x_{acd} - x_{abcd} \geq 0 \quad\quad&\left( 1 - x_{aa'c} \right) + \left(1 - x_{cdd'} \right) \leq 1 + x_{acd}. \\
x_{abcd} \geq 0 \quad\quad
\end{align*}
The left four inequalities are star inequalities.

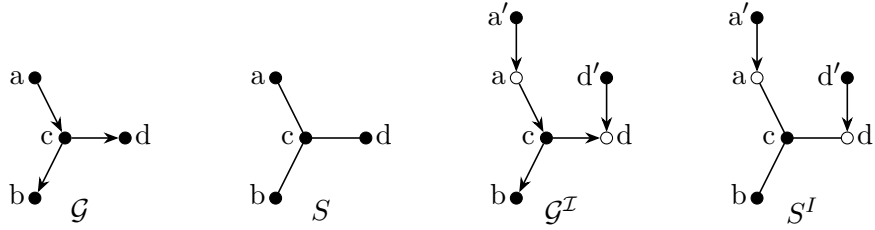
\begin{figure}
\centering
\begin{tikzpicture}[scale = .8]

\begin{scope}

\node[wB] (c) at (0-12,0) {};
\node[wB] (d) at (1-12,0) {};
\node[wB] (a) at (-0.5-12,1) {};
\node[wB] (b) at (-0.5-12,-1) {};

\draw[semithick, -Stealth] (a) -- (c);
\draw[semithick, -Stealth] (c) -- (b);
\draw[semithick, -Stealth] (c) -- (d);

\node at (-0.8-12,1) {a};
\node at (-0.8-12,-0.93) {b};
\node at (-0.3-12,0) {c};
\node at (1.3-12,0.07) {d};

\node at (0.25-12, -1.2) {$\cg$};

\draw [semithick] (0-8,0)--(1-8,0);
\draw [semithick] (0-8,0)--(-0.5-8,1);
\draw [semithick] (0-8,0)--(-0.5-8,-1);

\node[wB] (c) at (0-8,0) {};
\node[wB] (d) at (1-8,0) {};
\node[wB] (a) at (-0.5-8,1) {};
\node[wB] (b) at (-0.5-8,-1) {};

\node at (-0.8-8,1) {a};
\node at (-0.8-8,-0.93) {b};
\node at (-0.3-8,0) {c};
\node at (1.3-8,0.07) {d};

\node at (0.25-8, -1.2) {$S$};

\draw [semithick, -Stealth] (-0.5-4,2)--(-0.5-4,1.1);
\draw [semithick, -Stealth] (1-4,1)--(1-4,0.1);

\node[wB] (c) at (0-4,0) {};
\node[wW] (d) at (1-4,0) {};
\node[wW] (a) at (-0.5-4,1) {};
\node[wB] (b) at (-0.5-4,-1) {};
\node[wB] at (-0.5-4,2) {};
\node[wB] at (1-4,1) {};

\draw[semithick, -Stealth] (a) -- (c);
\draw[semithick, -Stealth] (c) -- (b);
\draw[semithick, -Stealth] (c) -- (d);

\node at (-0.8-4,1) {a};
\node at (-0.8-4,-0.93) {b};
\node at (-0.3-4,0) {c};
\node at (1.3-4,0.07) {d};
\node at (-0.8-4,2.07) {a$'$};
\node at (0.7-4,1.07) {d$'$};

\node at (0.25-4, -1.2) {$\cg^\ci$};

\draw [semithick] (0,0)--(1,0);
\draw [semithick] (0,0)--(-0.5,1);
\draw [semithick] (0,0)--(-0.5,-1);
\draw [semithick, -Stealth] (-0.5,2)--(-0.5,1.1);
\draw [semithick, -Stealth] (1,1)--(1,0.1);

\node[wB] at (0,0) {};
\node[wW] at (1,0) {};
\node[wW] at (-0.5,1) {};
\node[wB] at (-0.5,-1) {};
\node[wB] at (-0.5,2) {};
\node[wB] at (1,1) {};

\node at (-0.8,1) {a};
\node at (-0.8,-0.93) {b};
\node at (-0.3,0) {c};
\node at (1.3,0.07) {d};
\node at (-0.8,2.07) {a$'$};
\node at (0.7,1.07) {d$'$};

\node at (0.25, -1.2) {$S^I$};
\end{scope}

\end{tikzpicture}

\caption{From left to right: A DAG $\cg$, its skeleton $S$, an $\ci$-DAG with leaf interventions where the arrows $a' \to a$ and $d' \to d$ depict the interventions and the graph $S^I$ for the skeleton $S$ and interventions $\ci = (\emptyset, \{a\}, \{d\})$.} 
\label{fig:starInterventions}
\end{figure}
\end{example}

We note that the definition of star inequalities does not depend on the intervention targets $I$, and so the following proof is also independent of $I$.

\begin{lemma}
\label{lemma:starInequalities}
Let $S'$ be a star subgraph of $S$. Then the star inequality associated to $S'$ defines a facet of $\cim_S^I$.
\end{lemma}

\begin{proof}
Let $\cg$ be a DAG with skeleton $S$. 
Then $c_\cg(A) = 1$ if and only if $c\in A \subseteq \fa_\cg(c)$. 
It follows that
\[\sum\limits_{V(S') \subseteq A \subseteq V(S)} (-1)^{\#A \setminus V(S')} c_\cg(A) = \sum\limits_{V(S') \subseteq A \subseteq \fa_\cg(c)} (-1)^{\#A \setminus V(S')} = \begin{cases}
    1 &\mbox{if}~V(S') = \fa_\cg(c) \\
    0 &\mbox{otherwise}.
\end{cases}\]
Consequently Inequality~\ref{ieq:starInequality} is valid and defines a face $\mathcal{F}$ of $\cim_S^I$. The only DAG resulting in a characteristic imset vector not in $\mathcal{F}$ is given by orienting edges in $S'$ towards $c$ and all other edges away from $c$. 
Hence, $\mathcal{F}$ is a facet.
\end{proof}

\subsection{Bidirected-edge inequalities}
The remaining inequalities defining $\cim_S^I$ depend on $I$. For each internal edge of $S^I$ we now define an inequality which, in a sense we make precise, cuts off imset vectors which correspond to graphs with bidirected edges.

\begin{definition}
\label{definition:bidirectedEdgeInequality}
Let $S$ be a star graph and let $u - v$ be an internal edge of $S^I$. We define a functional

\[\ind_{v \from u}(x) = \begin{cases}
\sum\limits_{\substack{A \subseteq V(S) \setminus \{u,v\} \\ |A| \geq 1} }(-1)^{|A|+1} x_{A \cup \{u,v\}} &~\mbox{if}~v \not\in I \\
x_{uvv'} &\mbox{otherwise.}
\end{cases}\]
The \emph{bidirected-edge inequality} associated to $u - v$ is
\[\ind_{v \from u}(x) + \ind_{u \from v}(x) \leq 1.\]
\end{definition}

\begin{notation}\label{not:directionalitymatters}
The functional $\ind_{v \from u}(x)$ acts as an indicator function on $c_\cg$ for the event that $v \from u$ is part of a v-structure (and similarly $\ind_{u \from v}$) which we will demonstrate in the following example. Note that the direction of the arrow in the notation $\ind_{v \from u}(x)$ is specific.  Later on we will define an additional linear function in which the direction is reversed; e.g., $\ind_{u \to v}(x)$, and the interpretation of these functions will be distinct (see Notation~\ref{not:directionalitymatterstoo}).
\end{notation}

\begin{example}
We again consider the star with leaf interventions in Figure~\ref{fig:starInterventions}. Recall that the inequalities defining $\cim_S^I$ are
\begin{align*}
x_{abc} - x_{abcd} \geq 0 \quad\quad&\left(x_{aa' c} \right) + \left( x_{abc} + x_{acd} - x_{abcd} \right) \leq 1 \\
x_{bcd} - x_{abcd} \geq 0 \quad\quad&\left(x_{cdd'} \right) + \left( x_{acd} + x_{bcd} - x_{abcd} \right) \leq 1 \\
x_{acd} - x_{abcd} \geq 0 \quad\quad&\left( 1 - x_{aa'c} \right) + \left(1 - x_{cdd'} \right) \leq 1 + x_{acd}. \\
x_{abcd} \geq 0 \quad\quad
\end{align*}
We saw earlier that the left four inequalities are all star inequalities. The top two inequalities on the right are both bidirected-edge inequalities associated to the internal edges $a-c$ and $d-c$ respectively. For example, observe that since there is an intervention on node $a$, meaning $a \in I$, the indicator that the edge $a \leftarrow c$ is present is simply $\ind_{a \leftarrow c} = x_{aa'c}$ by definition. On the other hand,
\[
\ind_{c \leftarrow a} 
= \sum\limits_{\substack{A \subseteq \{b,d\} \\ |A| \geq 1} }(-1)^{|A|+1} x_{A \cup \{a,c\}}  = x_{abc} + x_{acd} - x_{abcd}.
\]
So we see that the top inequality on the right hand side in our original list is the bidirected-edge inequality
\[
\left(x_{aa' c} \right) + \left( x_{abc} + x_{acd} - x_{abcd} \right) = \ind_{a \leftarrow c} + \ind_{c \leftarrow a} \leq 1,
\]
which is the bidirected-edge inequality associated to the internal edge $a-c$. 
\end{example}

\begin{proposition}
\label{proposition:atLeastOneColliderIndicator}
Let $S$ be a star tree, let $I$ be a subset of the leaves of $S$, and let $v$ be an non-leaf node of $S^I$. Let $\cg^\ci$ be an $\ci$-DAG with skeleton $S$ and set of leaf targets $I$. For any edge $u - v$ of $S^I$ we have
\[\ind_{v \from u}(c_{\cg^\ci}) = \begin{cases}
    1 &\mbox{if}~ v \from u ~\mbox{is part of a v-structure in}~{\cg^\ci} \\
    0 &\mbox{otherwise}.
\end{cases}\]
\end{proposition}

\begin{proof}
If $ v \in I$, then it is clear that $x_{uvv'} = 1$ if $u \to v$ and 0 otherwise. Hence we will focus on the case $v \not\in I$. Since $v$ is not a leaf node of $S^I$, $v$ must be the center of $S$. First consider the case where $v \to u$ in ${\cg^\ci}$. In this case, every coordinate $c_{\cg^\ci}(A \cup \{u,v\})$ in the functional $\ind_{v \from u}(c_{\cg^\ci})$ is 0 and the functional evaluates to 0.

Otherwise we have $v \from u$ in ${\cg^\ci}$. Let $X = \{x \in V(S) \setminus \{u\}: x \to v\}$. Then for any nonempty $A \subseteq V(S) \setminus \{u,v\}$, we have $c_{\cg^\ci}(A \cup \{u,v\}) = 1$ if and only if $A \subseteq X$. It follows that
\[\sum\limits_{\substack{A \subseteq V(S) \setminus \{u,v\} \\ |A| \geq 1} }(-1)^{|A|+1} c_{\cg^\ci}(A \cup \{u,v\}) = \sum\limits_{\substack{A \subseteq X \\ |A| \geq 1} }(-1)^{|A|+1} = 1 + \sum\limits_{A \subseteq X}(-1)^{|A|+1} = \begin{cases}
    1 &\mbox{if}~X \neq \emptyset \\
    0 &\mbox{otherwise.}
\end{cases}\]
By the definition of $X$ we have that $\ind_{v \from u}(c_{\cg^\ci}) = 1$ if there exists $w \in V(S) \setminus \{u\}$ such that $u \to v \from w$ and is 0 otherwise.
\end{proof}

\begin{lemma}
\label{lemma:bidirectededgeinequality}
Let $S$ be a star tree and let $I$ be a subset of the leaves of $S$. For any internal edge of $S^I$, the associated bidirected-edge inequality defines a facet of $\cim_S^I$.
\end{lemma}

\begin{proof}
Let $\cg^\ci$ be an $\ci$-DAG with set of leaf targets $I$, and let $u - v$ be an internal edge of $S^I$. 
We first show the validity of the bidirected-edge inequality. 
By Proposition~\ref{proposition:atLeastOneColliderIndicator} we know that 
\begin{equation*}
    \begin{split}
        \ind_{v \from u}(c_{\cg^\ci}) &= \begin{cases} 1 &\mbox{if}~ v \from u~ \mbox{is in a v-structure of}~\cg^\ci \\ 0 &\mbox{otherwise,} \end{cases} \\ 
        \ind_{u \from v}(c_{\cg^\ci}) &= \begin{cases} 1 &\mbox{if}~ u \from v~ \mbox{is in a v-structure of}~\cg^\ci \\ 0 &\mbox{otherwise.} \end{cases}
    \end{split}
\end{equation*}
The edge $u - v$ can only be oriented one direction, and so only one of these indicator functions can evaluate to 1. 
Hence
\[
\ind_{v \from u}(c_{\cg^\ci}) + \ind_{u \from v}(c_{\cg^\ci}) \leq 1,
\]
and so bidirected-edge inequalities are valid on $\cim_S^I$.

To show that bidirected-edge inequalities are facet-defining, we follow the strategy from Lemma~\ref{lemma:starInequalities} and show that there exists a unique vertex which does not lie on the face defined by the bidirected-edge inequality. Suppose that 
\[
\ind_{v \from u}(c_{\cg^\ci}) + \ind_{u \from v}(c_{\cg^\ci}) = 0,
\]
so that we know $\ind_{v \from u}(c_{\cg^\ci}) = \ind_{u \from v}(c_{\cg^\ci}) = 0$. 
Without loss of generality, suppose that $u$ is a leaf of $S$ and $v$ is the center. 
Since $\ind_{u \from v}(c_{\cg^\ci}) = 0$ and $u' \to u$ is essential in ${\cg^\ci}$, we know $v \from u$. 
However, $\ind_{v \from u}(c_{\cg^\ci}) = 0$, and so for each leaf $w \neq u$ of $S$, $v \to w$ in ${\cg^\ci}$. 
Thus, we constructed a unique DAG (and hence $\ci$-Markov equivalence class) whose characteristic imset vector does not lie on the face defined by the bidirected-edge inequality. This implies that the bidirected-edge inequality defines a facet.
\end{proof}

\subsection{The Forked-Tree Inequality}
\label{subsection:theForkedTreeInequality}

\begin{definition}
\label{def:starindicator}
Let $c$ be the center of $S$ and let $L$ be a subset of the leaves in $S$. We define a linear functional
\[\ind_{c \to L}(x) = \sum\limits_{\substack{c \in A \subseteq V(S) \setminus L \\ |A| \geq 3}} \sum\limits_{B \subseteq L} (-1)^{|A \cup B| - 1} \left( |A| -2 \right) x_{A \cup B}.\]
\end{definition}


\begin{notation}\label{not:directionalitymatterstoo}
We note the distinction between $\ind_{c \from u}(x)$ and $\ind_{c \to L}(x)$. 
Specifically, $\ind_{c \from u}(x)$ is equal to $1$ if and only if the edge $u\to c$ is an edge in some v-structure whereas $\ind_{c \to L}(x)$ is equal to $1$ if and only if there exists a v-structure with center node $c$ that forces the edges $c\to\ell$ to be essential for all $\ell \in L$. In particular, $\ind_{c \to L}(x)$ is an indicator function for the latter condition. 
The left node in the subscript edge is always the center $c$ and the forcing of edge orientations is caused by a v-structure at $c$, which we saw in Proposition~\ref{proposition:atLeastOneColliderIndicator} and will see in Proposition~\ref{proposition:edgesForcedOut}.
\end{notation}

We first consider the case of $L = \emptyset$.

\begin{proposition}
\label{proposition:vstructureIndicator}
Let $S$ be a star graph with center node $c$. Let $\cg$ be a DAG with skeleton $S$ and let $X = \pa_{\cg}(c)$. 
We then have 
\[\ind_{c \to \emptyset}( c_\cg ) = \begin{cases} 1 &\mbox{if}~ |X| \geq 2 \\ 0 &\mbox{otherwise.}\end{cases}\]
\end{proposition}

\begin{proof}
In the case where $L = \emptyset$, we sum over all $B \subseteq \emptyset$, and so the sum simplifies to
\[\ind_{c \to \emptyset}(x) = \sum\limits_{\substack{c \in A \subseteq V(S) \setminus L \\ |A| \geq 3}} (-1)^{|A| - 1} \left( |A| -2 \right) x_A.\]
Then $c_\cg(A) = 1$ if and only if $c \in A \subseteq X \cup \{c\}$, and so $\ind_{c \to \emptyset}^S(c_\cg)$ simplifies to
\[\sum\limits_{\substack{c \in A \subseteq X \cup \{c\} \\ |A| \geq 3}} (-1)^{|A|-1} (|A|-2).\]
We add and subtract all terms $(-1)^{|A|-1} (|A|-2)$ associated to $A$ of size 1 or 2 which yields
\[1+\sum\limits_{c \in A \subseteq X \cup \{c\}} (-1)^{|A|-1} (|A|-2).\]
We reindex this sum via the substitution $B = A \setminus \{c\}$:
\begin{align*}
    1+\sum\limits_{B \subseteq X} (-1)^{|B|} (|B|-1) &= 1+ \sum\limits_{B\subseteq X} (-1)^{|B|} |B| + \sum\limits_{B \subseteq X} (-1)^{|B|-1}, \\
    &= 1 - \ind_{|X| = 1} - \ind_{|X| = 0}, \\
    &= \ind_{|X| \geq 2}.
\end{align*}
where $\ind_{|X| = 1}$, $\ind_{|X| = 0}$, and $\ind_{|X| \geq 2}$ represent the indicator functions for their respective events.
\end{proof}

\begin{proposition}
\label{proposition:edgesForcedOut}
Let $S$ be a star graph with center node $c$ and let $L$ be a subset of the leaves in $S$. 
Let $\cg$ be a DAG with skeleton $S$ and let $X = \pa_{\cg}(c)$. 
We then have
\[
\ind_{c \to L}( c_\cg ) = \begin{cases} 1 &\mbox{if}~ |X| \geq 2 ~\mbox{and}~ X \cap L = \emptyset \\ 0 &\mbox{otherwise.}\end{cases}.
\]
\end{proposition}

\begin{proof}
For each $A \subseteq V(S) \setminus L$ that contains $c$ and has $|A| \geq 3$ and for each $B \subseteq L$, we have $c_\cg(A \cup B) = 1$ if and only if $A \cup B \subseteq X \cup \{c\}$. This fact coupled with the fact that $A$ and $B$ are disjoint implies that $\ind_{c \to L}^S(c_\cg)$ may be rewritten as
\begin{align*}
    \ind_{c \to L}(c_\cg) &= \sum\limits_{\substack{c \in A \subseteq (X \cup \{c\}) \setminus L \\ |A| \geq 3}} \sum\limits_{B \subseteq X \cap L} (-1)^{|A \cup B| - 1} \left( |A| -2 \right), \\
    &= \left( \sum\limits_{\substack{c \in A \subseteq (X \cup \{c\}) \setminus L \\ |A| \geq 3}}  (-1)^{|A|-1}(|A|-2) \right) \cdot \left(\sum\limits_{B \subseteq X \cap L} (-1)^{|B|} \right).
\end{align*}
If $S'$ is the subgraph of $S$ on $X \cup \{c\} \setminus L$, then Proposition~\ref{proposition:vstructureIndicator} tells us that
\[\ind_{c \to L}^S(c_\cg) = \ind_{c \to \emptyset}^{S'}(c_{S'}) \cdot \ind_{X \cap L = \emptyset} = \ind_{|X| \setminus L \geq 2} \cdot \ind_{X \cap L = \emptyset} = \ind_{|X| \geq 2} \cdot \ind_{X \cap L = \emptyset}.\]
\end{proof}

\begin{definition}
\label{def:starhash}
Let $S$ be a star with center $c$, and define
\[\#^S(x) = \sum\limits_{\substack{c \in A \subseteq V(S) \\ |A| \geq 3}} (-1)^{|A|-1} x_A.\] 
\end{definition}

\begin{proposition}
\label{prop:numPointCenterward}
Let $S$ be a star graph with center $c$ and let $\cg$ be a DAG with skeleton $S$. Let $X = \pa_{\cg}(c)$. Then
\[\#^S\left( c_\cg \right) = \begin{cases}
    |X|-1 &\mbox{if}~|X| \geq 1 \\ 0 &\mbox{otherwise.} \end{cases}\]
\end{proposition}

\begin{proof}
We know that $c_\cg(A) = 1$ if and only if $c\in A \subseteq X \cup \{c\}$. Then $\#^S(c_\cg)$ simplifies to
\[\sum\limits_{\substack{c \in A \subseteq X \cup \{c\} \\ |A| \geq 3}} (-1)^{|A|-1}.\]
If $X = \emptyset$, then the sum is empty and $\#^S(c_\cg) = 0$. Otherwise we know that $\sum\limits_{c \in A \subseteq X \cup \{c\}} (-1)^{|A|-1} = 0$, and so we have
\[\#^S(c_\cg) = \sum\limits_{\substack{c \in A \subseteq X \cup \{c\} \\ |A| \leq 2}} (-1)^{|A|} = |X|-1.\]
\end{proof}

\begin{definition}
\label{definition:forkedTreeStar}
Let $S$ be a star graph with center $c$ and let $I$ be a subset of the leaves of $S$. If $S'$ is the substar of $S$ with leaves in $I$, then the \emph{forked-tree inequality} is
\[\ind_{c \to N_{S'}(c)}(x) + \sum\limits_{u \in I} (1 - x_{cuu'}) \leq 1 + \#^{S'}(x).\]
\end{definition}

In Section~\ref{section:eliminatingInterventions} we show that Definition~\ref{definition:forkedTreeStar} generalizes to a family of subtrees which in the case without interventions have forking behavior at leaf nodes.

\begin{example}
Once again consider the star tree with leaf interventions in Figure~\ref{fig:starInterventions}. We again recall that the inequalities defining $\cim_S^I$ are
\begin{align*}
x_{abc} - x_{abcd} \geq 0 \quad\quad&\left(x_{aa' c} \right) + \left( x_{abc} + x_{acd} - x_{abcd} \right) \leq 1 \\
x_{bcd} - x_{abcd} \geq 0 \quad\quad&\left(x_{cdd'} \right) + \left( x_{acd} + x_{bcd} - x_{abcd} \right) \leq 1 \\
x_{acd} - x_{abcd} \geq 0 \quad\quad&\left( 1 - x_{aa'c} \right) + \left(1 - x_{cdd'} \right) \leq 1 + x_{acd}. \\
x_{abcd} \geq 0 \quad\quad
\end{align*}
The four inequalities on the left are all star inequalities while the top two inequalities on the right are the bidirected-edge inequalities associated to the internal edges $a-c$ and $d-c$ respectively. The remaining inequality
\[
\left( 1 - x_{aa'c} \right) + \left(1 - x_{cdd'} \right) \leq 1 + x_{acd}
\]
is the forked-tree inequality associated to this star. Observe that since $I = \{a,d\}$, the substar $S'$ consists of all of $S$ except the node $b$. The first piece of the forked-tree inequality is $\ind_{c \to N_{S'}(c)}$ however we see that  
\[
\ind_{c \to N_{S'}(c)} = \sum\limits_{\substack{c \in A \subseteq \{a,b,c,d\} \setminus \{a,d\} \\ |A| \geq 3}} \sum\limits_{B \subseteq L} (-1)^{|A \cup B| - 1} \left( |A| -2 \right) x_{A \cup B},
\]
and thus no such $A$ in the first sum index can exist. 
Therefore, this sum is 0. 
So the left-hand side of the forked-tree inequality is just
\[
\ind_{c \to N_{S'}(c)} + \sum\limits_{u \in I} (1 - x_{cuu'}) = 0 + \sum\limits_{u \in {a,d}} (1 - x_{cuu'}) = (1-x_{aa'c}) + (1-x_{cdd'}).
\]
From the definition, the right-hand side of the forked-tree inequality associated to this star is
\[
1 + \#^{S'}(x) = 1 + \sum\limits_{\substack{c \in A \subseteq V(S') \\ |A| \geq 3}} (-1)^{|A|-1} x_A = 1 + \sum\limits_{\substack{c \in A \subseteq \{a,c,d\} \\ |A| \geq 3}} (-1)^{|A|-1} x_A = 1 + x_{acd}.
\]
Thus we see that our remaining inequality which defines $\cim_{S}^I$ is indeed the forked-tree inequality associated to this star. 
\end{example}

\begin{lemma}
Let $S$ be a star graph with center $c$ and let $I$ be a subset of the leaves of $S$. Then the forked-tree inequality defines a facet of $\cim_S^I$.
\end{lemma}

\begin{proof}
Let $S'$ be the substar of $S$ with leaves in $I$.
We first rearrange the forked-tree inequality to obtain the following:
\[|I|+ \ind_{c \to I}(x) \leq 1 + \sum\limits_{u \in I} x_{cuu'} +  \#^{S'}(x).\]
Let $\cg^\ci$ be an $\ci$-DAG with set of leaf targets $I$. 
Either there exists $u \in I$ such that $u \to c$ or no such $u$ exists.

\noindent \textbf{Case 1:}
If there exists $u \in I$ such that $u \to c$, then by Proposition~\ref{prop:numPointCenterward} we know $\#^{S'}(c_{\cg^\ci}) = |\{u \in V(S'): u \to c\}| = |\{u \in I: u \to c\}|$. 
We have
\[
\sum\limits_{u \in I} x_{cuu'} + \#^{S'}(c_{\cg^\ci}) = |\{u \in I: c \to u\} + |\{u \in I: u \to c\}| - 1 = |I| - 1.
\]
The forked-tree inequality then simplifies to
\[
\ind_{c \to I}(c_{\cg^\ci}) \leq 0.
\]
But $\ind_{c \to I}(c_{\cg^\ci}) = 0$ by Proposition~\ref{proposition:edgesForcedOut} since an element of $I$ is in the parents of $c$, and so equality holds on the forked-tree inequality and the forked-tree inequality is valid on $c_{\cg^\ci}$.

\noindent \textbf{Case 2:} Otherwise every $u \in I$ has $u \from c$, so $\sum\limits_{u \in I} 1-x_{cuu'} = 0$ and also $\#^{S'}(c_{\cg^\ci}) = 0$ (by Proposition~\ref{prop:numPointCenterward}). The forked-tree inequality is equivalent to
$
\ind_{c \to I}(c_{\cg^\ci}) \leq 1
$
which is valid since $\ind_{c \to I}(x)$ is an indicator function. 
However, the only way that this inequality can be strict is if there is no v-structure among the leaves of $S$ not in $I$. 

Combining these two cases, we see that the forked-tree inequality is valid on $\cim_S^I$ since every vertex satisfies the forked-tree inequality. 
Moreover, there exists a unique $\ci$-Markov equivalence class which does not achieve equality, and so the forked-tree inequality defines a facet of $\cim_S^I$.
\end{proof}

\begin{theorem}
\label{thm:FacetsOfStarTrees}
Let $S$ be a star graph and let $I$ be a subset of the leaves of $S$. Star inequalities, bidirected-edge inequalities, and the forked-tree inequality form a minimal H-representation for $\cim_S^I$. Moreover, $\cim_S^I$ is a simplex.
\end{theorem}

\begin{proof}
We show that the number of facets and the number of vertices of $\cim_S^I$ are the same. Let $n$ be the number of leaves of $S$ and let $k = |I|$. Each star inequality is associated to a subset of the leaves of $S$ of size at least 2. Hence there are $2^n -n -1$ star inequalities. Bidirected-edge inequalities are associated to elements of $I$, and so there are $k$ such inequalities. There is exactly 1 forked-tree inequality, and so we obtain $2^n - n + k$ facets in total. These inequalities are all distinct since no two functionals have the same support.

We now count the $\ci$-Markov equivalence classes. 
Given an $\ci$-DAG $\cg^\ci$ with set of targets $I$, if every $u \in I$ has $u \from c$, then the number of Markov equivalence class of this type is the same as the number of non-interventional Markov equivalence classes on the substar of $S$ obtained by deleting the leaves of $I$. This yields $2^{n-k} - n +k$ Markov equivalence classes.

Otherwise we choose some nonempty subset of $I$ to point toward $c$. This choice can be made in $2^k - 1$ ways. Regardless of choice, every leaf $v \not\in I$ has an orientation $c \to v$ or $c \from v$ of the edge $c -v$ which is essential, and so we have $2^{n-k}$ ways to choose such orientations. This yields $2^{n-k}(2^k - 1) = 2^n - 2^{n-k}$ Markov equivalence classes of this type.

Summing the two cases together, we see that there are $2^n - n +k$ interventional Markov equivalence classes, which we previously saw was also the number of facets we have computed.
\end{proof}

\section{Interventions on Trees}
\label{section:interventionsontrees}
In this section we compute an H-representation for a polytope we denote $\cim_T^{I,J}$ obtained by gluing star trees together at a node targeted by an intervention. The main tool used in this section is \cite[Lemma 3.2]{dinu2020gorenstein}, a result regarding the H-representation of the toric fiber product, which we review. In Section~\ref{section:eliminatingInterventions} we show that $\cim_T^{I,J}$ and some $\cim_X^I$ are related by coordinate projection and use this to compute an H-representation for $\cim_X^I$. We begin with a special collection of trees and an associated polytope.

\begin{definition}
\label{defn:GluingTree}
A \emph{gluing tree} $(T,I,J)$ is a tree $T$ with vertices colored black and white such that every white node has degree at most two and and no two white nodes adjacent. We partition the white nodes into $I \sqcup J$ where $I$ is the set of leaf nodes and $J$ is the set of degree two nodes.
\end{definition}

\begin{definition}
An edge $ij$ is \emph{monochromatic} if $i,j\in I$ or $i,j\in J$.
A gluing tree $(T,I,J)$ is \emph{separated} if the only monochromatic edges are leaves.
\end{definition}

Note that every star graph with black center vertex is a separated gluing tree. 
Furthermore, any separated gluing tree can be obtained via iteratively gluing black centered stars together at white nodes.
The goal of this section is to compute an H-representation of the polytope $\CIM_T^{I,J}$, which we now define, in the separated case.

\begin{definition}
Let $(T,I,J)$ be a gluing tree. We define $\cim_T^{I,J}$ to be
\[\conv\left( c_\ct \in \cim_T^{I \cup J}:~\mbox{for all}~j \in J~\mbox{with}~N_T(j) = \{i,k\}~\mbox{we have}~i \to j \to k~\mbox{or}~ i \from j \from k ~\mbox{in}~\ct\right).\]
\end{definition}

\begin{example}
\label{example:facets}
Consider the undirected tree $T$ in Figure~\ref{fig:facetsExample} with intervention at node 8. The star inequalities associated to the star induced on $N_T[3]$, $N_T[4]$, and $N_T[5]$ are facet-defining. These inequalities are:
\begin{alignat*}{3}
    x_{123} - x_{1234} &\geq 0 \quad\quad  & x_{345} - x_{3458} &\geq 0 \quad\quad & x_{456} - x_{4567} &\geq 0 \\
    x_{134} - x_{1234} &\geq 0 \quad\quad  & x_{348} - x_{3458} &\geq 0 \quad\quad & x_{457} - x_{4567} &\geq 0 \\
    x_{234} - x_{1234} &\geq 0 \quad\quad  & x_{458} - x_{3458} &\geq 0 \quad\quad & x_{567} - x_{4567} &\geq 0 \\
              x_{1234} &\geq 0 \quad\quad  &           x_{3458} &\geq 0 \quad\quad &           x_{4567} &\geq 0.
\end{alignat*}
We have bidirected-edge inequalities associated to each of the internal edges of $T^I$ given by 
\begin{align*}
    (x_{134} + x_{234} - x_{1234}) + (x_{345} + x_{348} - x_{3458}) \leq 1 \\
    (x_{345} + x_{458} - x_{3458}) + (x_{456} + x_{457} - x_{4567}) \leq 1 \\
    (x_{488'}) + (x_{348} + x_{458} - x_{3458}) \leq 1 
\end{align*}
The remaining inequalities are all generalizations of the forked-tree inequality in Definition~\ref{definition:forkedTreeStar} corresponding to subtrees induced on $\{3\}, \{5\}, \{4,8\}$, and $\{3,4,5,8\}$ respectively:
\begin{align*}
    x_{123} + x_{134} + x_{234} - 2x_{1234} \leq 1 \\
    x_{456} + x_{457} + x_{567} - 2x_{4567} \leq 1 \\
    x_{345} + (1-x_{488'}) \leq 1 \\
    (x_{123} + x_{134} + x_{234} - 2x_{1234}) + (x_{456} + x_{457} + x_{567} - 2x_{4567}) + (1-x_{488'}) \\ \leq 1 
    + (x_{345} + x_{348} + x_{458} - x_{3458}).
\end{align*}

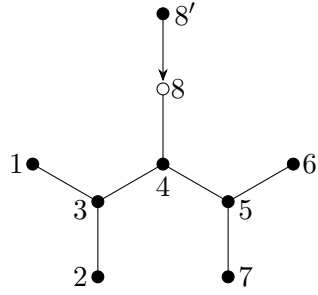
\begin{figure}
    \centering
\begin{tikzpicture}
\begin{scope}
\draw [rotate=120] (0,0)--(0,1);
\draw [rotate=120] (0,1)--(0.866,1.5);
\draw [rotate=120] (0,1)--(-0.866,1.5);

\node[wB] at (-1.732,0) {};
\node[wB] at (-0.866,-0.5) {};
\node[wB] at (-0.866,-1.5) {};

\node at (-1.95,0) {1};
\node at (-1.1,-0.6) {3};
\node at (-1.1,-1.5) {2};
\end{scope}

\begin{scope}
\draw [rotate=240] (0,0)--(0,1);
\draw [rotate=240] (0,1)--(0.866,1.5);
\draw [rotate=240] (0,1)--(-0.866,1.5);

\node[wB] at (1.732,0) {};
\node[wB] at (0.866,-0.5) {};
\node[wB] at (0.866,-1.5) {};

\node at (1.1,-0.6) {5};
\node at (1.95,0) {6};
\node at (1.1,-1.5) {7};
\end{scope}

\draw (0,0)--(0,1);
\draw [-Stealth] (0,2)--(0,1.1);

\node[wB] at (0,0) {};
\node[wW] at (0,1) {};
\node[wB] at (0,2) {};

\node at (0,-0.3) {4};
\node at (0.2,1) {8};
\node at (0.3,2) {$8'$};
\end{tikzpicture}
    \caption{An undirected tree $T$ with an intervention at 8.}
    \label{fig:facetsExample}
\end{figure}
\end{example}

By construction $\cim_T^{I,J} 
\subseteq \cim_T^{I \cup J}$. Moreover, $\dim(\cim_T^{I,J}) < \dim(\cim_T^{I \cup J})$ because for all $x \in \cim_T^{I,J}$ we have $x_{ijk} = x_{ijj'k} = 0$ and $x_{ijj'} + x_{jj'k} = 1$. 
We may consequently ignore the $ijk-$ and $ijj'k$-coordinates and as an abuse of notation, equate polytopes whose points differ only by omission of $x_{ijk}$ and $x_{ijj'k}$. This allows us to build $\cim_T^{I,J}$ via a gluing.

\begin{definition}
\label{defn:interventionalParting}
Let $(T, I, J)$ be a gluing tree. Let $j \in J$ with $N_T(j) = \{i,k\}$. The \emph{(undirected) interventional parting} of $T$ at $j$ is the pair of gluing trees $\left(T_j^i, (I \cup \{j\}) \cap V(T_j^i) 
, (J \setminus \{j\}) \cap V(T_j^i)\right)$ and $\left(T_j^k, (I \cup \{j\}) \cap V(T_j^k) 
, (J \setminus \{j\}) \cap V(T_j^k)\right)$ where
\begin{enumerate}
\item $T_j^i$ the connected component of $T\setminus{j-k}$ which contains $i$,
\item $T_j^k$ be the connected component of $T\setminus{i-j}$ which contains $k$.    
\end{enumerate}
For notational convenience we denote the gluing trees in the undirected parting by $(T^i, I^i, J^i)$ and $(T^k, I^k, J^k)$. Let $\ct$ be a DAG with skeleton $T$. The \emph{directed interventional parting of $\ct$} at $j$ is the pair of directed subgraphs induced on $T^i$ and $T^k$. 
This construction is pictured in Figure \ref{fig:interventionalParting}.
\end{definition}

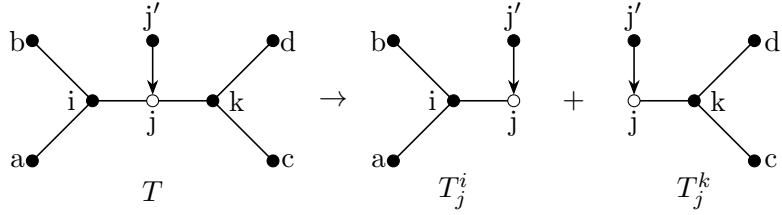
\begin{figure}
\centering
\begin{tikzpicture}[scale = .8]

\begin{scope}[shift={(0,0)}]
\draw [semithick] (0,0)--(0.94,0.94);
\draw [semithick] (0,2)--(1,1);
\draw [semithick] (2,1)--(1.08,1);
\draw [semithick] (3,1)--(2,1);
\draw [semithick] (3,1)--(4,0);
\draw [semithick] (3,1)--(4,2);
\draw [-Stealth, semithick] (2,2)--(2,1.1
);

\node at (-0.25,0) {a};
\node at (-0.25,2) {b};
\node at (0.65,1) {i};
\node at (3.4,1) {k};
\node at (4.25,0) {c};
\node at (4.25,2) {d};
\node[below] at (2,1) {j};
\node[above] at (2,2) {j$'$};

\node[wB] at (0,0) {};
\node[wB] at (0,2) {};
\node[wB] at (1,1) {};
\node[wW] at (2,1) {};
\node[wB] at (3,1) {};
\node[wB] at (4,0) {};
\node[wB] at (4,2) {};
\node[wB] at (2,2) {};

\node at (2, -.5) {$T$};
\end{scope}

\node at (5,1) {$\to$};

\begin{scope}[shift={(6,0)}]
\draw [semithick] (0,0)--(0.94,0.94);
\draw [semithick] (0,2)--(1,1);
\draw [semithick] (2,1)--(1.08,1);
\draw [semithick] (4,1)--(5,1);
\draw [semithick] (5,1)--(6,0);
\draw [semithick] (5,1)--(6,2);
\draw [-Stealth, semithick] (2,2)--(2,1.1);
\draw [-Stealth, semithick] (4,2)--(4,1.1);

\node at (-0.25,0) {a};
\node at (-0.25,2) {b};
\node at (0.65,1) {i};
\node[below] at (2,1) {j};
\node[above] at (2,2) {j$'$};

\node[below] at (4,1) {j};
\node[right] at (6,0) {c};
\node[right] at (6,2) {d};
\node[above] at (4,2) {j$'$};
\node [right] at (5.1,1) {k};

\node[wB] at (0,0) {};
\node[wB] at (0,2) {};
\node[wB] at (1,1) {};
\node[wW] at (2,1) {};
\node[wB] at (2,2) {};

\node[wB] at (5,1) {};
\node[wW] at (4,1) {};
\node[wB] at (6,0) {};
\node[wB] at (6,2) {};
\node[wB] at (4,2) {};

\node at (1, -.5) {$T_j^i$};
\node at (3, 1) {$+$};
\node at (5, -.5) {$T_j^k$};
\end{scope}
\end{tikzpicture}

\caption{A separated gluing tree $T$ (with interventions added in) is pictured on the left and the interventional parting at $j$ is on the right. Each nonzero coordinate of $\cim_T^{I,J}$ appears in exactly one of the two trees of the interventional parting.} 
\label{fig:interventionalParting}
\end{figure}

Note that the parting of $(T,I,J)$ at $j$ is a pair of gluing trees. Indeed, splitting $(T,I,J)$ into $(T^i,I^i,J^i)$ and $(T^k,I^k,J^k)$ yields two graphs in which the degree of each white node is preserved except at $j$, where the degree decreases by one. This moves $j$ from the set $J$ to the set $I$.

For example, consider the gluing tree $T$ and its associated parting at $j$  pictured in Figure \ref{fig:interventionalParting}. In the language of Definition \ref{defn:GluingTree} which is $(T, \emptyset, \{ j\})$. The resulting interventional parting of $T$ at $j$ yields the pair of gluing trees $(T_j^i, \{j\}, \emptyset)$ and $(T_j^k, \{j\}, \emptyset)$ which are also gluing trees. Observe that in the original gluing tree $T$, we had that 
$J = \{j\}$ however in both of the resulting gluing trees we have that $I = \{j\}$ and $J = \emptyset$. In other words, the interventional node $j$ was internal in the gluing tree $T$ but has become a leaf in the gluing trees $T_j^i$ and $T_j^k$. We now examine the facets of the gluing tree $T$ in comparison to those of $T_j^i$ and $T_k^i$.

\begin{example}
\label{ex:TFPofTreesIneqs}
Let $T$ be the gluing tree pictured in Figure \ref{fig:interventionalParting} and $T_j^i, T_j^k$ be the interventional parting at node $j$. Since both of these trees are interventional star trees, we know by Theorem \ref{thm:FacetsOfStarTrees} that the facets of their associated interventional CIM polytope are

\begin{tabular}{c| c c}
&                       $\cim_{T_j^i}^{\{j\}}$                  &           $\cim_{T_j^k}^{\{j\}}$  \\
\hline \\
                & $x_{abi} - x_{abij} \geq 0$  &  $x_{cdk} - x_{cdjk} \geq 0$\\
star ineqs.     & $x_{aij} - x_{abij} \geq 0$  &  $x_{cjk} - x_{cdjk} \geq 0$\\ 
                & $x_{bij} - x_{abij} \geq 0$  &  $x_{djk} - x_{cdjk} \geq 0$\\ 
                & $x_{abij}           \geq 0$  &  $          x_{cdjk} \geq 0$\\ \\ 
\hline
\\
bidirected-edge ineqs.& $(x_{ijj'}) + (x_{aij} + x_{bij} - x_{abij})$$  \leq 1$  &  $(x_{jj'k}) + (x_{cjk} + x_{djk} - x_{cdjk}) \leq 1$\\ \\

\hline

\\
forked-tree ineqs. & $(x_{abi} - x_{abij}) + (1-x_{ijj'}) \leq 1$   & $(x_{cdk} - x_{cdjk}) + (1-x_{jj'k}) \leq 1$\\ \\ 
\end{tabular}

Now observe that the only non-constant coordinates of the following three polytopes are
\begin{align*}
&\cim_{T_j^i}^{\{j\}}: x_{abi}, x_{aij}, x_{bij}, x_{ijj'}, x_{abij},\\ 
&\cim_{T_j^k}^{\{j\}}: x_{cdk}, x_{cjk}, x_{djk}, x_{jj'k}, x_{cdjk}, \\ 
&\cim_{T}^{\emptyset, \{j\}}: x_{abi}, x_{aij}, x_{bij}, x_{ijj'}, x_{abij}, x_{cdk}, x_{cjk}, x_{djk}, x_{jj'k}, x_{cdjk}
\end{align*}
because the coordinate $x_{ijj'k} = 0$ for all $x \in \cim_{T}^{\emptyset, \{j\}}$ by construction. Furthermore, we can naturally consider $\cim_{T}^{\emptyset, \{j\}}$ as a subset of the product polytope $\cim_{T_j^i}^{\{j\}} \times \cim_{T_j^k}^{\{j\}}$ since for any interventional DAG $\ct$ which corresponds to a vertex of $\cim_{T}^{\emptyset, \{j\}}$, it holds that
\[
c_\ct = (c_{\ct_j^i}, c_{\ct_j^k})
\]
where $\ct_j^i, \ct_j^k$ is the directed interventional parting of $\ct$ at $j$. Direct computation of the facets of $T$ shows that $\cim_{T}^{\emptyset, \{j\}} = \cim_{T_j^i}^{\{j\}} \times \cim_{T_j^k}^{\{j\}} \cap \{x \in  \rr^{10}  ~|~ x_{ijj'} + x_{jjk'} = 1\}$. That is, $\cim_{T}^{\emptyset, \{j\}}$ can be obtained by slicing the product polytope with the hyperplane $x_{ijj'} + x_{jjk'} = 1$. 
\end{example}

The phenomenon that we observed in the previous example occurs more generally. We prove this using the toric fiber product, which we now review. 

\begin{definition}
\label{defn:TFP}
Let $P_1$ and $P_2$ be lattice polytopes and $\pi_1: P_1 \to Q$ and $\pi_2: P_2 \to Q$ be integral projections. Then the \emph{toric fiber product} of these polytopes with respect to $\pi_1$ and $\pi_2$ is
\[
P_1 \times_Q P_2 = \conv \left(\{ (v_i, w_j) \in V(P_1) \times V(P_2) ~|~ \pi_1(v_i) = \pi_2(w_j) \}\right)
\]
\end{definition}

The toric fiber product was first introduced in \cite{sullivant2007toric}, and it has been frequently used to study families of toric ideals and polytopes which arise in algebraic statistics \cite{dinu2020gorenstein, engstrom2014multigraded,johnson2023codegree, rauh2016lifting}. A generalization of this operation, called a \emph{quasi-independence gluing} (QIG), was developed in \cite{hollering2022toric} and used to characterize the toric ideal associated to the polytope $\cim_{T}$ which we now study.  
We now show that our polytope $\cim_{T}^{I,J}$ can be built by taking toric fiber products of interventional CIM polytopes of smaller trees as we saw in the previous example. 

\begin{lemma}
Let $(T,I,J)$ be a gluing tree, let $j \in J$, and let $N_T(j) = \{i,k\}$. Define linear maps $\pi_i:\cim_{T^i}^{I^i,J^i} \to [0,1]$ and $\pi_k:\cim_{T^k}^{I^k,J^k} \to [0,1]$ by
\[\pi_i(x) = x_{ijj'} \quad~\mbox{and}~\quad \pi_k(x) = 1 - x_{jj'k}.\]
Then
\[\cim_T^{I,J} = \cim_{T^i}^{I^i,J^i} \times_{[0,1]} \cim_{T^k}^{I^k,J^k}.\]
\end{lemma}

\begin{proof}
By construction $\cim_T^{I,J} \subseteq \cim_T^{I \sqcup J}$, but because in each DAG yielding an imset vector of $\cim_T^{I,J}$ we have either $i \to j \to k$ or $i \from j \from k$, we know that $x_{ijk} = x_{ijj'k} = 0$. Each nonzero coordinate of $\cim_T^{I,J}$ is consequently a coordinate of either $\cim_{T^i}^{I^i,J^i}$ or $\cim_{T^k}^{I^k,J^k}$. Hence we may realize $\cim_T^{I,J}$ as a subpolytope of $\cim_{T^i}^{I^i,J^i} \times \cim_{T^k}^{I^k,J^k}$.

Since we have the edge $j' \to j$, the edges between $i,j$ and $j,k$ are both essential.
In the toric fiber product, we glue two characteristic imsets together if their coordinates satisfy $x_{ijj'} + x_{jj'k} = 1$, which happens if and only if $i \to j \to k$ or $i \from j \from k$ in every pair of DAGs in their corresponding Markov equivalence classes.
%
\end{proof}

The toric fiber product structure is useful due to the following lemma from \cite{dinu2020gorenstein}.

\begin{lemma}\cite[Lemma 3.2]{dinu2020gorenstein}
\label{lemma:TFPHRep}
Let $P_1$ and $P_2$ be polytopes and let $Q$ be a simplex. Let $\pi_i: P_i \to Q$ be an affine linear map such that $\pi_1(P_1) = \pi_2(P_2) = Q$. Then the toric fiber product is given by
\[P_1 \times_Q P_2 = \{(x,y): P_1 \times P_2: \pi_1(x) = \pi_2(y)\}.\] 
\end{lemma}

In Definition~\ref{defn:TFP} the toric fiber product is defined by a V-representation. If the H-representations of $P_1$ and $P_2$ are known, Lemma~\ref{lemma:TFPHRep} allows us to also compute an H-representation of the toric fiber product. This lemma immediately yields the following corollary.

\begin{corollary}
\label{corollary:hyperplane}
Let $\cim_{T^i}^{I^i,J^i} \times \cim_{T^k}^{I^k,J^k}$ denote the product polytope. Then $\cim_T^{I,J}$ is the intersection of $\cim_{T^i}^{I^i,J^i} \times \cim_{T^k}^{I^k,J^k}$ with the hyperplane $x_{ijj'} + x_{jj'k} = 1$.
\end{corollary}

Hence the inequalities defining $\cim_T^{I,J}$ are those defining the product. Given some inequality $a \cdot x  \leq b$ in the H-representation of $\cim_{T^i}^{I^i, J^i}$, the inequality $(a,0)\cdot (x,y) \leq b$ lies in the H-representation of the product polytope (and $\cim_T^{I,J}$). 
Hence the inequalities are exactly what we saw in Example \ref{ex:TFPofTreesIneqs}. 
In the following special case, we immediately obtain an H-representation.

\begin{theorem}
\label{theorem:facetsseparatedgluingtree}
Let $(T,I,J)$ be a separated gluing tree. The facets of $\cim_T^{I,J}$ are the star, bidirected edge, and forked-tree inequalities associated to each of the black centered stars in $(T,I,J)$.
\end{theorem}

\begin{proof}
Recall that $(T,I,J)$ is separated if the only monochromatic edges are leaves. Such a gluing tree can be decomposed via interventional partings into black centered stars. Hence the facets of $\cim_T^{I,J}$ are just those of the black centered star subgraphs.
\end{proof}

\section{Eliminating Interventions}
\label{section:eliminatingInterventions}

Throughout this section we fix a gluing tree $(T,I,J)$ and a node $j \in J$ with $N_T(j) = \{i,k\}$. We let $(X,I, J \setminus \{j\})$ be the gluing tree obtained by replacing $i - j - k$ with just $i - k$. The main goal of this section is to eliminate the non-leaf interventions (at elements in $J$), which allows us to construct an H-representation for $\cim_X^{I, J \setminus \{j\}}$ from an H-representation for $\cim_T^{I,J}$. In the extreme cases we obtain $\cim_T^{I, \emptyset} = \cim_T^I$ and $\cim_T^{\emptyset, \emptyset} = \cim_T$. 

The section is organized as follows. In Section~\ref{subsection:projection} we show that up to relabeling $\cim_T^{I,J}$ projects to $\cim_X^{I, J \setminus \{j\}}$. Then in Section~\ref{subsection:facetCIMTij} we state the main result of this section (Theorem~\ref{theorem:facetsOfCIMTIJ}), which is an H-representation of $\cim_T^{I,J}$ that generalizes Theorem~\ref{theorem:facetsseparatedgluingtree} beyond the case of separated gluing trees. We obtain this H-representation in Section~\ref{subsection:FMElimination} by performing iterated Fourier-Motzkin elimination.

\subsection{A Coordinate Projection on Imsets}
\label{subsection:projection}

We begin with the following relabeling. For each internal node $\ell \neq j$ and each vertex set $A$ of a star subgraph of $N_T[\ell]$, we relabel the coordinates $x_A$ of $\cim_T^{I,J}$ as follows:

\begin{enumerate}
    \item If $j \not\in A$, then $x_A \mapsto x_A$.
    \item If $\ell = i$ and $j \in A$, then $x_A \mapsto x_{(A \cup \{k\}) \setminus \{j\}}$.
    \item If $\ell = k$ and $j \in A$, then $x_A \mapsto x_{(A \cup \{i\}) \setminus \{j\}}$.
\end{enumerate}
We do not relabel coordinates associated to $N_T[j]$. This relabeling gives rise to a coordinate projection.

\begin{example}
Consider the tree $T^{I \cup J}$ from Figure~\ref{fig:FourierMotzkinRunningExample}. The coordinates of $\cim_T^{I,J}$ are associated to the sets
\[\{123,134,234,1234,345,455',55'6,567,568,678\}.\]
Consider the projection that deletes the $ijj' = 455'$ and $jj'k = 55'6$ coordinates. After relabeling we obtain
\[\{123,134,234,1234,346,455',55'6,567,568,678\}.\]
Now projecting away the appropriate coordinates yields the tree $X$ in Figure~\ref{fig:FourierMotzkinRunningExample}.
\end{example}

\begin{proposition}
Up to relabeling, $\cim_T^{I,J}$ projects onto $\cim_X^{I, J \setminus \{j\}}$.
\end{proposition}

\begin{proof}
Let $\ct$ be a DAG with skeleton $T$ and suppose $c_\ct \in \cim_T^{I,J}$. Without loss of generality we assume that $i \to j \to k$ since the other case is similar. We form a DAG $\cx$ with skeleton $X$ by replacing $i \to j \to k$ in $\ct$ with $i \to k$. Up to relabeling, $c_\ct$  maps to $c_\cx$ via a coordinate projection eliminating coordinates involving $j$. Hence $\cim_T^{I,J} \to \cim_X^{I,J \setminus \{j\}}$. Clearly on a given DAG this process may be reversed, and so the vertices of $\cim_T^{I,J}$ surject the vertices of $\cim_X^{I,J \setminus \{j\}}$.
\end{proof}

\subsection{Facets of Interventional CIM Polytopes}
\label{subsection:facetCIMTij}
In this subsection we describe the facets of $\cim_T^{I,J}$. We begin with star inequalities.

\begin{definition}
Let $\ell$ be an internal black node of $(T,I,J)$ and let $S'$ be a substar of $T$ with center $\ell$. The star inequality associated to $S'$ is
\[\sum\limits_{V(S') \subseteq A \subseteq N_T[\ell]} (-1)^{|A \setminus V(S')|} x_A \geq 0.\]
\end{definition}

For bidirected-edge inequalities, we refer readers back to Definition~\ref{definition:bidirectedEdgeInequality}, which naturally extends to the setting of general trees. Lastly, we have forked-tree inequalities. Recall that $T^I$ is the tree $T$ together with edges $i' \to i$ for $i \in I$ (see Subsection~\ref{sec:notation}).  

\begin{definition}
\label{definition:forkedTreeIJ}
Let $(T,I,J)$ be a gluing tree and let $T'$ be a subtree of $T$ containing no leaves of $T^I$. 
We say $T'$ is \emph{forked} if the following three conditions hold.
\begin{enumerate}
    \item For each leaf $u$ of $T'$ such that $u \not\in I \sqcup J$, $\deg_T(u) - \deg_{T'}(u) \geq 2$.
    \item For each $u \in I \sqcup J$, if $N_T[u] \cap V(T') \neq \emptyset$, then $u \in V(T')$.
    \item The set $I \sqcup J$ intersects $V(T')$ in only leaf nodes of $T'$.
\end{enumerate}
\end{definition}

\begin{definition}\label{def:forkedinequality}
Let $T'$ be a forked subtree of $(T,I,J)$ and let $L(T')$ denote the set of leaf edges $(c,d)$ of $T'$ with $c \in I \sqcup J$. We let $\interior(T')$ denote the set of nonleaf nodes of $T'$ and define $F(T')$ to be
\[F(T') = \{u \in V(T'):\deg_{T}(u) - \deg_{T'}(u) \geq 2\}.\]
The \emph{forked-tree inequality} is 
\[\sum\limits_{c \in F(T')} \ind_{c \to N_{T'}(c)}^{N_T[c]}(x) + \sum\limits_{(c,d) \in L(T')} (1-x_{cc'd}) \leq 1 + \sum\limits_{c \in \interior(T')} \#^{N_{T'}[c]}(x).\]
\end{definition}

Note that the neighborhood of $c$ in a tree forms a star subgraph, and so the function $\#^{N_{T'}[c]}(x)$ is as specifed in Definition~\ref{def:starhash}. Also, in the above definition, we use the notation $\ind_{c \to N_{T'}(c)}^{N_T[c]}(x)$ to denote the indicator function in Definition~\ref{def:starindicator} in the star subgraph of $T$ induced by the node set $N_T[c]$.  

The following H-representation is the first main result of this article.

\begin{theorem}
\label{theorem:facetsOfCIMTIJ}
The star inequalities, bidirected-edge inequalities, and forked-tree inequalities associated to $(T,I,J)$ form an H-representation of $\cim_T^{I,J}$.
\end{theorem}



\begin{figure}
    \centering
\begin{tikzpicture}
\draw (-0.5,0.866)--(0,0);
\draw (-0.5,-0.866)--(0,0);
\draw (0,0)--(3,0);
\draw[-Stealth] (2,1)--(2,0.1);
\draw (3,0)--(3.5,0.866);
\draw (3,0)--(3.5,-0.866);

\node[wB] at (-0.5,0.866) {};
\node[wB] at (-0.5,-0.866) {};
\node[wB] at (0,0) {};
\node[wB] at (1,0) {};
\node[wW] at (2,0) {};
\node[wB] at (2,1) {};
\node[wB] at (3,0) {};
\node[wB] at (3.5,0.866) {};
\node[wB] at (3.5,-0.866) {};

\node at (-0.75,0.866) {$1$};
\node at (-0.75,-0.866) {$2$};
\node at (-0.3,0) {$3$};
\node at (1,-0.3) {$4$};
\node at (2,-0.3) {$5$};
\node at (2,1.3) {$5'$};
\node at (3.3,0) {$6$};
\node at (3.75,0.866) {$7$};
\node at (3.75,-0.866) {$8$};
\node at (1.5, -1.5) {T};
\node at (5, 0) {};
\end{tikzpicture}
\begin{tikzpicture}
\draw [semithick] (0,0)--(0.94,0.94);
\draw [semithick] (0,2)--(1,1);
\draw [semithick] (2,1)--(1.08,1);
\draw [semithick] (3,1)--(2,1);
\draw [semithick] (3,1)--(4,0);
\draw [semithick] (3,1)--(4,2);
);

\node at (-0.25,0) {2};
\node at (-0.25,2) {1};
\node at (0.65,1) {3};
\node at (3.4,1) {6};
\node at (4.25,0) {8};
\node at (4.25,2) {7};
\node[below] at (2,1) {4};

\node[wB] at (0,0) {};
\node[wB] at (0,2) {};
\node[wB] at (1,1) {};
\node[wB] at (2,1) {};
\node[wB] at (3,1) {};
\node[wB] at (4,0) {};
\node[wB] at (4,2) {};

\node at (2, -.5) {$X$};
\end{tikzpicture}
    \caption{On the left is a gluing tree $(T, I, J)$ where $I = \emptyset$ and $J = \{5\}.$ The tree $X$ on the right is the resulting tree we obtain after projecting away the intervention on node $5$.}
    \label{fig:FourierMotzkinRunningExample}
\end{figure}
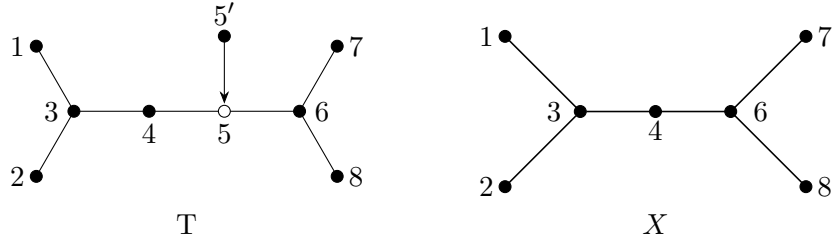

We note that restricting to the case of $I = J = \emptyset$ yields $\cim_T$ since we have no interventions.

\subsection{Fourier-Motzkin Elimination on CIM Polytopes}
\label{subsection:FMElimination}
We prove Theorem~\ref{theorem:facetsOfCIMTIJ} by induction. For the base case we note that if $J$ is a separated gluing tree, then one obtains Theorem~\ref{theorem:facetsseparatedgluingtree}. For the induction step we take the H-representation of $\cim_T^{I,J}$ and produce the H-representation of $\cim_X^{I, J \setminus \{j\}}$ via \emph{Fourier-Motzkin elimination} (see \cite[Chapter~1]{ziegler2012lectures} for a definition). 
We remind readers of the following equation
\begin{equation}
\label{equation:toricFiberProductEquation}
x_{ijj'} + x_{jj'k} = 1
\end{equation}
which relates the two coordinates we project away. The following proposition tells us how to perform this Fourier-Motzkin elimination. In the statement and proof of Proposition~\ref{proposition:signOfFunctional}, we use Equation~\ref{equation:toricFiberProductEquation} to eliminate $x_{jj'k}$ leaving only one coordinate to eliminate.

\begin{proposition}
\label{proposition:signOfFunctional}
Let $(T,I,J)$ be a gluing tree, let $j \in J$, and let $i - j$ be an edge of $T$.

\begin{enumerate}
\item If $a \cdot x \leq b$ is the bidirected-edge inequality associated to $i - j$, then $a_{ijj'} = 1$.
\item If $a \cdot x \leq b$ is the bidirected-edge inequality associated to $j - k$, then $a_{ijj'} = -1$.
\item If $a \cdot x \leq b$ is a forked-tree inequality with $i - j$ a leaf edge of the forked tree, then $a_{ijj'} = -1$.
\item If $a \cdot x \leq b$ is a forked-tree inequality with $j - k$ a leaf edge of the forked tree, then $a_{ijj'} = 1$.
\item If $a \cdot x \leq b$ is a star inequality, a bidirected-edge inequality not associated to $i - j$ or $j - k$, or a forked-tree inequality that does not contain $i - j$ or $j - k$, then $a_{ijj'} = 0$.
\end{enumerate}
\end{proposition}

\begin{proof}
(1) and (5) are immediate from the definitions of the respective inequalities. For (3) observe that the coordinate $x_{ijj'}$ only appears in the second sum on the left side of the forked-tree inequality in Definition~\ref{definition:forkedTreeIJ}, from which it is clear that its coefficient is $-1$. For (2) and (4), we know that the coefficient of $x_{jj'k}$ by (1) and (3) and we note that the sign changes when we replace $x_{jj'k}$ with $1 - x_{ijj'}$.
\end{proof}

We are now ready to perform Fourier-Motzkin elimination. Given some inequalities $a \cdot x \leq b$ and $c \cdot x \leq d$ from our H-representation of $\cim_T^{I,J}$, there are 7 cases (which by Proposition~\ref{proposition:signOfFunctional} cover all possibilities of ways we can sum inequalities to cancel the $ijj'$-coordinate). The first 3 cases are from Proposition~\ref{proposition:signOfFunctional} (5) and the remaining 4 cover ways (up to symmetry in $i-j$ and $j-k$) that a positive $a_{ijj'}$ in one inequality can cancel with a negative $a_{ijj'}$ in another when we sum them. There are 3 ways these signs can differ (sum two bidirected-edge inequalities, sum two forked-tree inequalities, and sum one of each, from the appropriate sides). The distinction between cases 6 and 7 is a technical distinction needed for the proof later.
The 7 cases presented as inequalities associated to $(T,I,J)$ are the following:

\begin{enumerate}
    \item $a \cdot x \leq b$ is a star inequality.
    \item $a \cdot x \leq b$ is a bidirected-edge inequality corresponding to an edge other than $i - j$ and $j - k$.
    \item $a \cdot x \leq b$ is a forked-tree inequality with forked tree containing neither $i - j$ nor $j - k$.
    \item $a \cdot x \leq b$ and $c \cdot x \leq d$ are the bidirected-edge inequalities associated to $i - j$ and $j - k$
    \item $a \cdot x \leq b$ is a forked-tree inequality with forked tree containing $i - j$ and $c \cdot x \leq d$ is a forked-tree inequality with forked tree containing $j - k$.
    \item $a \cdot x \leq b$ is the bidirected-edge inequality associated to $i - j$, $c \cdot x \leq d$ is a forked-tree inequality with forked tree $T^\prime$  containing $i - j$, and $i \in \interior(X') \cup F(X')$ where $X' = T' \setminus \{j\}$.
    \item $a \cdot x \leq b$ is the bidirected-edge inequality associated to $i - j$, $c \cdot x \leq d$ is a forked-tree inequality with forked tree $T^\prime$ containing $i - j$, and $i \not\in \interior(X') \cup F(X')$ where $X' = T' \setminus \{j\}$.
\end{enumerate}

In this subsection we show that the 7 previous cases correspond to the following 7 collections of inequalities defined for $(X, I, J\setminus j)$, respectively:

\begin{enumerate}
    \item Star inequalities in $X$.
    \item bidirected-edge inequalities in $X$ corresponding to an edge other than $i - k$.
    \item Forked-tree inequalities with forked tree not containing $i - k$.
    \item The bidirected-edge inequality associated to $i - k$.
    \item Forked-tree inequalities with forked tree containing $i - k$.
    \item Forked-tree inequalities with forked tree containing $i$ or $k$ but not both.
    \item Extraneous inequalities.
\end{enumerate}

We briefly check the first three cases. By Proposition~\ref{proposition:signOfFunctional}, in these cases $a \cdot x \leq b$ has $a_{ijj'} = a_{jj'k} = 0$, and so by Fourier-Motzkin elimination $a \cdot x \leq b$ is valid on $\cim_X^{I, J \setminus \{j\}}$. Any star inequality, bidirected-edge inequality (other than the one associated to $i - k$), or forked-tree inequality (with forked tree containing neither $i$ nor $k$) of $\cim_X^{I, J \setminus \{j\}}$ arises in this way. We now check case 4.

\begin{lemma}
\label{lemma:bidirectedEdgeIeqProjection}
Let $a \cdot x \leq b$ and $c \cdot x \leq d$ be the bidirected-edge inequalities of $\cim_T^{I,J}$ associated to the edges $i - j$ and $j - k$. Then $(a + c) \cdot x \leq b + d$ is the bidirected-edge inequality associated to $i - k$.
\end{lemma}

\begin{proof}
The bidirected-edge inequalities associated to $i - j$ and $j - k$ are
\[\ind_{i \from j}(x) + x_{ijj'} \leq 1 \quad \mbox{and} \quad \ind_{k \from j}(x)  + x_{jj'k} \leq 1.\]
In $\cim_T^{I,J}$ we know that $x_{ijj'} + x_{jj'k} = 1$ by Corollary~\ref{corollary:hyperplane}, and so we may rewrite the latter inequality to have the following form
\[\ind_{k \from j}(x)  - x_{ijj'} \leq 0.\]
Summing these two inequalities yields the bidirected-edge inequality of $\cim_X^{I, J \setminus \{j\}}$ up to relabeling.
\end{proof}


\begin{example}
\label{ex:FMCase4}
Consider the gluing tree $(T, I, J)$ pictured on the left in Figure \ref{fig:FourierMotzkinRunningExample}. In this case, $i = 4,~ j = 5$ and $ k = 6$.   The bidirected-edge inequality associated to the edges $i-j = 4-5$ and $j-k=5-6$ are
\begin{align*}
&\ind_{4 \leftarrow 5} + \ind_{5 \leftarrow 4} = \sum\limits_{\substack{A \subseteq N_T(4) \setminus \{4,5\} \\ |A| \geq 1}} (-1)^{|A|+1}x_{A \cup \{4,5\}} + x_{455'} = x_{345} + x_{455'} \leq 1, \\
&\ind_{6 \leftarrow 5} + \ind_{5 \leftarrow 6} = \sum\limits_{\substack{A \subseteq N_T(6) \setminus \{5,6\} \\ |A| \geq 1}} (-1)^{|A|+1}x_{A \cup \{4,5\}} + x_{55'6} = (x_{567} + x_{568} - x_{5678})+ x_{55'6} \leq 1. \\\end{align*}
Recall that $x_{455'}+x_{55'6} = 1$ for all points in $\cim_T^{I,J}$ by construction. Making the substitution $x_{55'6} = 1 - x_{455'}$ in the second inequality transforms our two inequalities into
\begin{align*}
&x_{345} + x_{455'} \leq 1, \\
&(x_{567} + x_{568} - x_{5678}) + 1 - x_{455'} \leq 1 \iff (x_{567} + x_{568} - x_{5678}) -x_{455'} \leq 0.
\end{align*}
Summing these two inequalities yields the inequality 
\[
x_{345} + (x_{567} + x_{568} - x_{5678}) \leq 1. 
\]
We then relabel this inequality according to the rules described at the beginning of this section. Observe that each coordinate involving node $5$ involves either $4$ or $6$ but not both. The relabelling simply says that we replace every $5$ with a $4$ if the coordinate does not contain $4$ and a $5$ with a $6$ if the coordinate does not contain $6$ which yields the inequality
\[
x_{346} + (x_{467} + x_{468} - x_{4678}) \leq 1. 
\]
Observe that this is exactly the bidirected-edge inequality
\[
\ind_{4 \leftarrow 6} + \ind_{6 \leftarrow 4} \leq 1
\]
associated to the edge $4-6$ in the tree $X$ which is pictured on the right in Figure \ref{fig:FourierMotzkinRunningExample}. 
\end{example}

We have now handled cases 1 - 4.
It remains to consider cases 5 - 7 and obtain the forked-tree inequalities containing $i - k$ and no additional inequalities. We move on to case 5.

\begin{lemma}
\label{lemma:forkedTreeOppositeSides}
Let $a \cdot x \leq b$ and $c \cdot x \leq d$ be forked-tree inequalities of $\cim_T^{I,J}$ with forked trees containing leaves $i - j$ and $j - k$ respectively. Then $(a + c) \cdot x \leq b + d$ is a forked-tree inequality with forked tree of $(X, I, J \setminus \{j\})$ containing internal edge $i - k$. Moreover, all forked-tree inequalities with forked tree containing $i - k$ arise in this way.
\end{lemma}

\begin{proof}
Let $T'$ and $T''$ be forked subtrees of $(T,I,J)$ with $i - j$ an edge of $T'$ and $j - k$ an edge of $T''$ both with leaf node $j$. We have the following forked-tree inequalities:
\begin{equation}
\label{equation:twoForkedTrees}
\begin{split}
    \sum\limits_{c \in F(T')} \ind_{c \to N_{T'}(c)}^{N_T[c]}(x) + \sum\limits_{(c,d) \in L(T')} (1-x_{cc'd}) \leq 1 + \sum\limits_{c \in \interior(T')} \#^{N_{T'}[c]}(x), \\
    \sum\limits_{c \in F(T'')} \ind_{c \to N_{T''}(c)}^{N_T[c]}(x) + \sum\limits_{(c,d) \in L(T'')} (1-x_{cc'd}) \leq 1 + \sum\limits_{c \in \interior(T'')} \#^{N_{T''}[c]}(x).
\end{split}
\end{equation}
From $T'$ and $T''$ we form a subtree $X'$ of $X$ by taking the union of the vertex and edge sets of $T'$ and $T''$ and replacing $i - j - k$ with $i - k$. The three conditions for $X'$ to be a forked subtree of $(X,I,J \setminus \{j\})$ depend only on the adjacencies of leaf nodes of $X'$ and elements of $(I \sqcup J) \setminus \{j\}$. Note that $T'$ and $T''$ are forked subtrees of $(T,I,J)$ and for all $v \in V(T') \setminus \{j\}$ (and similarly for $T''$) we have both $\deg_T(v) = \deg_X(v)$ and $\deg_{T'}(v) = \deg_{X'}(v)$. Since the definition of forked subtrees depends only on these adjacencies, we know $X'$ is forked.

Conversely, given a forked subtree $X'$ of $(X, I, J \setminus \{j\})$ that contains $i-k$, we may reverse the deletion of $j$ and $j'$ and obtain trees $T'$ and $T''$ which combine to yield $X'$. Hence any forked tree associated to a tree $X'$ containing $i - k$ arises from combining forked trees containing $i-j$ and $j-k$.
To prove the claimed result, it remains to show that the inequalities reflect this.  

Namely, we claim that the sum of the inequalities in Equation~\ref{equation:twoForkedTrees} is the forked-tree inequality associated to $X'$. We first show that (up to relabeling)
\[\sum\limits_{c \in F(X')} \ind_{c \to N_{X'}(c)}^{N_X[c]}(x) = \sum\limits_{c \in F(T')} \ind_{c \to N_{T'}(c)}^{N_T[c]}(x) + \sum\limits_{c \in F(T'')} \ind_{c \to N_{T'}(c)}^{N_T[c]}(x).\]
To see this, we note that elements  $v \in F(T')$ (and similarly $v \in F(T'')$) have $\deg_T(v) - \deg_{T'}(v) \geq 2$ by definition. By Definition~\ref{definition:forkedTreeIJ} any $v \in J \cap F(T')$ must satisfy $\deg_T(v) \geq 2 + \deg_{T'}(v) \geq 3$, which contradicts $v \in J$. Hence $J \cap F(T') = \emptyset$. It follows that $F(X') = F(T') \cup F(T'')$ since the degrees of each vertex within $T$ and $T' \cup T''$ are the same in $X$ and $X'$ except at $j$ which does not appear in the sum.

We now show that up to relabeling,
\[\sum\limits_{c \in \interior(X')} \#^{N_{X'}[c]}(x) = \sum\limits_{c \in \interior(T')} \#^{N_{T'}[c]}(x)  + \sum\limits_{c \in \interior(T'')} \#^{N_{T''}[c]}(x).\]
Again the adjacencies of $X'$ and $T' \cup T''$ are the same except at $i,j,k$. Since $j$ is not an interior node of $X'$, $T'$, or $T''$ we may disregard this node. The degrees of $i$ and $k$ are preserved when passing from $T' \cup T''$ to $X'$ since $i - j - k$ is replaced with $i - k$. 
Hence, after relabeling, the equation holds. The sum of the inequalities in Equation~\ref{equation:twoForkedTrees} becomes
\begin{equation}
\label{equation:forkedTreeSum}
\sum\limits_{c \in F(X')} \ind_{c \to N_{X'}(c)}^{N_X[c]}(x) + \sum\limits_{(c,d) \in L(T')} (1-x_{cc'd}) + \sum\limits_{(c,d) \in L(T'')} (1-x_{cc'd}) \leq 2 + \sum\limits_{c \in \interior(X')} \#^{N_{X'}[c]}(x).
\end{equation}
Then
\begin{align*}
&\sum\limits_{(c,d) \in L(T')} (1-x_{cc'd}) + \sum\limits_{(c,d) \in L(T'')} (1-x_{cc'd}) - 1 \\
&= \sum\limits_{\substack{(c,d) \in L(T') \\ (c,d) \neq (j,i)}} (1-x_{cc'd}) + \sum\limits_{\substack{(c,d) \in L(T'') \\ (c,d) \neq (j,k)}} (1-x_{cc'd}) + (1 - x_{ijj'}) + (1 - x_{jj'k}) - 1, \\
&= \sum\limits_{\substack{(c,d) \in L(T') \\ (c,d) \neq (j,i)}} (1-x_{cc'd}) + \sum\limits_{\substack{(c,d) \in L(T'') \\ (c,d) \neq (j,k)}} (1-x_{cc'd}). \\
\end{align*}
where the last equality holds by Corollary~\ref{corollary:hyperplane}. Note that $L(X') = (L(T') \cup L(T'')) \setminus \{(j,i), (j,k)\}$ and so
\[\sum\limits_{(c,d) \in L(X')} (1-x_{cc'd}) = \sum\limits_{(c,d) \in L(T')} (1-x_{cc'd}) + \sum\limits_{(c,d) \in L(T'')} (1-x_{cc'd}) - 1.\]
Substituting this equation into Equation~\ref{equation:forkedTreeSum} yields the forked-tree inequality associated to $X'$.
\end{proof}

\begin{example}
\label{ex:FMCase5}
Again consider the gluing tree $(T, I, J)$ pictured in Figure \ref{fig:FourierMotzkinRunningExample}. Recall in this setting that $i = 4, j = 5$, and $k = 6$. Observe that the subtree $T' = 3-4-5$ is a forked tree which contains the edge $i-j = 4-5$ and the subtree $T'' = 5-6$ is a also a forked subtree but which clearly contains the edge $j-k = 5-6$. We can quickly check that $T'$ is forked by verifying it satisfies conditions (1), (2), and (3) in Definition \ref{definition:forkedTreeIJ}. For condition (1), the leaf $5 \in J$ so the condition is trivially true and for the leaf $3$ it holds that $\deg_T(3) - \deg_{T'}(3) = 3 - 1 \geq 2$. For condition (2) we have that $I \sqcup J = \{5\} \subseteq V(T')$ so this condition is trivially true and similarly for condition (3) thus $T'$ is a forked subtree. A similar check can be conducted for $T''$. 

Now recall that the inequalities associated to the forked subtree $T'$ is
\begin{align*}
\sum\limits_{c \in F(T')} \ind_{c \to N_{T'}(c)}^{N_T[c]}(x) + \sum\limits_{(c,d) \in L(T')} (1-x_{cc'd}) \leq 1 + \sum\limits_{c \in \interior(T')} \#^{N_{T'}[c]}(x)
\end{align*}
where $F(T') = \{u \in V(T'):\deg_{T}(u) - \deg_{T'}(u) \geq 2\} = \{3\}$, $L(T') = \{(4,5)\}$, and $\mathrm{int}(T') = \{4\}$. Thus the inequality above becomes
\[
(x_{123} - x_{1234}) + (1 - x_{455'}) \leq 1 + x_{345}
\]
associated to the forked tree $T'$. A similar calculation yields the inequality
\[
(x_{678} - x_{5678}) + (1-x_{55'6}) \leq 1
\]
for the forked tree $T'' = 5-6$. Now we again utilize the fact that $x_{445'} + x_{55'6} = 1$ to combine these two inequalities. Doing so yields
\begin{align*}
(x_{123} - x_{1234}) + (x_{678} - x_{5678}) \leq 1 + x_{345}.
\end{align*}
Again we apply the relabeling to get the inequality
\[
(x_{123} - x_{1234}) + (x_{678} - x_{4678}) \leq 1 + x_{346},
\]
which is exactly the inequality associated to the forked subtree $X' = 3-4-6$ of $X$. 
\end{example}

We now address Cases 6 and 7, which have identical Fourier-Motzkin steps but do not both yield facet-defining inequalities. Let $T'$ be a forked subtree of $T$ containing the edge $i - j$. In $X$ we have a subtree $X'$ associated to $T'$ which is induced on $V(T') \setminus \{j\}$. In the previous case, we considered forked trees from opposite sides of $j$ which combined to give forked trees containing $i - k$. Now we consider the operation which ``closes off'' a forked tree at $i$, so that $X'$ does not contain $i-k$. These trees arise from our remaining instance of Fourier-Motzkin elimination: combining the forked-tree inequality associated to $T'$ and the bidirected-edge inequality associated to $i - j$ to obtain the forked-tree inequality associated to $X'$. We first showcase this with an example which provides a concrete illustration of what will happen in these steps of Fourier-Motzkin elimination.

\begin{example}
Again consider the gluing tree $(T, I, J)$ pictured on the left in Figure \ref{fig:FourierMotzkinRunningExample}. Recall from Examples \ref{ex:FMCase4} and \ref{ex:FMCase5} that the bidirected-edge inequality associated to the edge 
$5-6$ is
\[
(x_{567} + x_{568} - x_{5678})+ x_{55'6} \leq 1,
\]
while the forked tree inequality associated to $T'' = 5-6$ of $T$ is
\[
(x_{678} - x_{5678}) + (1-x_{55'6}) \leq 1.
\]
Observe that adding these two inequalities together yields the inequality
\[
x_{567} + x_{568} + x_{678} - 2x_{5678} \leq 1,
\]
which after relabelling is
\[
x_{467} + x_{468} + x_{678} - 2x_{4678} \leq 1. 
\]
Note that this is exactly the forked-tree inequality associated to the lone node $6$ which is a valid, facet-defining inequality on the polytope $\cim_X$ where $X$ is the tree pictured on the right in Figure \ref{fig:FourierMotzkinRunningExample}. 
Notice also that this forked tree contains $k = 6$ but not $i = 4$. Moreover, observe that in this case, $6 \in F(X'')$ where $X'' = T'' \setminus \{5\}$. 
Thus, we are in case (6) of the Fourier-Motzkin elimination steps described at the beginning of this section

On the other hand, recall the bidirected-edge inequality associated to $4-5$ is
\[
x_{345} + x_{455'} \leq 1,
\]
while the forked-tree inequality associated to $T' = 3-4-5$ of $T$ is
\[
(x_{123} - x_{1234}) + (1 - x_{455'}) \leq 1 + x_{345}. 
\]
Observe that adding these two inequalities together yields the inequality
\[
(x_{123} - x_{1234}) \leq 1.
\]
Observe that in this case, $4 \notin \mathrm{int}(X') \cup F(X')$ where $X' = T'' \setminus \{5\}$ so we are in case (7) of the Fourier-Motzkin elimination steps described at the beginning of this section. 
Thus, we claim that this inequality is extraneous. Indeed, observe that it is simply the sum of the inequalities
\begin{align*}
x_{123} + x_{134} + x_{234} - 2x_{1234} \leq 1 \\
x_{234} - x_{1234} \geq 0 \\
x_{134} - x_{1234} \geq 0 \\
x_{1234} \geq 0.
\end{align*}
Therefore, $(x_{123} - x_{1234}) \leq 1$ is an extraneous inequality as claimed. 
\end{example}

We now complete steps (6) and (7) in full generality. To perform this Fourier-Motzkin step, we compute differences of functionals in the forked-tree inequalities associated to $T'$ and $X'$. 

\begin{proposition}
\label{proposition:differenceHashtagFunctional}
Let $T'$ be a forked subtree of $T$ containing the edge $i - j$ and let $X'$ be the associated forked subtree of $X$. Then 
\[\#^{N_{T'}[i]}(x) - \#^{N_{T'}[i] \setminus \{j\}}(x) = \sum\limits_{\substack{\{i,j\} \subseteq C \subseteq N_{T'}[i] }} (-1)^{|C|} x_C.\]
\end{proposition}

\begin{proof}
The difference on the left-hand side is
\[\sum\limits_{i \in C \subseteq N_{T'}[i]} (-1)^{|C|-1} x_C - \sum\limits_{i \in C \subseteq N_{T'}[i] \setminus \{j\}} (-1)^{|C|-1} x_C.\]
All terms in the right sum cancel leaving only terms associated to a set $C$ which contains both $i$ and $j$.
\end{proof}

Using Proposition~\ref{proposition:differenceHashtagFunctional} we now analyze the difference in the right-hand side of the forked-tree inequalities associated to $T'$ and $X'$.

\begin{proposition}
\label{proposition:combineHashtagFunctionals}
Let $T'$ be a forked subtree of $T$ containing $i - j$ and let $X'$ be obtained from $T'$ by deleting $j$ (considered as a subtree of $X$). Then up to relabeling
\[
\sum\limits_{c \in \interior(T')} \#^{N_{T'}[c]}(x) - \sum\limits_{\{i,j\} \subseteq C \subseteq N_{T'}[i]} (-1)^{|C|-1} x_C = \sum\limits_{c \in \interior(X')} \#^{N_{X'}[c]}(x).
\]
\end{proposition}

\begin{proof}
Note that $\#^{N_{T'}[i]}(x) = 0$ if $i$ is not an interior node of $T'$ because the sum is indexed by the empty set. By Proposition~\ref{proposition:differenceHashtagFunctional} we have
\begin{align*}
&\sum\limits_{c \in \interior(T')} \#^{N_{T'}[c]}(x) - \sum\limits_{\{i,j\} \subseteq C \subseteq N_{T'}[i]} (-1)^{|C|-1} x_C \\
&=\sum\limits_{c \in \interior(T') \setminus \{i\}} \#^{N_{T'}[c]}(x) + \#^{N_{T'}[i]}(x) - \sum\limits_{\{i,j\} \subseteq C \subseteq N_{T'}[i]} (-1)^{|C|-1} x_C, \\
&=\sum\limits_{c \in \interior(T') \setminus \{i\}} \#^{N_{T'}[c]}(x) + \#^{N_{T'}[i] \setminus \{j\}}(x).
\end{align*}
Now recall that $X'$ is obtained from $T'$ by deleting $j$. In particular, the adjacencies of $T'$ and $X'$ only differ at $i$, the unique neighbor of $j$ in $T'$. Up to relabeling, the previous expression is just
\[\sum\limits_{c \in \interior(X') \setminus \{i\}} \#^{N_{X'}[c]}(x) + \#^{N_{X'}[i]}(x) = \sum\limits_{c \in \interior(X')} \#^{N_{X'}[c]}(x).\]
\end{proof}

%

Towards analyzing the difference in the left-hand sides of the forked-tree inequalities of $T'$ and $X'$, we compute the following difference at $i$.

\begin{proposition}
\label{proposition:differenceForcingFunctional}
Let $T'$ be a forked subtree of $T$ containing the edge $i - j$. Then 
\[\ind_{i \to N_{T'}(i) \setminus \{j\}}^{N_T[i]}(x) - \ind_{i \to N_{T'}(i)}^{N_T[i]}(x) = \sum\limits_{\substack{\{i,j\} \subseteq C \subseteq N_T[i] \\ \exists \, \ell \in C \setminus N_{T'}[i]}} (-1)^{|C|-1} x_C.\]
\end{proposition}

We refer the reader back to Proposition~\ref{proposition:differenceHashtagFunctional} to note the similarity in the right-hand sides.

\begin{proof}
The left-hand side is given by the following difference:
\begin{align*}
\left( \sum\limits_{i \in A \subseteq \left( N_T[i] \setminus N_{T'}(i) \right) \cup \{j\}} \sum\limits_{B \subseteq N_{T'}(i) \setminus \{j\}} (-1)^{|A \cup B| - 1} \left(|A|-2\right) x_{A \cup B} \right) \\
- \left( \sum\limits_{i \in A \subseteq \left( N_T[i] \setminus N_{T'}(i) \right) } \sum\limits_{B \subseteq N_{T'}(i) } (-1)^{|A \cup B| - 1} \left(|A|-2\right) x_{A \cup B} \right).
\end{align*}
Per the definition of the indicator function $\ind_{i \to S}(x)$, the sets $A$ in the sums above must have cardinality at least $3$.
These double sums differ only in whether $j$ is allowed to be in $A$ or in $B$. We give a sign-reversing involution between terms of the two double sums. Define $\tau$ by the following:
\[\tau(A,B) = \begin{cases}
    (A,B) &\mbox{if}~j \not\in A ~\mbox{and}~ j \not\in B, \\
    (A \cup \{j\},B \setminus \{j\}) &\mbox{if}~j \not\in A ~\mbox{and}~ j \in B, \\
    (A \setminus \{j\},B \cup \{j\}) &\mbox{if}~j \in A ~\mbox{and}~ j \not\in B. \\
\end{cases}\]
In both sums we have $A$ and $B$ lying within two disjoint sets, and so we need not consider the case where $j \in A$ and $j \in B$. Clearly $\tau$ is an involution, and since we may only have $j \in A$ within terms of the first sum and $j \in B$ within terms of the second, $\tau$ maps between terms of the two sums. We note that $\tau$ preserves $A \cup B$, and consequently reverses sign in the difference of the two double sums. Hence $\tau$ is a sign reversing involution.

Consider the term in the first double sum given by $(A,B)$. If $j \not\in A$ and $j \not\in B$ then $\tau(A,B) = (A,B)$, and the two terms cancel. We cannot have $j \not\in A$ and $j \in B$ in the first term, so the remaining case is to consider is $j \in A$ and $j \not\in B$. In this case, the weight of $(A,B)$ in the first sum is
\[(-1)^{|A \cup B|-1} (|A|-2) x_{A \cup B}.\]
The weight of $\tau(A,B) = (A \setminus \{j\}, B \cup \{j\})$ in the second sum is
\[(-1)^{|\left(A \setminus \{j\}\right) \cup \left(B \cup \{j\}\right)| - 1} \left( |A \setminus \{j\}| - 2 \right) x_{\left(A \setminus \{j\}\right) \cup \left( B \cup \{j\}\right)} = (-1)^{|A \cup B| - 1} \left( |A| - 3 \right) x_{A  \cup  B}.\]
Hence subtracting the double sums yields exactly one instance of $(-1)^{|A \cup B|-1} x_{A \cup B}$ in the first sum and we obtain
\[
\sum\limits_{\{i,j\} \subseteq A \subseteq \left( N_T[i] \setminus N_{T'}(i) \right) \cup \{j\}} \sum\limits_{B \subseteq N_{T'}(i) \setminus \{j\}} (-1)^{|A \cup B| - 1} x_{A \cup B}.
\]
We reindex the sum with $C = A \cup B$. Since $A$ contains at least 3 vertices of $T$ that are not in $T'$, we know $C$ ranges over all subsets of $N_T[i]$ which contain $i$ and $j$ and some other $\ell$ not in $T'$. Hence we obtain
\[
\sum\limits_{\substack{\{i,j\} \subseteq C \subseteq N_T[i] \\ \exists \, \ell \in C \setminus N_{T'}[i]}} (-1)^{|C|-1} x_C,
\]
as desired.
\end{proof}

\begin{proposition}
\label{proposition:combineForcingFunctionals}
Let $T'$ be a forked subtree of $T$ containing $i - j$ and let $X'$ be obtained from $T'$ by deleting $j$. Then up to relabeling
\[
\sum\limits_{c \in F(T')} \ind_{c \to N_{T'}(c)}^{N_T[c]} + \sum\limits_{\substack{\{i,j\} \subseteq C \subseteq N_{T}[i] \\ \exists \, \ell \in C \setminus N_{T'}[i]} } (-1)^{|C|-1} x_C = \sum\limits_{c \in F(X')} \ind_{c \to N_{X'}(c)}^{N_X[c]}(x).
\]
\end{proposition}

\begin{proof}
Note that $\ind_{i \to N_{T'}(i)}^{N_T[i]} = 0$ if $i \not\in F(T')$. By Proposition~\ref{proposition:differenceForcingFunctional} the left hand side is the same as
\[
\sum\limits_{c \in F(T')} \ind_{c \to N_{T'}(c)}^{N_T[c]}(x) + \ind_{i \to N_{T'}(i) \setminus \{j\}}^{N_T[i]}(x) - \ind_{i \to N_{T'}(i)}^{N_T[i]}(x) = \sum\limits_{c \in F(T') \setminus \{i\}} \ind_{c \to N_{T'}(c)}^{N_T[c]}(x) + \ind_{i \to N_{T'}(i) \setminus \{j\}}^{N_T[i]}(x)
\]
which we simplified by peeling off the $c=i$ term of the sum (regardless of whether $i \in F(T')$ or not). Since $X'$ is obtained from $T'$ by deletion of node $j$ and leaf edge $i-j$, we know that $\deg_T(\ell) - \deg_{T'}(\ell) = \deg_X(\ell) - \deg_{X'}(\ell)$ for all vertices $\ell \neq i$ of $X'$. So $F(T')$ and $F(X')$ differ only at possibly $i$. After relabeling, we obtain
\[
\sum\limits_{c \in F(X') \setminus \{i\}} \ind_{c \to N_{X'}(c)}^{N_X[c]}(x) + \ind_{i \to N_{X'}(i)}^{N_X[i]}(x) = \sum\limits_{c \in F(X')} \ind_{c \to N_{X'}(c)}^{N_X[c]}(x).
\]
\end{proof}

We are now ready to complete the Fourier-Motzkin elimination for cases 6 and 7. 

\begin{lemma}
\label{lemma:FTIPlusBDIpart1}
Let $T'$ be a forked tree containing the edge $i - j$ and let $X' = T' \setminus \{j\}$ be the associated subtree of $X$. Let $a \cdot x \leq b$ be the forked-tree inequality associated to $T'$ and let $c \cdot x \leq d$ be the bidirected-edge inequality associated to the edge $i - j$. Then $(a + c) \cdot x \leq b + d$ simplifies to
\[
\sum\limits_{c \in F(X')} \ind_{c \to N_{X'}(c)}^{N_X[c]}(x) + \sum\limits_{(c,d) \in L(X')} (1- x_{cdd'}) \leq 1 + \sum\limits_{c \in \interior(X')} \#^{N_{X'}[c]}(x).
\]
In particular, if $X'$ is a forked subtree of $X$, then this is the forked-tree inequality associated to $X'$.
\end{lemma}

\begin{proof}
We sum the forked-tree inequality associated to $T'$ and the bidirected-edge inequality associated to $i - j$. This yields
\[\sum\limits_{c \in F(T')} \ind_{c \to N_{T'}(c)}^{N_T[c]}(x) + \sum\limits_{(c,d) \in L(T')} (1- x_{cc'd}) + \ind_{i \from j}(x) + x_{ijj'} \leq 2 + \sum\limits_{c \in \interior(T')} \#^{N_{T'}[c]}(x).\]
Note that $(j,i) \in L(T')$ since $j$ was intervened on, and so we may peel off a copy of $1-x_{ijj'}$ from the sum and simplify to obtain
\[\sum\limits_{c \in F(T')} \ind_{c \to N_{T'}(c)}^{N_T[c]}(x) + \sum\limits_{(c,d) \in L(T'\setminus \{j\})} (1- x_{cc'd}) + \ind_{i \from j}(x) \leq 1 + \sum\limits_{c \in \interior(T')} \#^{N_{T'}[c]}(x).\]
We may rewrite $\ind_{i \from j}(x)$ as
\[\ind_{i \from j}(x) = \sum\limits_{\{i,j\} \subseteq C \subseteq N_{T}[i]} (-1)^{|C|-1} x_C = \sum\limits_{\{i,j\} \subseteq C \subseteq N_{T'}[i]} (-1)^{|C|-1} x_C + \sum\limits_{\substack{\{i,j\} \subseteq C \subseteq N_{T}[i] \\ \exists \, \ell \in C \setminus N_{T'}[i]} } (-1)^{|C|-1} x_C.\]
In words, we split $\ind_{i \from j}(x)$ into the terms with $C$ completely contained in $T'$ and the terms with $C$ not contained in $T'$. Substituting this into the previous inequality yields
\begin{align*}
&\sum\limits_{c \in F(T')} \ind_{c \to N_{T'}(c)}^{N_T[c]}(x) + \sum\limits_{(c,d) \in L(X')} (1- x_{cc'd}) + \sum\limits_{\{i,j\} \subseteq C \subseteq N_{T'}[i]} (-1)^{|C|-1} x_C \\
&\leq 1 + \sum\limits_{c \in \interior(T')} \#^{N_{T'}[c]}(x) - \sum\limits_{\substack{\{i,j\} \subseteq C \subseteq N_{T}[i] \\ \exists \, \ell \in C \setminus N_{T'}[i]} } (-1)^{|C|-1} x_C.
\end{align*}
By Proposition~\ref{proposition:combineHashtagFunctionals}, Proposition~\ref{proposition:combineForcingFunctionals}, and the fact that $L(X') = L(T' \setminus \{j\})$ we obtain the following inequality
\[
\sum\limits_{c \in F(X')} \ind_{c \to N_{X'}(c)}^{N_X[c]}(x) + \sum\limits_{(c,d) \in L(X')} (1- x_{cc'd}) \leq 1 + \sum\limits_{c \in \interior(X')} \#^{N_{X'}[c]}(x).
\]
\end{proof}

So far we have not assumed $X'$ to be forked, only that it is obtained from the forked tree $T'$ via deletion of $j$. We now show that cases 6 and 7 correspond to whether $X'$ is forked or not.

\begin{proposition}
\label{proposition:xPrimeForkedCase6}
Let $T'$ be a forked tree containing the edge $i-j$ and let $X' = T' \setminus \{j\}$ be the associated subtree of $X$. Then $X'$ is a forked subtree of $X$ if and only if $i \in \interior(X') \cup F(X')$. Furthermore, any forked subtree $X'$ of $X$ gives rise to a forked subtree $X' \cup \{i-j\}$ of $T$.
\end{proposition}

\begin{proof}
Let $T'$ be a forked subtree of $T$ containing $i-j$ and let $X'$ be the subtree of $X$ obtained by removing $j$ from $T'$. We check the three conditions in Definition~\ref{definition:forkedTreeIJ}. Since $T'$ is forked, $I \sqcup J$ intersects $V(T')$ in only leaf nodes of $T'$. Deletion of $j$ preserves this property since $j$ is not adjacent to any other element of $I \sqcup J$ (by the coloring of the gluing tree). Condition (2) remains true in $X'$ since we have only deleted leaf node $j \in J$, which is not adjacent to any other element of $I \sqcup J$. 

It remains to check condition (1), which need only be checked at $i$. Note that $i \not\in I \sqcup J$. Condition (1) is vacuously true if $i \in \interior(X')$ and otherwise holds exactly when $i \in F(X') \setminus \interior(X')$. Hence $X'$ is forked if and only if $i \in \interior(X') \cup F(X')$.

For the second statement in Proposition~\ref{proposition:xPrimeForkedCase6}, we let $X'$ be a new forked subtree of $X$ and let $T'$ be the subtree of $T$ obtained from $X'$ by adding $i-j$. For Condition (3), we added a leaf node $j$ to $J$ with edge $i-j$. Since $V(X')$ only intersects $I \sqcup J \setminus \{j\}$ in leaf nodes, we see $V(T')$ only intersects $I \sqcup J$ in leaf nodes. 

For Condition (2), note that adjacencies in $X$ and $T$ (and adjacencies in $X'$ and $T'$) only differ at $i$ and $j$. Since Condition (2) holds for $X'$ in $X$, we check Condition (2) for $T'$ in $T$ only at the new node $j$. Note that $N_T[j] \cap V(T') \neq \emptyset$ and that $j \in V(T')$.
So Condition (2) holds.

For Condition (1), $j$ is the unique leaf of $T'$ which is not a leaf in $X'$. But $j \in I \sqcup J$ and so Condition (1) does not apply to this node and Condition (1) holds for $T'$ since it holds in $X'$. 
Hence, $T'$ is a forked subtree of $T$.
\end{proof}

Proposition~\ref{proposition:xPrimeForkedCase6} completes Case 6 and shows that this case is where all forked subtrees $X'$ of $X$ containing $i$ or $k$ (but not both) arise. 
We refer readers back to the two lists of 7 cases at the beginning of the subsection to see that we have now obtained all inequalities promised in the statement of Theorem~\ref{theorem:facetsOfCIMTIJ}. It remains to show the inequalities in Case 7 are extraneous. The following lemma concludes the proof of Theorem~\ref{theorem:facetsOfCIMTIJ}.

\begin{lemma}
\label{lemma:FTIPlusBDIpart2}
Let $T'$ be a forked tree containing the edge $i - j$ and let $X' = T' \setminus \{j\}$ be the associated subtree of $X$. Let $a \cdot x \leq b$ be the forked-tree inequality associated to $T'$ and let $c \cdot x \leq d$ be the bidirected-edge inequality associated to the edge $i - j$. If $X'$ is not a forked subtree of $X$, then $(a + c) \cdot x \leq b + d$ is not facet-defining.
\end{lemma}

\begin{proof}
We first compute the degree of $i$ in $X$ and $X'$. Since $X'$ is not forked in $X$, by Proposition~\ref{proposition:xPrimeForkedCase6} we know $i \not\in \interior(X') \cup F(X')$. Since $i \not\in \interior(X')$, we know $\deg_{X'}(i) \leq 1$ and in fact equality holds. To see this, observe that $i$ is a non-leaf of $T$ (since $i$ is in a forked subtree $T'$) and hence $i$ is a nonleaf of $X$. If $i$ has no neighbor in $X'$, it must be the case that $N_X(i) \cap (I \sqcup J \setminus \{j\}) = \emptyset$ because otherwise $T'$ would not be forked. It then follows that $X' = \{i\}$ is a nonleaf singleton and a forked subtree of $X$, which contradicts our assumption.

Hence $\deg_X'(i) = 1$. We know that $i-k$ is an edge of $X$ and not $X'$. Since $i \not\in F(X')$, $\deg_X(i) - \deg_{X'}(i) \leq 1$, and by the existence of $i-k$ this holds with equality, so that $\deg_{X'}(i) = 1$ and $\deg_{X}(i) = 2$.

Let $i_1 = i$ and $N_X[i] = \{i_2,k\}$. There exists a path $i_1 - i_2 - \cdots - i_n$ of nodes in $X'$ such that $\deg_X(i_\ell) = 2$ for $\ell \in [n-1]$. Eventually this path must end at some $i_n$ with $\deg_X(i_n) \geq 3$ or $i_n \in I \cup J$ since $T'$ was a forked subtree of $T$. We consider the inequality associated to $X'$:

\begin{equation}
\label{ieq:extraneousInequality1}
\sum\limits_{c \in F(X')} \ind_{c \to N_{X'}(c)}(x) + \sum\limits_{(c,d) \in L(X')} (1- x_{cc'd}) \leq 1 + \sum\limits_{c \in \interior(X')} \#^{N_{X'}[c]}(x).    
\end{equation}

We break into the subcases where $i_n \in I \sqcup J$ or $\deg_X(i_n) \geq 3$.

\textbf{Subcase 1:} First suppose that $i_n \in I \sqcup J$. By Definition~\ref{definition:forkedTreeIJ}~(3), we know $X'$ is the path $i_1 - i_2 - \cdots - i_n$. Each of these nodes has degree at most 2 in $X$, and so $F(X') = \emptyset$. Also by Definition~\ref{definition:forkedTreeIJ} we know $L(X') = \{(i_n, i_{n-1})\}$. Then Inequality~\ref{ieq:extraneousInequality1} reduces to
\[1- x_{i_ni_n'i_{n-1}} \leq 1 + x_{i_1i_2i_3} + \cdots + x_{i_{n-2}i_{n-1}i_n}.\]
By the nonnegativity of all coordinates in $\cim_T^{I,J}$, this inequality is a nontrivial sum of valid inequalities, which is extraneous.

\textbf{Subcase 2:} Now suppose that $\deg_X(i_n) \geq 3$ (which corresponds to Figure~\ref{fig:case7trees}). Let $X''$ be the subtree of $X'$ obtained by removing $i_1,\dots,i_{n-1}$. We claim $X''$ is a forked subtree of $X$. To see this, we note that the adjacencies of nodes in $X''$ are the same in $T'$ (and they also agree in $X$ and $T$) except at $i_n$, so we need only check the conditions in Definition~\ref{definition:forkedTreeIJ} at $i_n$. For (1), if $i_n$ is a leaf of $X''$, then $\deg_X(i_n) - \deg_{X''}(i_n) \geq 2$ since $\deg_X(i_n) \geq 3$. For (2) and (3), we recall that $X''$ is a subtree of $T'$ given by deleting one node of $J$ and a path to that node. In particular, the deleted nodes intersect the neighbors (in $T$) of $J$ only in the neighbors of $j$. Since $j$ was also deleted, (2) and (3) still hold for $X''$.

\begin{figure}
    \centering
\begin{tikzpicture}
\draw[red] (0,0)--(2.2,0);
\draw[red] (2.8,0)--(4,0);
\draw (4,0)--(5,0);
\draw (5,0)--(6,0.5);
\draw (5,0)--(6,-0.5);

\draw[red] (-1,0.5)--(0,0);
\draw[red] (-1,-0.5)--(0,0);
\draw (-0.6,-1.2)--(0,0);
\draw (0.6,-1.2)--(0,0);

\draw [red] (-1,0.5)--(-1.5,0.75);
\draw [red] (-1,0.5)--(-1.5,0.25);
\draw [red] (-1,-0.5)--(-1.5,-0.25);
\draw [red] (-1,-0.5)--(-1.5,-0.75);

\node at (2.5,0) {\color{red}$\cdots$};

\node[wR] at (0,0) {}; 
\node[wR] at (1,0) {}; 
\node[wR] at (2,0) {}; 
\node[wR] at (3,0) {}; 
\node[wR] at (4,0) {}; 
\node[wB] at (5,0) {}; 

\node[wR] at (-1,0.5) {};
\node[wR] at (-1,-0.5) {};
\node[wB] at (0.6,-1.2) {};
\node[wB] at (-0.6,-1.2) {};

\node at (0,0.3) {$i_n$};
\node at (1,0.3) {$i_{n-1}$};
\node at (2,0.3) {$i_{n-2}$};
\node at (3,0.3) {$i_2$};
\node at (4,0.3) {$i=i_1$};
\node at (5,0.3) {$k$};

\draw[blue, thick] (-1.7,-0.9)--(0.5,-0.9)--(0.5,0.9)--(-1.7,0.9)--cycle;
\node at (-2,0) {\color{blue}{$X''$}};
\node at (1.5,-0.3) {\color{red}{$X'$}};
\end{tikzpicture}
    \caption{The structure of all $X$ and $X'$ considered in case 7. Here $X''$ is the largest forked subtree of $X$ that is contained in $X'$.}
    \label{fig:case7trees}
\end{figure}
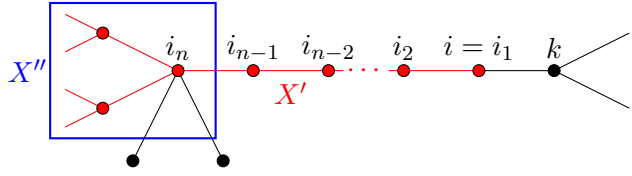

We claim that Inequality~\ref{ieq:extraneousInequality1} is the sum of the forked-tree inequality associated to $X''$ and other valid inequalities. Note that $L(X'') = L(X')$. Also since $i_1,\dots,i_{n-1}$ are degree 2 and connected in $X'$, $F(X')$ and $F(X'')$ differ only at (possibly) $i_n$. Lastly, $\interior(X')$ and $\interior(X'')$ differ at all $i_\ell$ for $\ell \in [n-1]$ and possibly at $i_n$. Then Inequality~\ref{ieq:extraneousInequality1} is equivalent to

\begin{equation}
\label{ieq:extraneousInequality2}
\begin{split}
&\sum\limits_{c \in F(X'')} \ind_{c \to N_{X''}(c)}(x) + \left( \ind_{i_n \to N_{X'}(i_n)}(x) - \ind_{i_n \to N_{X''}(i_n)}(x) \right) + \sum\limits_{(c,d) \in L(X'')} (1- x_{cc'd}) \\
&\leq 1 + \sum\limits_{c \in \interior(X'')} \#^{N_{X''}[c]}(x) + \left(\#^{N_{X'}[i_n]}(x) - \#^{N_{X''}[i_n]}(x)\right) + x_{i_1i_2i_3} + \cdots + x_{i_{n-2}i_{n-1}i_n}. 
\end{split}
\end{equation}
Note that these inequalities are equivalent since $\ind_{c \to N_{X''}(c)}(x) = 0$ whenever $c \not\in F(X'')$, and similarly for $\#^{N_{X''}[c]}(x)$. This is the sum of the valid forked-tree inequality associated to the forked tree $X''$ and the following inequalities:
\begin{align*}
&0 \leq x_{i_1i_2i_3} + \cdots + x_{i_{n-2}i_{n-1}i_n}, \\
&0 \leq \#^{N_{X'}[i_n]}(x) - \#^{N_{X''}[i_n]}(x), \\
&0 \leq \ind_{i_n \to N_{X''}(i_n)}(x) - \ind_{i_n \to N_{X'}(i_n)}(x).
\end{align*}
It remains to show that these inequalities are valid. Clearly the first inequality is valid since $\cim_X^{I,J \setminus \{j\}}$ is a 0/1 polytope. We will show the third inequality is valid since the proof of the second is similar.

For any interventional graph $\cg$, by Proposition~\ref{proposition:edgesForcedOut},
\[\ind_{i_n \to N_{X'}(i_n)}(c_{\cg}) = \begin{cases}
1 &\mbox{if}~|\pa_X(i_n)|\geq 2 ~\mbox{and}~ \pa_X(i_n) \cap N_{X'}(i_n) = \emptyset \\
0 &\mbox{otherwise.}
\end{cases}\]
 
That is, the functional evaluates to 1 if there exists a v-structure in $X$ at $i_n$ that forces the edges in $X'$ at $i_n$ to be essential and pointing away from $i_n$. The inequality in question can only be invalid on a vertex if $\ind_{i_n \to N_{X'}(i_n)}(x)$ evaluates to 1. In this case, the other functional evaluates to 1 as well, since $N_{X''}(i_n) \subset N_{X'}(i_n)$.
\end{proof}

To conclude this Section~\ref{section:eliminatingInterventions}, we briefly summarize the given argument. 
We showed that the Fourier-Motzkin elimination needed to project $\cim_T^{I,J}$ to $\cim_X^{I,J \setminus \{j\}}$ decomposed into 7 cases. 
The first three were easily seen and cases 4 and 5 were handled by Lemma~\ref{lemma:bidirectedEdgeIeqProjection} and Lemma~\ref{lemma:forkedTreeOppositeSides} respectively. 
Cases 6 and 7 were handled jointly by Lemma~\ref{lemma:FTIPlusBDIpart1}, Proposition~\ref{proposition:xPrimeForkedCase6}, and Lemma~\ref{lemma:FTIPlusBDIpart2}. 
Hence the proof of Theorem~\ref{theorem:facetsOfCIMTIJ} is complete.

\section{Applications}\label{sec:applications}

In this section, we apply the H-representations of CIM polytopes given in Theorem~\ref{thm:mainthm} to solve the linear program for which these polytopes were originally introduced; i.e., the problem of causal discovery (Problem~1.1). 
Given a real vector $\beta$ that depends only on a random sample and pre-specified modeling assumptions, we may treat the value $\beta\cdot c_\cg$ as a score assigned to the $\ci$-DAG $\cg^\ci$.  
For the sequence of intervention targets $\ci$ and a set of characteristic imsets $\mathcal{S}$, we may then compute the optimal $\ci$-MEC in $\mathcal{S}$ with respect to the score function $\beta\cdot x$ as the solution to the linear program 
\begin{equation}\label{eqn:LP}
    \begin{split}
\mbox{maximize }& \beta\cdot x\\
\mbox{subject to }& x\in \CIM_\mathcal{S}^\ci.
    \end{split}
\end{equation}

For instance, if we take $\mathcal{S} = \{c_\cg : \cg \mbox{ has skeleton $G$}\}$ and $\ci = (\emptyset)$ (i.e., $\CIM_\mathcal{S}^\ci = \CIM_G$), then the solution to the resulting linear program~\eqref{eqn:LP} will be the optimal DAG model $\cm(\cg)$ where $\cg$ has skeleton $G$.
If we replace $\ci = (\emptyset)$ with a sequence of intervention targets $\ci = (\emptyset, I_1,\ldots, I_K)$ where $I_k = \{i_k\}$ for all $k\in [K]$ and $i_1,\ldots, i_K$ is a collection of leaves in $G$, we solve the linear program with feasible region having H-representation specified by Theorem~\ref{thm:mainthm}. 
The result is then an optimal $\ci$-MEC of DAGs with skeleton $G$. 

In addition to the H-representation, implementing this causal discovery method requires a linear function $\beta\cdot x$ where $\beta$ is a vector depending only on the data and pre-specified modeling assumptions. 
We would further like that $\beta\cdot x$ is a statistically well-motivated score function for DAG models.  
A popular choice of such a function is the \emph{Bayesian Information Criterion} (BIC), which is 
\begin{equation*}
    \begin{split}
        \BIC(\cg; \dd) := \log L(\hat{\vec{\theta}} | \cg, \dd) - \frac{k\log(n)}{2},
    \end{split}
\end{equation*}
where $\hat{\vec{\theta}}$ is the MLE of the model parameters for the DAG model specified for $\cg$, $k$ is the number of free parameters in the model and $\dd$ is the random sample of size $n$. 

In this section, we establish the necessary theory for implementing this linear optimization approach to causal discovery under the assumption of Gaussian data. 
Our main contribution is a linear optimization framework for estimating causal polytree models from a mixture of observational and interventional Gaussian data. 
However, the method specializes to the case where one has access to only observational data. 

In Subsection~\ref{subsec:interventionalCIMlearn} we determine the vector $\beta$ expressing the BIC score for Gaussian $\ci$-DAG models.  
In Subsection~\ref{subsec:QIGTreeLearn}, we combine these results with Theorem~\ref{thm:mainthm} to yield a novel causal discovery algorithm for learning $\ci$-DAGs from a mixture of observational and interventional Gaussian data.

Throughout the remainder of this paper, we will use $S\in\rr^{p\times p}$ to denote a sample covariance matrix and $S_A$ to denote the submatrix of $S$ with rows and columns specified by $A\subseteq[p]$. 
The matrix $S_A^{\Phil}$ is the embedding of the matrix $S_A$ into a $p\times p$ matrix with all coordinates not given by $S_A$ being set equal to zero.
For a (positive definite) submatrix $S_A$ of a sample covariance matrix $S$, we let $P_A = S_A^{-1}$ denote the corresponding marginal sample precision matrix.

\subsection{BIC for characteristic imsets of Gaussian $\ci$-DAG models.}\label{subsec:interventionalCIMlearn}
In this subsection, we show that the BIC of a Gaussian $\ci$-DAG model $\cg^\ci$ can be computed as the inner product of the standard imset $u_{\cg^\ci}$ and a vector $\alpha$ depending only on the data and intervention targets $\ci$. 
Since $\CIM_{p, K}$ is isomorphic to $\SIM_{p,K}$ (see Subsection~\ref{subsec:interventionalCIMpolytopes}), we obtain the vector $\beta$ in~\eqref{eqn:LP} where
\begin{equation}\label{eqn:pullback}
\beta_{A} := \sum_{A\subseteq B}(-1)^{|B| - |A|}(1 - \alpha_{B}), \qquad A\subseteq[p]\cup \mathcal{Z}_\ci.
\end{equation}


Let $\ci = (I_0,\ldots, I_K)$ be a sequence of intervention targets with $I_0 = \emptyset$. 
Let $\dd = ([x_{i,j}^0], \ldots, [x_{i,j}^K])$ be a random sample from an interventional setting $(f^0,\ldots, f^K)\in\cm(\cg,\ci)$. 
In particular, $[x_{i,j}^k]$ is a random sample from $f^k\in\cm(\cg)$ for every $k$. 
For $A\subseteq \{0,\ldots, K\}$, we let $S^{\pool(A)}$ denote the sample covariance matrix of the data in $([x_{i,j}]^k : k \in A)$. 
For simplicity, we denote $S^{\pool(\{k\})}$ by $S^{(k)}$ for all $k$. 
For $i\in[p]$, we let $Z_i = \{ k\in\{0,\ldots, K\} : i\in I_k\}$ and $Z_i^c$ denote the complement of $Z_i$ in $\{0,\ldots, K\}$.
Finally, we let $\mathcal{Z}_i = \{z_k \in \mathcal{Z}_\ci : k\in Z_i\}$ and $\mathcal{Z}_i = \{z_k\in \mathcal{Z}_\ci : k\in Z_i^c\}$ denote the sets of vertices in $\mathcal{Z}_\ci$ indexed by $Z_i$ and $Z_i^c$, respectively. 

The following lemma generalizes a well-known result for DAG models $\cm(\cg)$ to the interventional setting. 
Since the proof of this lemma is mainly technical derivations, we defer it until Appendix~\ref{app:derivations}.

\begin{lemma}
    \label{lem:IMLE}
    Let $\ci = (I_0,I_1,\ldots, I_K)$ be a sequence of intervention targets with $I_0 = \emptyset$. 
Let $\dd = ([x_{i,j}^0], \ldots, [x_{i,j}^K])$ be a random sample from an interventional setting $(f^0,\ldots, f^K)\in\cm(\cg,\ci)$ where $\cm(\cg,\ci)$ is the Gaussian $\ci$-DAG model defined in Example~\ref{ex:interventionalgaussian}.
    The MLE of the covariance parameters $(\Sigma^{(0)},\ldots, \Sigma^{(K)})$ of the model is given by $(\hat\Sigma^{(0)},\ldots, \hat\Sigma^{(K)})$ where
    \begin{equation*}
    \begin{split}
    \hat K^{(k)} &= \sum_{i\notin I_k}\left(\left(P_{\fa_\cg(i)}^{\pool(Z_i^c)}\right)^{\Phil} - \left(P_{\pa_\cg(i)}^{\pool(Z_i^c)}\right)^{\Phil}\right) + \sum_{i\in I_k}\left(\left(P_{\fa_\cg(i)}^{(k)}\right)^{\Phil} - \left(P_{\pa_\cg(i)}^{(k)}\right)^{\Phil}\right)
    \end{split}
    \end{equation*}
    and $\hat K^{(k)} = \left(\hat\Sigma^{(k)}\right)^{-1}$. 
    Moreover, the MLE exists with probability one if and only if for all $k\in\{0,\ldots, K\}$, we have $n_k \geq |\fa_\cg(i)|$ for all $i\in I_k$ and $\sum_{k\in Z_i^c}n_k \geq |\fa_\cg(i)|$ for all $i\notin I_k$. 
\end{lemma}

We now define a vector $\alpha = (\alpha_{A\sqcup \mathcal{Z}} : A\subseteq[p], \mathcal{Z}\subseteq \mathcal{Z}_\ci)\subset \rr^{2^{p + K}}$. 
For $\mathcal{Z}\subseteq \mathcal{Z}_\ci$, let $Z = \{k \in [K]: z_k \in \mathcal{Z}\}$ denote the set of indices of nodes in $\mathcal{Z}$ and let $Z^c$ denote the complement of $Z$ in $\{0,1,\ldots, K\}$. 
For $A\subseteq[p]$ and $\mathcal{Z} \subseteq \mathcal{Z}_\ci$, let
\begin{equation}\label{eqn:ialphacoords}
    \begin{split}
        \alpha_{A\sqcup \mathcal{Z}} &= -\sum_{k\in Z^c}\left(\frac{1}{2}\sum_{j=1}^{n_k}(\vx_{:,j}^k)^T\left(\left(S_{A}^{\pool(Z^c)}\right)^{-1}\right)^{\Phil}\vx_{:,j}^k + \frac{n_k}{2}\log\det S_{A}^{\pool(Z^c)}\right)\\
        &\qquad -\sum_{k\in Z}\left(\frac{1}{2}\sum_{j=1}^{n_k}(\vx_{:,j}^k)^T\left(\left(S_{A}^{(k)}\right)^{-1}\right)^{\Phil}\vx_{:,j}^k  + \frac{n_k}{2}\log\det S_{A}^{(k)}\right)
        -\frac{1}{2}\binom{|A|}{2} - \frac{1}{2}\sum_{k\in Z}\binom{|A|}{2}.
    \end{split}
\end{equation}

We then have the following theorem, whose proof is in Appendix~\ref{app:derivations}.

\begin{theorem}
    \label{thm:IBIC}
    Let $\ci = (I_0,\ldots, I_K)$ be a sequence of intervention targets where $I_k\subseteq [p]$ for all $k\in[K]$ and $I_0 = \emptyset$. 
    Let $\dd = ([x_{i,j}^k] : k\in \{0,\ldots, K\})$ be a random sample from an interventional setting $(f^0,\ldots, f^K)$.   
    The BIC for the Gaussian $\ci$-DAG model $\cm(\cg, \ci)$ is
    $
        \BIC(\cg, \ci; \dd) = C + \sum_{i=1}^p(\alpha_{\pa_\cg(i)\sqcup \mathcal{Z}_i} - \alpha_{\fa_\cg(i)\sqcup \mathcal{Z}_i}),
    $
    where
    \[
    C = -\sum_{k=0}^K\frac{pn_k}{2}\log(2\pi) - \frac{1}{2}\left(p + \sum_{k=1}^K |I_k|\right).
    \]
\end{theorem}

Consider the CIM polytope
\[
\CIM_p^\ci = \conv(c_{\cg^\ci} : \cg \mbox{ a DAG on node set $[p]$})
\]
for a set of known intervention targets $\ci$, as defined in Section~\ref{sec:interventionalCIMpolytopes}.
Let $\alpha$ denote the vector with coordinates $\alpha_{A\sqcup \mathcal{Z}}$ as defined in~\eqref{eqn:ialphacoords}, and let $\beta$ denote the vector obtained by applying the transformation~\eqref{eqn:pullback} to $\alpha$. 
We obtain the following general result.

\begin{theorem}
    \label{thm:ICIM}
    Let $\ci = (I_0,\ldots, I_K)$ be a sequence of intervention targets where $I_k\subseteq [p]$ for all $k\in[K]$ and $I_0 = \emptyset$. 
    Let $\dd = ([x_{i,j}^k] : k\in \{0,\ldots, K\})$ be a random sample from an interventional setting $(f^0,\ldots, f^K)$ with collection of intervention targets $\ci$. 
    A BIC-optimal $\ci$-DAG $\cg^\ci$ over all Gaussian $\ci$-DAG models for the data $\dd$ is given by a solution to the linear program
    \begin{equation}\label{eqn:ILP}
    \begin{split}
        \mbox{maximize }& \beta\cdot x\\
        \mbox{subject to }& x\in \CIM_p^\ci.
    \end{split}
    \end{equation}
\end{theorem}

\begin{proof}
    Consider the standard imset associated to the $\ci$-DAG $\cg^\ci$:
    \[
    u_{\cg^\ci} = \delta_{[p]\sqcup \mathcal{Z}_\ci} - \delta_{\emptyset\sqcup \emptyset} + \sum_{i = 1}^p\left(\delta_{\pa_{\cg^\ci}(i)} - \delta_{\fa_{\cg^\ci}(i)}\right), 
    \]
    Note first that $\alpha_{[p]\sqcup \mathcal{Z}_\ci} - \alpha_{\emptyset \sqcup \emptyset}$ is a constant that depends only on the data and set of intervention targets $\ci$, both of which are fixed for a given instance of the problem. 
    Now let $A_{\fa(i)} = \fa_{\cg^\ci}(i) \cap [p]$, $\mathcal{Z}_{\fa(i)} = \fa_{\cg^\ci}(i) \cap \mathcal{Z}_\ci$, $A_{\pa(i)} = \pa_{\cg^\ci}(i)\cap [p]$ and $\mathcal{Z}_{\pa(i)} = \pa_{\cg^\ci}(i)\cap \mathcal{Z}_\ci$. 
    Note that $\mathcal{Z}_{\fa(i)} = \mathcal{Z}_{\pa(i)}$ by the definition of parent and family sets. 
    In particular, this is the set $\mathcal{Z}_i$.
    It follows from the definition of the standard imset $u_{\cg^\ci}$ and the vector $\alpha$ that 
    \begin{equation*}
        \begin{split}
            \alpha\cdot u_{\cg^\ci}
            &= \alpha_{[p]\sqcup \mathcal{Z}_\ci} - \alpha_{\emptyset\sqcup \emptyset} + \sum_{i = 1}^p\left(\alpha_{\pa_{\cg^\ci}(i)} - \alpha_{\fa_{\cg^\ci}(i)}\right),\\
            &= \alpha_{[p]\sqcup \mathcal{Z}_\ci} - \alpha_{\emptyset\sqcup \emptyset} + \BIC(\cg, \ci; \dd) - C, 
        \end{split}
    \end{equation*}
    where $C$ is the constant in Theorem~\ref{thm:IBIC} that depends only on the data and the pre-specified intervention targets. 
    Hence, $\alpha\cdot u_{\cg^\ci}$ is maximized over all $u_{\cg^\ci}$ at a BIC-optimal $\ci$-DAG. 
    It follows that a BIC-optimal $\ci$-DAG is a vertex solution to 
    \begin{equation}\label{eqn:simILP}
        \begin{split}
            \mbox{maximize }& \alpha \cdot x\\
            \mbox{subject to }& x\in \SIM_p^\ci.
        \end{split}
    \end{equation}
    Since $\CIM_p^\ci$ and $\SIM_p^\ci$ are isomorphic via the map in Definition~\ref{def:characteristicimset}, the $\ci$-MECs corresponding to solutions to~\eqref{eqn:ILP} and~\eqref{eqn:simILP} will be identical. 
\end{proof}

\subsection{QIGTreeLearn.}\label{subsec:QIGTreeLearn}
In this section, we present an algorithm following from 
Theorems~\ref{thm:mainthm} and~\ref{thm:ICIM} 
for learning an interventional Gaussian DAG model $(\cg, \ci)$ where $\cg$ has skeleton $G$ a tree and the nonempty targets in $\ci$ are singletons, each containing a leaf node of $G$. 
The learned DAG $\cg$ is known as a \emph{polytree} in the computer science literature.
We note that the problem of learning a polytree representation of a distribution from data is of interest, as seen in  \cite{jakobsen2022structure, linusson2022edges, rebane2013recovery}. 
The assumption of known intervention targets that are singletons is also well-accepted in applied fields such as biology, where in many instances it is possible to formulate experiments that target only specific variables of interest. 
These methods include, for instance, the use of CRISPR-Cas9 technology for specific single-gene knockout experiments \cite{dixit2016perturb}, and specific perturbations in protein-signaling networks induced by certain reagents \cite{sachs2005causal}. 
An example applying our method to one such data set from a biological study is presented in Subsection~\ref{subsec:realexps}.

Given Theorem~\ref{thm:mainthm}, we have the facets of the feasible region $\CIM_G^\ci$ of the linear program in Theorem~\ref{thm:ICIM} whenever $G$ is a tree and the nonempty targets in $\ci$ are singletons, each containing a leaf node of $G$. 
Hence, in the following we will assume that we have a random sample $\dd = ([x_{ij}^k] : k\in \{0,\ldots, K\})$ from an interventional setting $(f^0,\ldots, f^K)$ with set of known intervention targets $\ci = \{\emptyset, I_1,\ldots, I_K\}$ where $I_k$ is a singleton for all $k\in[K]$. 

To learn a model $(\cg,\ci)$ fulfilling the additional constraints of our theorems, it is necessary to first identify a skeleton for the model that is a tree. 
We use the same method for identifying an optimal skeleton that is a tree as used in \cite{jakobsen2022structure, linusson2022edges, rebane2013recovery}.
Namely, for each pair $i,j\in[p]$ we compute the negative mutual information $-I(X_i; X_j)$ and assign it as an edge weight to a complete graph on node set $[p]$. 
We then apply Kruskal's algorithm for identifying a minimum weight spanning tree $\hat G$ of this weighted complete graph. 
Once the skeleton $\hat G$ is estimated, we determine its leaf nodes and remove from consideration any interventional data sets that do not target a leaf.

The result is a skeleton $\hat G$, a sequence of interventional targets $\ci = (\emptyset, \{k_1\},\ldots, \{k_m\})$ with $k_1,\ldots, k_m$ leaves of $\hat G$ and a random sample $\dd = ([x_{i,j}^k] : k\in\{0,\ldots, K\})$.
We solve the linear program in Theorem~\ref{thm:ICIM} with feasible region restricted to $\CIM_{\hat G}^\ci$, returning a characteristic imset $c_{\cg^\ci}^\ast$ for an $\ci$-DAG $(\cg,\ci)$ where the skeleton of $\cg$ is equal to $\hat G$. 
A representative of the $\ci$-Markov equivalence class may be recovered from the coordinates of $c_{\cg^\ci}^\ast$. 
The method is summarized in Algorithm~\ref{alg:QIGTreeLearn}, and an implementation is available at \url{https://github.com/soluslab/causalCIM/tree/master/QIG}.
\begin{algorithm}
  \caption{QIGTreeLearn}
  \label{alg:QIGTreeLearn}
  \raggedright
  \hspace*{\algorithmicindent} \textbf{Input:} a sequence of intervention targets $\ci = (\emptyset, \{i_1\},\ldots, \{i_K\})$.\\
  \hspace*{\algorithmicindent} \textbf{Input:} a random sample $\mathbb{D} = ([x_{i,j}^k] : k\in\{0,\ldots, K\})$.\\
  \hspace*{\algorithmicindent} \textbf{Output:} An $\ci$-DAG $(\cg, \ci)$ where $\cg$ is a polytree.
  \begin{algorithmic}[1]
    \State $K_p \gets$ {a weighted complete graph on $[p]$ where $i - j$ is assigned the weight $-I(X_i;X_j)$.}
    \State $\hat G = ([p],E) \gets \mbox{a minimum weight spanning tree of $K_p$}$
    \State $L \gets \{j\in[p]: \deg_{\hat G}(j) = 1\}$
    \State $\ci \gets \{\emptyset\}\cup\{\{i_k\} : i_k\in L\}$
    \State $\beta \gets \mbox{CIM data vector for $\hat G$ and $\ci$}$
    \State $c_{\cg^\ci}^\ast \gets \text{argmax}(\beta\cdot x : x\in\CIM_{\hat G}^\ci)$
    \State \Return $c_{\cg^\ci}^\ast$
  \end{algorithmic}
\end{algorithm}

\begin{remark}
    \label{rmk: whynotalltrees?}
An tentative reader may wonder if it would be best to instead have a completely linear optimization-based method; e.g., solving the linear program with feasible region $\CIM_{\mathcal{S}}^\ci$ where $\mathcal{S}$ is the collection of characteristic imsets for any DAG whose skeleton is a tree, as opposed to first estimating the skeleton. 
However, according to computations in low dimensions, the polytope $\CIM_{\mathcal{S}}^\ci$ has more facets than vertices, meaning that any computation that uses all facets would touch more objects than a method that touches each MEC exactly once. 
In this way, QIGTreeLearn (Algorithm~\ref{alg:QIGTreeLearn}) presents a more efficient approach in regards to time complexity. 
\end{remark}

In our implementation of Algorithm~\ref{alg:QIGTreeLearn}, we solve the linear program using a standard interior-point method solver from \texttt{scikit-learn}. 
We note, however, that any linear optimization method using an H-representation of the feasible region can be substituted into this step.
Notice that setting $\ci = (\emptyset)$ solves the problem for observational data only. 

Finally, we remark that it is possible to learn models where intervention targets are not limited to leaf nodes, so long as we assume that such interventions occur at degree two nodes in the skeleton $T$ where there is no v-structure in the data-generating causal DAG. 
This is accomplished by replacing $\CIM_T^{\ci}$ with $\CIM_T^{I, J}$ for $J \neq \emptyset$.
Doing so uses the full power of Theorem~\ref{theorem:facetsOfCIMTIJ}, whereas the current method only uses the special case of this theorem summarized in Theorem~\ref{thm:mainthm}. 
However, we are not currently aware of any studies whose assumptions allow for the somewhat strange restrictions on the imsets whose convex hull forms $\CIM_T^{I, J}$ when $J$ is nonempty.

\subsection{Real data example.}\label{subsec:realexps}
We applied Algorithm~\ref{alg:QIGTreeLearn} to the protein mass spectrometry dataset of Sachs et al. \cite{sachs2005causal} consisting of 7644 measurements of the abundances of certain phospholipids and phosphoproteins in primary human immune system cells under different experimental conditions. 
The experimental conditions were generated by inhibiting or activating different proteins or receptor enzymes in the protein-signaling network through the use of various reagents. 
As described in \cite{sachs2005causal}, nine different reagents were used to produce the data, and the molecules perturbed by each reagent are known. 
In particular, several reagents induce \emph{specific perturbations} \cite{sachs2005causal}, targeting exactly one (known) molecule in the network. 
The data set is purely interventional (e.g., having no random sample from the unperturbed distribution). 
However, several of the reagents used are known to target only receptor enzymes and none of the signaling molecules being measured. 
In \cite{wang2017permutation}, the authors used this observation to extract an observational data set consisting of the $1755$ samples in which only receptor enzymes are targeted and the induced perturbations are identical. 
Following the set-up of \cite{wang2017permutation}, we consider the interventional data produced from five specific perturbations at signaling molecules. 
The intervention targets and the number of samples from each experiment are cataloged in the following table.

\begin{center}
    \smallskip
    
    \begin{tabular}{| c | c | c | c | c | c | c |}\hline
    intervention target & $\emptyset$ & Akt & PKC & PIP2 & Mek & PIP3 \\\hline
    sample size & 1755 & 911 & 723 & 810 & 799 & 848\\\hline
    \end{tabular}
    \smallskip
    
\end{center}

The skeleton $\hat G$ learned for the model by Algorithm~\ref{alg:QIGTreeLearn} is presented in Figure~\ref{fig:sachsskel}. 
We see from this figure that two of the five intervention targets are leaf nodes in the skeleton. 
Hence, Algorithm~\ref{alg:QIGTreeLearn} discards the data from the other three (nonempty) intervention targets and uses the remaining 3465 samples to estimate an $\ci$-DAG with skeleton $\hat G$ for the intervention targets $\ci = (I_0 = \emptyset, I_1 = \{\textrm{Akt}\}, I_2 = \{\textrm{Mek}\})$.
The resulting characteristic imset is
\begin{center}
    \smallskip

    \begin{tabular}{| l | l | l | l | l | l | l | l | l | }\hline
    $S$ & 014 \hspace{1pt} & 01$z_2$ & 234 & 034 \hspace{10pt} & 04(10)  & 34(10) \hspace{2pt} & 034(10) & 567  \\\hline
    $c_{\cg^{\ci}}(S)$ & 0 & 0 & 0 & 0 & 0 & 0 & 0 & 0 \\\hline

    \end{tabular}

    \hspace{-26pt}
    \begin{tabular}{| l | l | l | l | l | l | l | l | }\hline
    $S$ & 56$z_1$ & 578 \hspace{1pt} & 789 & 78(10) & 89(10) & 789(10) & 48(10) \hspace{2pt} \\\hline
    $c_{\cg^{\ci}}(S)$ & 1 & 0 & 1 & 0 & 0 & 0 & 1 \\\hline

    \end{tabular}
    \smallskip
\end{center}
where $z_1$ denotes the intervention at Akt, $z_2$ denotes the intervention at Mek and the protein molecules are indexed as
\begin{center}
    \smallskip

    \begin{tabular}{| c | c | c | c | c | c | c | c | c | c | c | c | c |}\hline
    index & 0 & 1 & 2 & 3 & 4 & 5 & 6 & 7 & 8 & 9 & 10\\\hline
    molecule & Raf & Mek & PLCg & PIP2 & PIP3 & Erk & Akt & PKA & PKC & p38 & Jnk\\\hline
    \end{tabular}
    \smallskip
\end{center} 
The corresponding $\ci$-DAG is presented in Figure~\ref{fig:sachstree}. 
We see that the optimal characteristic imset $c_{\cg^\ci}^\ast$ returned by Algorithm~\ref{alg:QIGTreeLearn} has coordinates $c_{\cg^\ci}(\{z_1, \textrm{Akt}, \textrm{Erk}\}) = 1$ and $c_{\cg^\ci}(\{z_2, \textrm{Mek}, \textrm{Raf}\}) = 0$.
Since the edge direction $z_1 \rightarrow \textrm{Akt}$ is essential in the $\ci$-DAG, we see that the first coordinate determines the direction of the edge $\textrm{Erk} \rightarrow \textrm{Akt}$ whose direction would otherwise be undetermined in the Markov equivalence class of $\cg$. 
Similarly, the second coordinate value determines the directions of the edge $\textrm{Mek}\rightarrow \textrm{Raf}$.  
Moreover, since the learned characteristic imset $c_{\cg^\ci}^\ast$ also satisfies 
\[
c_{\cg^\ci}(\{\textrm{Mek}, \textrm{Raf}, \textrm{PIP3}\}) = c_{\cg^\ci}(\{\textrm{Raf}, \textrm{PIP3}, \textrm{PIP2}\}) = c_{\cg^\ci}(\{\textrm{PIP3}, \textrm{PIP2}, \textrm{PCLg}\}) = 0, 
\]
the edges
$
\textrm{Raf} \rightarrow \textrm{PIP3}, \, \textrm{PIP3} \rightarrow \textrm{PIP2}, \, \mbox{and} \, \textrm{PIP2} \rightarrow \textrm{PCLg}
$
are also essential.

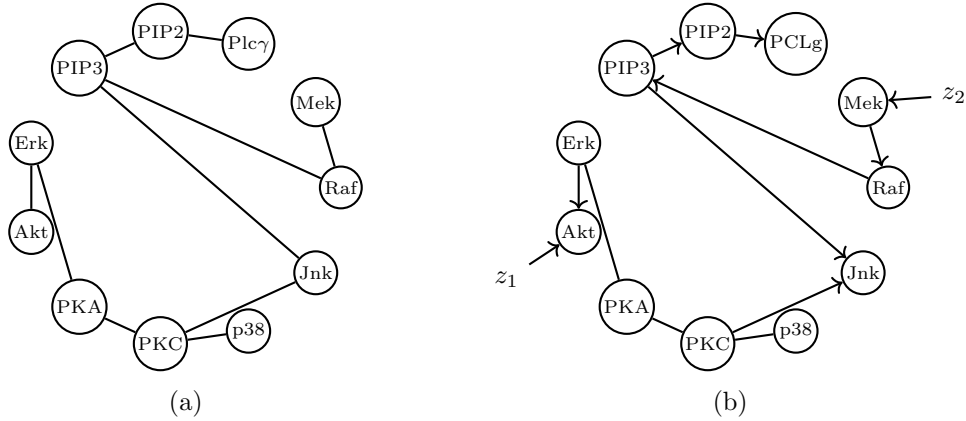
\begin{figure}
\begin{subfigure}[b]{0.4\textwidth}
    \centering
    \begin{tikzpicture}[thick, scale=0.175]

    \def \n {11}
    \def \radius {12cm}
    \def \margin {3} 

  \node[circle, draw, inner sep=1pt, minimum width=1pt] (PIP2) at ({360/\n * (3)}:\radius) {\tiny PIP2};
  \node[circle, draw, inner sep=1pt, minimum width=1pt] (Plcg) at ({360/\n * (2)}:\radius) {\tiny Plc$\gamma$};
  \node[circle, draw, inner sep=1pt, minimum width=1pt] (Mek) at ({360/\n * (1)}:\radius) {\tiny Mek};
  \node[circle, draw, inner sep=1pt, minimum width=1pt] (Raf) at ({360/\n * (11)}:\radius) {\tiny Raf};
  \node[circle, draw, inner sep=1pt, minimum width=1pt] (Jnk) at ({360/\n * (10)}:\radius) {\tiny Jnk};
  \node[circle, draw, inner sep=1pt, minimum width=1pt] (p38) at ({360/\n * (9)}:\radius) {\tiny p38};
  \node[circle, draw, inner sep=1pt, minimum width=1pt] (PKC) at ({360/\n * (8)}:\radius) {\tiny PKC};
  \node[circle, draw, inner sep=1pt, minimum width=1pt] (PKA) at ({360/\n * (7)}:\radius) {\tiny PKA};
  \node[circle, draw, inner sep=1pt, minimum width=1pt] (Akt) at ({360/\n * (6)}:\radius) {\tiny Akt};
  \node[circle, draw, inner sep=1pt, minimum width=1pt] (Erk) at ({360/\n * (5)}:\radius) {\tiny Erk};
  \node[circle, draw, inner sep=1pt, minimum width=1pt] (PIP3) at ({360/\n * (4)}:\radius) {\tiny PIP3};

  \draw[-] (Raf) -- (Mek) ;
  \draw[-] (Raf) -- (PIP3) ;
  \draw[-] (PIP3) -- (PIP2) ;
  \draw[-] (PIP2) -- (Plcg) ;
  \draw[-] (PIP3) -- (Jnk) ;
  \draw[-] (PKC) -- (Jnk) ;
  \draw[-] (PKC) -- (p38) ;
  \draw[-] (PKC) -- (PKA) ;
  \draw[-] (PKA) -- (Erk) ;
  \draw[-] (Erk) -- (Akt) ;

    \end{tikzpicture}
    \caption{}
    \label{fig:sachsskel}
    \end{subfigure}
    \hspace{0.15in}
     \begin{subfigure}[b]{0.4\textwidth}
         \centering
         \begin{tikzpicture}[thick, scale=0.175]

\def \n {11}
\def \radius {12cm}
\def \margin {3} 

  \node[circle, draw, inner sep=1pt, minimum width=1pt] (PIP2) at ({360/\n * (3)}:\radius) {\tiny PIP2};
  \node[circle, draw, inner sep=1pt, minimum width=1pt] (Plcg) at ({360/\n * (2)}:\radius) {\tiny PCLg};
  \node[circle, draw, inner sep=1pt, minimum width=1pt] (Mek) at ({360/\n * (1)}:\radius) {\tiny Mek};
  \node[circle, draw, inner sep=1pt, minimum width=1pt] (Raf) at ({360/\n * (11)}:\radius) {\tiny Raf};
  \node[circle, draw, inner sep=1pt, minimum width=1pt] (Jnk) at ({360/\n * (10)}:\radius) {\tiny Jnk};
  \node[circle, draw, inner sep=1pt, minimum width=1pt] (p38) at ({360/\n * (9)}:\radius) {\tiny p38};
  \node[circle, draw, inner sep=1pt, minimum width=1pt] (PKC) at ({360/\n * (8)}:\radius) {\tiny PKC};
  \node[circle, draw, inner sep=1pt, minimum width=1pt] (PKA) at ({360/\n * (7)}:\radius) {\tiny PKA};
  \node[circle, draw, inner sep=1pt, minimum width=1pt] (Akt) at ({360/\n * (6)}:\radius) {\tiny Akt};
  \node[circle, draw, inner sep=1pt, minimum width=1pt] (Erk) at ({360/\n * (5)}:\radius) {\tiny Erk};
  \node[circle, draw, inner sep=1pt, minimum width=1pt] (PIP3) at ({360/\n * (4)}:\radius) {\tiny PIP3};

  \node (iAkt) at (-17,-7) {\small $z_1$};
  \node (iMek) at (17,7) {\small $z_2$};

  \draw[<-] (Raf) -- (Mek) ;
  \draw[->] (Raf) -- (PIP3) ;
  \draw[->] (PIP3) -- (PIP2) ;
  \draw[->] (PIP2) -- (Plcg) ;
  \draw[->] (PIP3) -- (Jnk) ;
  \draw[->] (PKC) -- (Jnk) ;
  \draw[-] (PKC) -- (p38) ;
  \draw[-] (PKC) -- (PKA) ;
  \draw[-] (PKA) -- (Erk) ;
  \draw[->] (Erk) -- (Akt) ;

  \draw[->] (iAkt) -- (Akt) ;
  \draw[->] (iMek) -- (Mek) ;

\end{tikzpicture}
\caption{}
\label{fig:sachstree}
     \end{subfigure}
\caption{(a) The estimated skeleton for the Sachs et al.~\cite{sachs2005causal} data learned by Algorithm~\ref{alg:QIGTreeLearn}. (b) The learned $\ci$-essential graph of the $\ci$-Markov equivalence class of $\ci$-DAGs. The three undirected edges can be oriented in four ways without producing new v-structures. These four DAGs constitute the $\ci$-Markov equivalence class.}
\end{figure}

In summary, the interventional data allows us to deduce the causal direction of five arrows in the DAG whose directions would have been undetermined from the observational data alone. 
By Theorem~\ref{thm:Iverma}, this leaves only three edges in the $\ci$-Markov equivalence class undetermined, meaning we have learned the underlying causal model up to four possible graphs, as described in the caption of Figure~\ref{fig:sachstree}.
Finally, we note that additional data from a specific perturbation at p38 would allow us to either learn the causal graph entirely or reduce the $\ci$-equivalence class to only three possible graphs.

\section{Discussion and Future Work}\label{sec:futurework}
The results in this article provide hyperplane descriptions for the family of CIM polytopes given by the convex hull of all characteristic imsets for MECs of DAGs having a fixed skeleton that is a tree. 
The results are obtained by generalizing CIM polytopes to include models that accommodate experimental data. 
The main theorem (Theorem~\ref{thm:mainthm}) provides a hyperplane description of a family of interventional CIM polytopes.
The identified hyperplane descriptions and the newly introduced interventional CIM polytopes provide new methods for learning causal polytree structures from a combination of observational and interventional data, as shown in Section~\ref{sec:applications}. 
The corresponding data vector is also derived for general Gaussian $\ci$-DAG models (Theorem~\ref{thm:IBIC}), setting the stage for the expansion of linear programming methods to larger families of models. 

The derivation of Theorem~\ref{thm:mainthm} relies on a mixture classic techniques, e.g. Fourier-Motzkin elimination, and the observation that combinatorial interpretations of well-placed intervention targets realizes the CIM polytopes of interest as projections of a toric fiber product whose hyperplane description is easier to compute. 
This is the first time either of these techniques have been brought to bear on this particular optimization problem. The hyperplane descriptions obtained in Theorem~\ref{thm:mainthm} are easily phrased in terms of the combinatorics of directed graphs.
It would be reasonable to investigate if these same techniques admit broader applications to other families of (interventional) CIM polytopes. 
An interesting example to consider is the chordal graph polytope \cite{studeny2017towards} 
Another interesting case would be to generalize Theorem~\ref{theorem:facetsOfCIMTIJ} to include interventional targets at internal nodes of the skeleton with no restrictions on node degree or incident edge direction.

\section*{Acknowledgements}
We thank Steffen Lauritzen for helpful discussions.
Liam Solus was partially supported by the Wallenberg Autonomous Systems and Software Program (WASP) funded by the Knut and Alice Wallenberg Foundation, the Digital Futures Lab at KTH, the G\"oran Gustafsson Stiftelse Prize for Young Researchers, and Starting Grant No. 2019-05195 from The Swedish Research Council (Vetenskapsr\aa{}det).
Joseph Johnson was partially support by the latter grant. Benjamin Hollering was supported by the Alexander von Humboldt Foundation.

\bibliography{refs.bib}{}

\begin{thebibliography}{10}

\bibitem{andrews2023fast}
Bryan Andrews, Joseph Ramsey, Ruben Sanchez~Romero, Jazmin Camchong, and Erich
  Kummerfeld.
\newblock Fast scalable and accurate discovery of dags using the best order
  score search and grow shrink trees.
\newblock {\em Advances in Neural Information Processing Systems}, 36, 2023.

\bibitem{beinlich1989alarm}
Ingo~A Beinlich, Henri~Jacques Suermondt, R~Martin Chavez, and Gregory~F
  Cooper.
\newblock The alarm monitoring system: A case study with two probabilistic
  inference techniques for belief networks.
\newblock In {\em AIME 89: Second European Conference on Artificial
  Intelligence in Medicine, London, August 29th--31st 1989. Proceedings}, pages
  247--256. Springer, 1989.

\bibitem{chickering2002optimal}
David~Maxwell Chickering.
\newblock Optimal structure identification with greedy search.
\newblock {\em Journal of machine learning research}, 3(Nov):507--554, 2002.

\bibitem{dinu2020gorenstein}
Rodica Dinu and Martin Vodička.
\newblock Gorenstein property for phylogenetic trivalent trees.
\newblock {\em Journal of Algebra}, 575:233--255, 2021.

\bibitem{dixit2016perturb}
Atray Dixit, Oren Parnas, Biyu Li, Jenny Chen, Charles~P Fulco, Livnat
  Jerby-Arnon, Nemanja~D Marjanovic, Danielle Dionne, Tyler Burks, Raktima
  Raychowdhury, et~al.
\newblock Perturb-seq: dissecting molecular circuits with scalable single-cell
  rna profiling of pooled genetic screens.
\newblock {\em cell}, 167(7):1853--1866, 2016.

\bibitem{drton2018algebraic}
Mathias Drton.
\newblock Algebraic problems in structural equation modeling.
\newblock In {\em The 50th anniversary of Gr{\"o}bner bases}, volume~77, pages
  35--87. Mathematical Society of Japan, 2018.

\bibitem{engstrom2014multigraded}
Alexander Engstr\"{o}m, Thomas Kahle, and Seth Sullivant.
\newblock Multigraded commutative algebra of graph decompositions.
\newblock {\em J. Algebraic Combin.}, 39(2):335--372, 2014.

\bibitem{friedman2000using}
Nir Friedman, Michal Linial, Iftach Nachman, and Dana Pe'er.
\newblock Using bayesian networks to analyze expression data.
\newblock In {\em Proceedings of the fourth annual international conference on
  Computational molecular biology}, pages 127--135, 2000.

\bibitem{hauser2012characterization}
Alain Hauser and Peter B{\"u}hlmann.
\newblock Characterization and greedy learning of interventional markov
  equivalence classes of directed acyclic graphs.
\newblock {\em The Journal of Machine Learning Research}, 13(1):2409--2464,
  2012.

\bibitem{hemmecke2012characteristic}
Raymond Hemmecke, Silvia Lindner, and Milan Studen{\`y}.
\newblock Characteristic imsets for learning bayesian network structure.
\newblock {\em International Journal of Approximate Reasoning},
  53(9):1336--1349, 2012.

\bibitem{hollering2022toric}
Benjamin Hollering, Joseph Johnson, Irem Portakal, and Liam Solus.
\newblock Toric ideals of characteristic imsets via quasi-independence gluing,
  2022.

\bibitem{hu2022towards}
Zhongyi Hu and Robin Evans.
\newblock Towards standard imsets for maximal ancestral graphs.
\newblock {\em arXiv preprint arXiv:2208.10436}, 2022.

\bibitem{jakobsen2022structure}
Martin~E Jakobsen, Rajen~D Shah, Peter B{\"u}hlmann, and Jonas Peters.
\newblock Structure learning for directed trees.
\newblock {\em Journal of Machine Learning Research}, 23(159):1--97, 2022.

\bibitem{johnson2023codegree}
Joseph Johnson and Seth Sullivant.
\newblock The codegree, weak maximum likelihood threshold, and the gorenstein
  property of hierarchical models, 2023.

\bibitem{koller2009probabilistic}
Daphne Koller and Nir Friedman.
\newblock {\em Probabilistic graphical models: principles and techniques}.
\newblock MIT press, 2009.

\bibitem{lam2022greedy}
Wai-Yin Lam, Bryan Andrews, and Joseph Ramsey.
\newblock Greedy relaxations of the sparsest permutation algorithm.
\newblock In {\em Uncertainty in Artificial Intelligence}, pages 1052--1062.
  PMLR, 2022.

\bibitem{lindner2012discrete}
Silvia Lindner.
\newblock {\em Discrete optimisation in machine learning-learning of Bayesian
  network structures and conditional independence implication}.
\newblock PhD thesis, Technische Universit{\"a}t M{\"u}nchen, 2012.

\bibitem{LRS2022a}
Svante Linusson, Petter Restadh, and Liam Solus.
\newblock Greedy causal discovery is geometric.
\newblock {\em To appear in SIAM Journal on Discrete Mathematics}, 2022.

\bibitem{LRS2022b}
Svante Linusson, Petter Restadh, and Liam Solus.
\newblock On the edges of characteristic imset polytopes.
\newblock {\em arXiv preprint arXiv:2209.07579}, 2022.

\bibitem{linusson2022edges}
Svante Linusson, Petter Restadh, and Liam Solus.
\newblock On the edges of characteristic imset polytopes.
\newblock {\em arXiv preprint arXiv:2209.07579}, 2022.

\bibitem{maathuis2018handbook}
Marloes Maathuis, Mathias Drton, Steffen Lauritzen, and Martin Wainwright.
\newblock {\em Handbook of graphical models}.
\newblock CRC Press, 2018.

\bibitem{pearl2009causality}
Judea Pearl.
\newblock {\em Causality}.
\newblock Cambridge university press, 2009.

\bibitem{peters2017elements}
Jonas Peters, Dominik Janzing, and Bernhard Sch{\"o}lkopf.
\newblock {\em Elements of causal inference: foundations and learning
  algorithms}.
\newblock The MIT Press, 2017.

\bibitem{rauh2016lifting}
Johannes Rauh and Seth Sullivant.
\newblock Lifting {M}arkov bases and higher codimension toric fiber products.
\newblock {\em J. Symbolic Comput.}, 74:276--307, 2016.

\bibitem{rebane2013recovery}
George Rebane and Judea Pearl.
\newblock The recovery of causal poly-trees from statistical data.
\newblock {\em Proc. of Workshop on Uncertainty in Artificial Intelligence},
  pages 222--228, 1987.

\bibitem{robins2000marginal}
James~M Robins, Miguel~Angel Hernan, and Babette Brumback.
\newblock Marginal structural models and causal inference in epidemiology,
  2000.

\bibitem{sachs2005causal}
Karen Sachs, Omar Perez, Dana Pe'er, Douglas~A Lauffenburger, and Garry~P
  Nolan.
\newblock Causal protein-signaling networks derived from multiparameter
  single-cell data.
\newblock {\em Science}, 308(5721):523--529, 2005.

\bibitem{solus2017consistency}
Liam Solus, Yuhao Wang, and Caroline Uhler.
\newblock Consistency guarantees for greedy permutation-based causal inference
  algorithms.
\newblock {\em Biometrika}, 108(4):795--814, 2021.

\bibitem{spirtes1991algorithm}
Peter Spirtes and Clark Glymour.
\newblock An algorithm for fast recovery of sparse causal graphs.
\newblock {\em Social science computer review}, 9(1):62--72, 1991.

\bibitem{studeny2006probabilistic}
Milan Studeny.
\newblock {\em Probabilistic conditional independence structures}.
\newblock Springer Science \& Business Media, 2006.

\bibitem{studeny2017towards}
Milan Studen{\`y} and James Cussens.
\newblock Towards using the chordal graph polytope in learning decomposable
  models.
\newblock {\em International Journal of Approximate Reasoning}, 88:259--281,
  2017.

\bibitem{studeny2021dual}
Milan Studen{\`y}, James Cussens, and V{\'a}clav Kratochv{\'\i}l.
\newblock The dual polyhedron to the chordal graph polytope and the rebuttal of
  the chordal graph conjecture.
\newblock {\em International Journal of Approximate Reasoning}, 138:188--203,
  2021.

\bibitem{sullivant2007toric}
Seth Sullivant.
\newblock Toric fiber products.
\newblock {\em J. Algebra}, 316(2):560--577, 2007.

\bibitem{sullivant2018algebraic}
Seth Sullivant.
\newblock {\em Algebraic statistics}, volume 194.
\newblock American Mathematical Soc., 2018.

\bibitem{tsamardinos2006max}
Ioannis Tsamardinos, Laura~E Brown, and Constantin~F Aliferis.
\newblock The max-min hill-climbing bayesian network structure learning
  algorithm.
\newblock {\em Machine learning}, 65(1):31--78, 2006.

\bibitem{verma1990equivalence}
Thomas Verma and Judea Pearl.
\newblock Equivalence and synthesis of causal models.
\newblock In {\em Proceedings of the Sixth Annual Conference on Uncertainty in
  Artificial Intelligence}, pages 255--270, 1990.

\bibitem{wang2017permutation}
Yuhao Wang, Liam Solus, Karren Yang, and Caroline Uhler.
\newblock Permutation-based causal inference algorithms with interventions.
\newblock {\em Advances in Neural Information Processing Systems}, 30, 2017.

\bibitem{xi2015characteristic}
Jing Xi and Ruriko Yoshida.
\newblock The characteristic imset polytope of bayesian networks with ordered
  nodes.
\newblock {\em SIAM Journal on Discrete Mathematics}, 29(2):697--715, 2015.

\bibitem{yang2018characterizing}
Karren Yang, Abigail Katcoff, and Caroline Uhler.
\newblock Characterizing and learning equivalence classes of causal dags under
  interventions.
\newblock In {\em International Conference on Machine Learning}, pages
  5541--5550. PMLR, 2018.

\bibitem{ziegler2012lectures}
G{\"u}nter~M Ziegler.
\newblock {\em Lectures on polytopes}, volume 152.
\newblock Springer Science \& Business Media, 2012.

\end{thebibliography}
\bibliographystyle{plain}

\appendix

\section{Derivations for results in Subsection~\ref{subsec:interventionalCIMlearn}}
\label{app:derivations}

\begin{proof}[Proof of Lemma~\ref{lem:IMLE}]
    Let $\vec{\theta} = ((\vec{\lambda}^k,\vec{\omega}^k) : k \in\{0,\ldots, K\})$ denote the vector of model parameters defining the interventional setting. 
    The likelihood function for $\vec{\theta}$ given the data $\dd$ is then
    \begin{equation*}
        \begin{split}
            L(\vec{\theta} | \cg, \dd)
            &= \prod_{k=0}^K\prod_{j=1}^{n_k} f^k(\vx_{:,j}^k | \vec{\theta}).
        \end{split}
    \end{equation*}
    Since $f^k\in\cm(\cg)$ and $\omega_i^k = \omega_i^0$ and $\Lambda_{\pa_\cg(i),i}^k = \Lambda_{\pa_\cg(i),i}^0$ whenever $i\notin I_k$, we have that
    \begin{equation*}
        \begin{split}
            f^k(\vx_{:,j}^k | \vec{\theta}) 
            &= \prod_{i=1}^p f^k(x_{i,j}^k | \vx_{\pa_\cg(i),j}^k, \Lambda_{\pa_\cg(i),i}^k, \omega_i^k),\\
            &= \prod_{i\notin I_k}f^0(x_{i,j}^k | \vx_{\pa_\cg(i),j}^k, \Lambda_{\pa_\cg(i),i}^0, \omega_i^0)\prod_{i\in I_k}f^k(x_{i,j}^k | \vx_{\pa_\cg(i),j}^k, \Lambda_{\pa_\cg(i),i}^k, \omega_i^k).
        \end{split}
    \end{equation*}
    It follows that
    \begin{equation}\label{eqn:ilikelihood}
        \begin{split}
            L(\vec{\theta} | \cg, \dd)
            &= \left(\prod_{i=1}^p\prod_{k : i\notin I_k}\prod_{j=1}^{n_k} f^0(x_{i,j}^k | \vx_{\pa_\cg(i),j}^k, \Lambda_{\pa_\cg(i),i}^0,\omega_i^0) \right)\\
            &\qquad \times \left(\prod_{k=1}^K \prod_{i\in I_k}\prod_{j=1}^{n_k}f^k(x_{i,j}^k | \vx_{\pa_\cg(i),j}^k, \Lambda_{\pa_\cg(i),i}^k, \omega_i^k)\right).
        \end{split}
    \end{equation}
    For $i\in[p]$, the $i$-th product in the left product is the likelihood function $L(\Lambda_{\pa_\cg(i),i}^0, \omega_i^0 | ([x_{i,j}^k] : k\in Z_i^c)$. Similarly, the $k$-th product in the right product is the likelihood function $L(\Lambda_{\pa_\cg(i),i}^k,\omega_i^k | [x_{i,j}^k])$. Since the parameter space is a product space and the likelihood function factors according to this product space, the MLE is obtained by maximizing each of these factors separately.  
    
    Let $\vY = [Y_0, Y_1,\ldots, Y_m]^T$ denote a collection of Gaussian random variables. 
    Suppose that we have a random sample $y_{0,1},\ldots, y_{0,n}$ and for each $y_{0,i}$ we have the observations $y_{[m], i}$, for which we assume a simple linear regression model
    \[
    Y_{0,i} \sim \textrm{N}(\Lambda^Ty_{[m],i}, \sigma^2), 
    \]
    with regression coefficients $\Lambda = [\lambda_1,\ldots, \lambda_m]^T$.  
    Results from basic regression tell us that 
    \[
    \hat\Lambda = S_{0,[m]}S_{[m],[m]}^{-1} \qquad \mbox{and} \qquad \hat \sigma^2 = S_{0,0} - S_{0,[m]}S_{[m],[m]}^{-1}S_{[m],0},
    \]
    where $S = (1/n)\sum_{i=1}^n\vy_i\vy_i^T$ for the paired samples $\vy = [y_{0,i},y_{1,i},\ldots,y_{m,i}]^T$. 
    Since $(\hat\Lambda, \hat\sigma^2)$ is the MLE of the parameters $(\Lambda, \sigma^2)$ given the data, we have that the likelihood function
    \[
    L(\Lambda, \sigma^2 | \vy) = \prod_{i=1}^nf(y_{0,i} | \vy_{[m],i}, \Lambda, \sigma^2) 
    \]
    is maximized at $(\hat\Lambda, \hat\sigma^2)$. 
    Since $(\hat\Lambda, \hat\sigma^2)$ are also the parameters for the conditional distribution $Y_0 | \vY_{[m]} = \vy_{[m]}$, the conditional density function for this distribution with parameters set equal to the MLE is 
    \begin{equation}
        \label{eqn: max at marg sampcov}
        f_{Y_0 | Y_{[m]}}(y_0 | \vy_{[m]}) = \frac{f_{\vY}(\vy | S)}{f_{\vY_{[m]}}(\vy_{[m]} | S_{[m],[m]})}, 
    \end{equation}
    where $f_{\vY}(\vy | S)$ and $f_{\vY_{[m]}}(\vy_{[m]} | S_{[m],[m]})$ are, respectively the densities for $\vY$ and $\vY_{[m]}$ parametrized in the covariance parameters. 
    

    Applying this to each factor in our factorized likelihood above, we obtain that each factor used in the definition of the interventional setting $(f^0,\ldots, f^K)$ maximizes for the parameter choices
    \[
    f^0(x_i | x_{\pa_\cg(i)}, \Lambda_{\pa_\cg(i)}^0, \omega_i^0) = \frac{f^0\left(\vx_{\fa_\cg(i)} | S^{\pool(Z_i^c)}_{\fa_\cg(i)}\right)}{f^0\left(\vx_{\pa_\cg(i)} | S^{\pool(Z_i^c)}_{\pa_\cg(i)}\right)},
    \]
    for all $i\in[p]$, and 
    \[
    f^k(x_i | x_{\pa_\cg(i)}, \Lambda_{\pa_\cg(i)}^k, \omega_i^k) = \frac{f^k\left(\vx_{\fa_\cg(i)} | S^{(k)}_{\fa_\cg(i)}\right)}{f^k\left(\vx_{\pa_\cg(i)} | S^{(k)}_{\pa_\cg(i)}\right)},
    \]
    for all $k\in[K]$ and each $i\in I_k$.

    Finally, to see the claim on the existence of the MLE, notice that \eqref{eqn: max at marg sampcov} implies, for each $k\in\{0,\ldots, K\}$, that the MLE $\hat\Sigma^{(k)}$ exists if and only if the MLEs for each of the relevant factors above exist.  
    This happens if and only if the corresponding matrices $S_{\fa_\cg(i)}^{\pool(Z_i^c)}$ and $S_{\fa_\cg(i)}^{(k)}$ are positive definite. 
    The latter condition happens almost surely if and only if $n_k \geq |\fa_\cg(i)|$ for all $i\in I_k$ and $\sum_{k\in Z_i^c}n_k \geq |\fa_\cg(i)|$ for all $i\notin I_k$. 
\end{proof}

To prove Theorem~\ref{thm:IBIC}, we will use the following lemma, whose proof is well-known.

\begin{lemma}
    \label{lem:det}
    Let $\Sigma$ be the covariance matrix of a Gaussian distribution belonging to the DAG model $\cm(\cg)$ for some DAG $\cg$. 
    Then
    \[
    \det\Sigma = \prod_{i=1}^p \frac{\det \Sigma_{\fa_\cg(i)}}{\det \Sigma_{\pa_\cg(i)}}.
    \]
\end{lemma}

\begin{proof}[Proof of Theorem~\ref{thm:IBIC}]
Since the likelihood function for the data $\dd = ([x_{i,j}^k] : k\in \{0,\ldots, K\})$ for a Gaussian $\ci$-DAG model with parameter vector $\vec{\theta}$ is 
\[
L(\vec{\theta} | \cg, \ci, \dd) = \prod_{k=0}^K\prod_{j=1}^{n_k}f^k(\vx_{:,j}^k | \vec{\theta}),
\]
where $f^k$ is the density function for a multivariate normal distribution with parameters $(\vec{\lambda}^k, \vec{\omega}^k)$, it follows that
\begin{equation*}
    \begin{split}
        L(\hat{\vec{\theta}} | \cg, \ci, \dd)
        &= \prod_{k=0}^K\left(\frac{1}{(2\pi)^{pn_k/2}(\det\hat\Sigma^{(k)})^{n_k/2}}\right)\prod_{j=1}^{n_k}\exp\left(-\frac{1}{2}(\vx_{:,j}^k)^T\hat K^{(k)}\vx_{:,j}^k\right),
    \end{split}
\end{equation*}
where $\hat{\vec{\theta}}$ denotes the MLE for the model with DAG $\cg$ and intervention targets $\ci$ according to Theorem~\ref{lem:IMLE}.
As seen in the model definition in Example~\ref{ex:interventionalgaussian}, the number of free parameters in this interventional model is
\[
p + |E| + \sum_{k=1}^K\left(|I_k| + \sum_{i\in I_k}|\pa_\cg(i)|\right).
\]
It follows that
\begin{equation*}
    \begin{split}
        \BIC(\cg, \ci: \dd) 
        &= \sum_{k=0}^K\left(-\frac{pn_k}{2}\log(2\pi) - \frac{n_k}{2}\log\det\hat\Sigma^{(k)} - \frac{1}{2}\sum_{j=1}^{n_k}(\vx_{:,j}^k)^T\hat K^{(k)}\vx_{:,j}^k\right)\\
        & \qquad -\frac{\log(n)(p + |E|)}{2} - \frac{\log(n)}{2}\sum_{k=1}^K\left(|I_k| + \sum_{i\in I_k}|\pa_\cg(i)|\right),
    \end{split}
\end{equation*}
where we let $n$ denote the number of samples in the data array $\dd$; e.g., $n = n_0 + \cdots + n_K$. 
Distributing the first sum and applying Lemma~\ref{lem:det}, we obtain
\begin{equation*}
    \begin{split}
        \BIC(\cg, \ci: \dd) 
        &= \sum_{k=0}^K-\frac{pn_k}{2}\log(2\pi) - \sum_{k=0}^K\frac{n_k}{2}\log\prod_{i=1}^p\frac{\det\hat\Sigma_{\fa_\cg(i)}^{(k)}}{\det\hat\Sigma_{\pa_\cg(i)}^{(k)}} - \sum_{k=0}^K\frac{1}{2}\sum_{j=1}^{n_k}(\vx_{:,j}^k)^T\hat K^{(k)}\vx_{:,j}^k\\
        & \qquad -\frac{\log(n)(p + |E|)}{2} - \frac{\log(n)}{2}\sum_{k=1}^K\left(|I_k| + \sum_{i\in I_k}|\pa_\cg(i)|\right),\\
        &= \sum_{k=0}^K-\frac{pn_k}{2}\log(2\pi) - \sum_{k=0}^K\sum_{i=1}^p\left(\frac{n_k}{2}\log\det\hat\Sigma_{\fa_\cg(i)}^{(k)} - \frac{n_k}{2}\log\det\hat\Sigma_{\pa_\cg(i)}^{(k)}\right) \\
        & \qquad - \sum_{k=0}^K\frac{1}{2}\sum_{j=1}^{n_k}(\vx_{:,j}^k)^T\hat K^{(k)}\vx_{:,j}^k\\
        & \qquad -\frac{\log(n)(p + |E|)}{2} - \frac{\log(n)}{2}\sum_{k=1}^K\left(|I_k| + \sum_{i\in I_k}|\pa_\cg(i)|\right),\\
        &= \sum_{k=0}^K-\frac{pn_k}{2}\log(2\pi) - \sum_{i=1}^p\sum_{k=0}^K\left(\frac{n_k}{2}\log\det\hat\Sigma_{\fa_\cg(i)}^{(k)} - \frac{n_k}{2}\log\det\hat\Sigma_{\pa_\cg(i)}^{(k)}\right) \\
        & \qquad - \sum_{k=0}^K\frac{1}{2}\sum_{j=1}^{n_k}(\vx_{:,j}^k)^T\hat K^{(k)}\vx_{:,j}^k\\
        & \qquad -\frac{\log(n)(p + |E|)}{2} - \frac{\log(n)}{2}\sum_{k=1}^K\left(|I_k| + \sum_{i\in I_k}|\pa_\cg(i)|\right).\\
    \end{split}
\end{equation*}
The final line above is the negative of the penalization term, which in this case is equal to 
\begin{equation*}
    \begin{split}
        \textrm{pen} &= \frac{\log(n)}{2}\sum_{i = 1}^p\left[\left[\binom{|\fa_\cg(i)|}{2} + \sum_{k\in Z_i}\binom{|\fa_\cg(i)|}{2}\right] - \left[\binom{|\pa_\cg(i)|}{2} + \sum_{k\in Z_i}\binom{|\pa_\cg(i)|}{2}\right]\right]\\
        & \qquad + \frac{\log(n)}{2}\left(p + \sum_{k=1}^K|I_k|\right).
    \end{split}
\end{equation*}
Applying Lemma~\ref{lem:IMLE} and recalling that $Z_i = \{k \in [K] : i\in I_k\}$ and $Z_i^c = \{k\in \{0,1,\ldots, K\} : k\notin Z_i\}$, we obtain

\begin{equation*}
    \begin{split}
        \BIC(\cg, \ci; \dd)
        &= -\sum_{k=0}^K\frac{pn_k}{2}\log(2\pi) + \sum_{i = 1}^p \left[-\sum_{k=0}^K\left(\frac{n_k}{2}\log\det\hat\Sigma_{\fa_\cg(i)}^{(k)} - \frac{n_k}{2}\log\det\hat\Sigma_{\pa_\cg(i)}^{(k)}\right)\right.\\
        & \qquad -\sum_{k\in Z_i^c}\left(\frac{1}{2}\sum_{j=1}^{n_k}(\vx_{:,j}^k)^T\left(P_{\fa(i)}^{\pool(Z_i^c)}\right)^{\Phil}\vx_{:,j}^k - \frac{1}{2}\sum_{j=1}^{n_k}(\vx_{:,j}^k)^T\left(P_{\pa(i)}^{\pool(Z_i^c)}\right)^{\Phil}\vx_{:,j}^k \right)\\
        & \qquad \left. -\sum_{k\in Z_i}\left(\frac{1}{2}\sum_{j=1}^{n_k}(\vx_{:,j}^k)^T\left(P_{\fa(i)}^{(k)}\right)^{\Phil}\vx_{:,j}^k \right. \left.- \frac{1}{2}\sum_{j=1}^{n_k}(\vx_{:,j}^k)^T\left(P_{\pa(i)}^{(k)}\right)^{\Phil}\vx_{:,j}^k \right)\right]\\
        & \qquad - \textrm{pen}.
    \end{split}
\end{equation*}

Notice next that by \eqref{eqn: max at marg sampcov}, the probability density function of a normal distribution with parameters given by the MLE $\hat \Sigma^{(k)}$ is equal to the product of functions
\[
\prod_{i\notin I_k}\frac{f^0\left(\vx_{\fa_\cg(i)} | S_{\fa_\cg(i)}^{\pool(Z_i^c)}\right)}{f^0\left(\vx_{\pa_\cg(i)} | S_{\pa_\cg(i)}^{\pool(Z_i^c)}\right)}\prod_{i\in I_k}\frac{f^k\left(\vx_{\fa_\cg(i)} | S_{\fa_\cg(i)}^{(k)}\right)}{f^k\left(\vx_{\pa_\cg(i)} | S_{\pa_\cg(i)}^{(k)}\right)}.
\]
It follows that
\begin{equation*}
    \begin{split}
        f(\vx | \hat\Sigma^{(k)}) &= \prod_{i\notin I_k}\frac{f^0\left(\vx_{\fa_\cg(i)} | S_{\fa_\cg(i)}^{\pool(Z_i^c)}\right)}{f^0\left(\vx_{\pa_\cg(i)} | S_{\pa_\cg(i)}^{\pool(Z_i^c)}\right)}
        \prod_{i\in I_k}\frac{f^k\left(\vx_{\fa_\cg(i)} | S_{\fa_\cg(i)}^{(k)}\right)}{f^k\left(\vx_{\pa_\cg(i)} | S_{\pa_\cg(i)}^{(k)}\right)},\\
        \frac{e^{-\frac{1}{2}\vx^T\hat K^{(k)} \vx}}{(2\pi)^{p/2}}\frac{1}{\left(\det\hat\Sigma^{(k)}\right)^{1/2}} &= \frac{e^{-\frac{1}{2}\vx^T\hat K^{(k)} \vx}}{(2\pi)^{p/2}}\prod_{i\notin I_k}\frac{\left(\det S_{\pa_\cg(i)}^{\pool(Z_i^c)}\right)^{1/2}}{\left(\det S_{\fa_\cg(i)}^{\pool(Z_i^c)}\right)^{1/2}}\prod_{i\in I_k}\frac{\left(\det S_{\pa_\cg(i)}^{(k)}\right)^{1/2}}{\left(\det S_{\fa_\cg(i)}^{(k)}\right)^{1/2}}.
    \end{split}
\end{equation*}
Therefore, 
\[
\det\hat\Sigma^{(k)} = \prod_{i\notin I_k}\frac{\det S_{\fa_\cg(i)}^{\pool(Z_i^c)}}{\det S_{\pa_\cg(i)}^{\pool(Z_i^c)}}\prod_{i\in I_k}\frac{\det S_{\fa_\cg(i)}^{(k)}}{\det S_{\pa_\cg(i)}^{(k)}}
\]
Applying this observation to the formula for $\BIC(\cg, \ci; \dd)$ above, we obtain
\begin{equation*}
    \begin{split}
        \BIC(\cg, \ci; \dd)
        &= -\sum_{k=0}^K\frac{pn_k}{2}\log(2\pi) + \sum_{i = 1}^p \left[-\sum_{k\in Z_i^c}\left(\frac{n_k}{2}\log\det S_{\fa_\cg(i)}^{\pool(Z_i^c)} - \frac{n_k}{2}\log\det S_{\pa_\cg(i)}^{\pool(Z_i^c)}\right)\right.\\
        & \qquad -\sum_{k\in Z_i}\left(\frac{n_k}{2}\log\det S_{\fa_\cg(i)}^{(k)} - \frac{n_k}{2}\log\det S_{\pa_\cg(i)}^{(k)}\right)\\
        & \qquad -\sum_{k\in Z_i^c}\left(\frac{1}{2}\sum_{j=1}^{n_k}(\vx_{:,j}^k)^T\left(P_{\fa(i)}^{\pool(Z_i^c)}\right)^{\Phil}\vx_{:,j}^k - \frac{1}{2}\sum_{j=1}^{n_k}(\vx_{:,j}^k)^T\left(P_{\pa(i)}^{\pool(Z_i^c)}\right)^{\Phil}\vx_{:,j}^k \right)\\
        & \qquad -\left. \sum_{k\in Z_i}\left(\frac{1}{2}\sum_{j=1}^{n_k}(\vx_{:,j}^k)^T\left(P_{\fa(i)}^{(k)}\right)^{\Phil}\vx_{:,j}^k \right. \left.- \frac{1}{2}\sum_{j=1}^{n_k}(\vx_{:,j}^k)^T\left(P_{\pa(i)}^{(k)}\right)^{\Phil}\vx_{:,j}^k \right)\right]\\
        & \qquad - \textrm{pen}.
    \end{split}
\end{equation*}

Grouping by $\fa_\cg(i)$ and $\pa_\cg(i)$, we obtain 
$
\BIC(\cg, \ci; \dd) = C + \sum_{i=1}^p(\alpha_{\pa_\cg(i)\sqcup \mathcal{Z}_i} - \alpha_{\fa_\cg(i)\sqcup \mathcal{Z}_i}),
$
where
\[
    C = -\sum_{k=0}^K\frac{pn_k}{2}\log(2\pi) - \frac{\log(n)}{2}\left(p + \sum_{k=1}^K |I_k|\right).
\]
    
\end{proof}

\end{document}